\documentclass[a4paper,11pt,leqno]{amsart}
\usepackage[big]{quot_style_v3}

\title{The Hilb-vs-Quot Conjecture}

\author{Oscar Kivinen}
\address{Aalto University, Department of Mathematics and Systems Analysis, P.O. Box 11100, FI-00076 Aalto, Finland}
\email{oscar.kivinen@aalto.fi}

\author[Minh-T\^{a}m Trinh]{Minh-T\^{a}m Quang Trinh}
\address{Department of Mathematics, Howard University, Washington, DC 20059}
\email{minhtam.trinh@howard.edu}

\begin{document}

\begin{abstract}
Let $R$ be the complete local ring of a complex plane curve germ and $S$ its normalization.
We propose a ``Hilb-vs-Quot'' conjecture relating the virtual weight polynomials of the Hilbert schemes of $R$ to those of the Quot schemes that parametrize $R$-submodules of $S$.
By relating the Quot side to a type of compactified Picard scheme, we show that our conjecture generalizes a conjecture of Cherednik's, and that it would relate the perverse filtration on the cohomology of the Picard side to a more elementary filtration.
Next, we propose a Quot version of the Oblomkov--Rasmussen--Shende conjecture, relating parabolic refinements of our Quot schemes to Khovanov--Rozansky link homology.
It becomes equivalent to the original version under (refined) Hilb-vs-Quot, but can also be strengthened to incorporate polynomial actions and $y$-ification.
For germs $y^n = x^d$, where $n$ is either coprime to or divides $d$, we prove the Quot version of ORS through combinatorics.
When $n = 3$ and $3 \nmid d$, we deduce Hilb-vs-Quot by an asymptotic argument, and hence, establish the original ORS conjecture for these germs.
\end{abstract}

\maketitle

\thispagestyle{empty}

\section{Introduction}

\subsection{}

Let $\RRR$ be the complete local ring of a complex algebraic plane curve germ: a reduced, complete, local $\bb{C}$-algebra of Krull dimension $1$, embedding dimension at most $2$, and residue field $\bb{C}$.
Let $\SSS \supseteq \RRR$ be its normalization.
For any finitely-generated $\RRR$-module $E$, let $\Quot^\ell(E)$ denote the Quot scheme whose $\bb{C}$-points parametrize submodules of $E$ of $\bb{C}$-codimension $\ell$.
It is a scheme of finite type.
When $E = \RRR$, it is the Hilbert scheme of $\ell$ points on $\Spec(\RRR)$.
We write
\begin{align}
	\cal{H}^\ell = \Quot^\ell(\RRR)
	\quad\text{and}\quad
	\cal{Q}^\ell = \Quot^\ell(\SSS).
\end{align}
For any $\bb{C}$-scheme of finite type $X$, let $\chi(X, \sf{t}) \in \bb{Z}[\sf{t}]$ denote the virtual weight polynomial of $X$ in the sense of mixed Hodge theory.
Let
\begin{align}
	\sf{Hilb}(\sf{q}, \sf{t}) 
	= \sum_{\ell \geq 0} \sf{q}^\ell \chi(\cal{H}^\ell, \sf{t})
	\quad\text{and}\quad
	\sf{Quot}(\sf{q}, \sf{t}) 
	&= \sum_{\ell \geq 0} \sf{q}^\ell \chi(\cal{Q}^\ell, \sf{t}).
\end{align}
We start by proposing the following conjecture.  

\begin{mainconj}[Hilb-vs-Quot]\label{conj:main}
	For any plane curve germ,
	\begin{align}
		\sf{Hilb}(\sf{q}, \sf{t}) = \sf{Quot}(\sf{q}, \sf{q}^{\frac{1}{2}} \sf{t}) .
	\end{align}
\end{mainconj}

Our first goal in this paper is to show that \Cref{conj:main} extends a conjecture of Cherednik's to plane curve germs with multiple branches.
Our second goal is to use a parabolic refinement of \Cref{conj:main} to clarify the Oblomkov--Rasmussen--Shende (ORS) Conjecture, which relates the Hilbert schemes of the germ to the Khovanov--Rozansky (KhR) homology of its link.
For germs of the form $y^n = x^d$, where $n$ is either coprime to or divides $d$, we will prove a Quot analogue of ORS.
For germs $y^3 = x^d$ with $d$ coprime to $3$, we will prove all of the conjectures above.

Separately, we will also describe a refinement of \Cref{conj:main} that incorporates known polynomial actions on the link homology and on its $y$-ification.

\subsection{}

We first review Cherednik's conjecture.
Let $K$ be the ring of fractions of $\SSS$.
The compactified Picard scheme of $\RRR$ is a reduced ind-scheme $\CptPic$ over $\bb{C}$ whose points parametrize finitely-generated $\RRR$-submodules $M \subseteq K$ such that $KM = K$.
Let $\gap$ be the \dfemph{gap function} on $\CptPic(\bb{C})$ given by
$\gap(M) = \sf{dim}_\bb{C}((\SSS M)/M).$
It takes values between $0 = \gap(\SSS)$ and the \dfemph{delta invariant} $\delta \vcentcolon= \gap(\RRR)$.

Let $\CptJac \subseteq \CptPic$ be the locus parametrizing those $M \subseteq K$ such that $M \cap \RRR$ has the same index in both $\RRR$ and $M$.
For any fixed integer $\gap$, let $\CptJac(\gap) \subseteq \CptJac$ be the constructible subvariety parametrizing $M$ such that $\gap(M) = \gap$.
In \cite[Conj.\ 4.5]{cherednik}, Cherednik essentially conjectured that when $\RRR$ is unibranch,
\begin{align}\label{eq:cherednik-weight}
	\sf{Hilb}(\sf{q}, \sf{t}) \stackrel{?}{=}
	\frac{1}{1 - \sf{q}}
	\sum_{0 \leq \gap \leq \delta}
	\sf{q}^\gap \chi(\CptJac(\gap), \sf{q}^{\frac{1}{2}}\sf{t}).
\end{align}
We will show that \Cref{conj:main} generalizes \eqref{eq:cherednik-weight} beyond the unibranch case.
For this, fix a uniformization $\SSS \simeq \prod_{i = 1}^b \bb{C}[\![\varpi_i]\!]$.
The scaling action of $K^\times$ on $\CptPic$ restricts to a free action of the lattice $\Lattice$ of elements $\vec{\varpi}^{\vec{x}} \vcentcolon= \varpi_1^{x_1} \cdots \varpi_b^{x_b}$ for $x_1, \ldots, x_b \in \bb{Z}$.
The quotient $\CptPic/\Lattice$ is a projective variety.
In the unibranch case where $b = 1$ and $\Lattice \simeq \bb{Z}$, we have $\CptPic/\Lattice \simeq \CptJac$.

The subvariety $\CptJac(\gap) \subseteq \CptJac$ is analogous to a $\Gamma$-stable subvariety $\CptPic(\gap) \subseteq \CptPic$.
When $b = 1$, the identity
\begin{align}\label{eq:hilb-vs-p}
	\sf{Hilb}(\sf{q}, \sf{t}) \stackrel{?}{=} 
	\frac{1}{(1 - \sf{q})^b} 
	\sum_\gap
	\sf{q}^\gap \chi(\CptPic(\gap)/\Lattice, \sf{q}^{\frac{1}{2}}\sf{t})
\end{align}
specializes to \eqref{eq:cherednik-weight}.
We will prove that
\begin{align}\label{eq:quot-vs-p}
	\sf{Quot}(\sf{q}, \sf{t}) \stackrel{?}{=} 
	\frac{1}{(1 - \sf{q})^b} 
	\sum_\gap
	\sf{q}^\gap \chi(\CptPic(\gap)/\Lattice, \sf{t}),
\end{align}
thereby proving that \Cref{conj:main} is equivalent to \eqref{eq:hilb-vs-p}.

\subsection{}

In fact, we will propose a conjecture stronger than  \Cref{conj:main}, and prove a statement stronger than \eqref{eq:quot-vs-p}.

We may assume that $\RRR = \bb{C}[\![x]\!][y]/(f)$, where $f(x, y) = 0$ defines a generically separable, degree-$n$ cover of the $x$-axis fully ramified at $(x, y) = (0, 0)$.
For $\nu$ an integer composition of $n$, there is a scheme of finite type $\cal{H}_\nu^\ell$, \emph{resp.}\@ $\cal{Q}_\nu^\ell$, whose $\bb{C}$-points are pairs $(M, F)$ in which $M$ is a point of $\cal{H}^\ell$, \emph{resp.}\@ $\cal{Q}_\ell$, and $F$ is a $y$-stable flag on $\bar{M} \vcentcolon= M/xM$ of parabolic type $\nu$.
Let $\sf{Hilb}_\nu$ and $\sf{Quot}_\nu$ be the analogues of $\sf{Hilb}$ and $\sf{Quot}$ for these schemes.
Then \Cref{conj:main} is refined by:

\begin{mainconj}[Parabolic Hilb-vs-Quot]\label{conj:main-flag}
	For any $\RRR$ and $\nu$ as above,
	\begin{align}
		\sf{Hilb}_\nu(\sf{q}, \sf{t}) = \sf{Quot}_\nu(\sf{q}, \sf{q}^{\frac{1}{2}} \sf{t}) .
	\end{align}
\end{mainconj}

Let $\CptPic_\nu, \CptPic_\nu(\gap)$ be defined analogously to $\cal{H}_\nu^\ell, \cal{Q}_\nu^\ell$.
There is an obvious refinement of \eqref{eq:quot-vs-p} using $\sf{Quot}_\nu$ and $\sf{Pic}_\nu$.
We can make a further motivic improvement.
Let $\Schfin{\bb{C}}$ be the category of $\bb{C}$-schemes of finite type, and for any object $X$ of $\Schfin{\bb{C}}$, let $[X]$ denote its class in the Grothendieck ring of $\Schfin{\bb{C}}$.
The virtual weight polynomial of $X$ is a specialization of $[X]$.
Let
\begin{align}
	\sf{Quot}_\nu^\mot(\sf{q})
	= \sum_\ell \sf{q}^\ell [\cal{Q}^\ell(\nu)]
	\quad\text{and}\quad
	\sf{Pic}_\nu^\mot(\sf{q})
	= \sum_\gap \sf{q}^\gap [\CptPic_\nu(\gap)/\Lattice].
\end{align}
In \Cref{sec:quot}, we prove:

\begin{mainthm}\label{thm:main}
	For any $\RRR$ and $\nu$ as above,
	\begin{align}
		\sf{Quot}_\nu^\mot(\sf{q})
		= \frac{1}{(1 - \sf{q})^b}\,\sf{Pic}_\nu^\mot(\sf{q}).
	\end{align}
\end{mainthm}

\noindent
The main idea is to embed $\coprod_\ell \cal{Q}^\ell$ into $\CptPic$, then relate $\ell$ to $\gap$ by way of a certain fundamental domain for the $\Lattice$-action.

Let $\Lambda_{\sf{q}, \sf{t}}^n$ be the vector space of degree-$n$ symmetric functions in infinitely many variables over $\bb{Q}(\sf{q}, \sf{t})$.
In \Cref{sec:springer}, we explain that \Cref{conj:main-flag} and the virtual weight specialization of \Cref{thm:main} can be rephrased in terms of elements $\cal{F}\sf{Hilb}$, $\cal{F}\sf{Quot}$, $\cal{F}\sf{Pic}$ of $\Lambda_{\sf{q}, \sf{t}}^n$, which recover the corresponding $\sf{q}, \sf{t}$-series involving $\nu$ via the Hall pairing with the homogeneous symmetric function $h_\nu$.
See \eqref{eq:conj-main-flag} and \eqref{eq:thm-main} for the precise formulas.

\subsection{}\label{subsec:ors}

Henceforth, we take $f(x, y) \in \bb{C}[x, y]$.
Fix a $3$-sphere in $\bb{C}^2$ around $(0, 0)$.
The intersection of the zero locus $\{f(x, y) = 0\}$ with this $3$-sphere is a topological link $L_f$, whose isotopy class depends only on $f$ when the sphere is small enough.
The number of branches $b$ is the number of connected components of $L_f$.

There is an isotopy invariant of links taking values in triply-graded vector spaces, known as HOMFLYPT or Khovanov--Rozansky (KhR) homology \cite{dgr, khr}.
In \cite{ors}, Oblomkov--Rasmussen--Shende conjectured an identity expressing the KhR homology of $L_f$ in terms of the Hilbert schemes $\cal{H}^\ell$.
The full statement requires nested versions $\cal{H}_{\nest{m}}^\ell \subseteq \cal{H}^\ell \times \cal{H}^{\ell + m}$, parametrizing pairs of ideals $(I, J)$ such that $xI + yI \subseteq J \subseteq I$.

For any link $L$, let $\bar{\cal{P}}_{L, \ORS}(a, q, t)$ be the graded dimension of the unreduced KhR homology of $L$ in the conventions of \cite{ors}, so that our $\bar{\cal{P}}_{L, \ORS}$ is their $\bar{\cal{P}}(L)$.
We will use a normalization $\bar{\sf{X}}_f(\sf{a}, \sf{q}, \sf{t}) \in \bb{Z}[\![\sf{q}]\!][\sf{a}^{\pm 1}, \sf{t}^{\pm 1}]$ satisfying
\begin{align}\label{eq:khr-normalization}
	\bar{\cal{P}}_{L_f, \ORS}(a, q, t)
	&=
	(aq^{-1})^{2\delta - b}\,
	\bar{\sf{X}}_f(a^2 t, q^2, q^2 t^2).
\end{align}
Then the ORS Conjecture \cite[Conj.\ 2]{ors} states that
\begin{align}\label{eq:ors-hilb}
	\bar{\sf{X}}_f(\sf{a}, \sf{q}, \sf{q}\sf{t}^2)
	= \sum_{\ell, m}
	\sf{q}^\ell
	\sf{a}^{2m}
	\sf{t}^{m(m - 1)}
	\chi(\cal{H}_{\nest{m}}^\ell, \sf{t}).
\end{align}
Note that this conjecture would imply that the virtual weight polynomials above contain only even powers of $\sf{t}$.

It was essentially observed in \cite{gors} that once we fix the presentation of $f(x, y) = 0$ as a degree-$n$ cover of the $x$-axis, the right-hand side of \eqref{eq:ors-hilb} can be written in terms of $\cal{F}\sf{Hilb}$.
Namely, let $\Psi(\sf{a}, -) : \Lambda_{\sf{q}, \sf{t}}^n \to \bb{Q}(\sf{q}, \sf{t})[\sf{a}]$ be the map 
\begin{align}\label{eq:hook}
	\Psi(\sf{a}, -)
	= (1 + \sf{a}) \sum_{0 \leq k \leq n - 1}
	{\sf{a}^k}
	\langle s_{(n - k, 1^k)}, -\rangle,
\end{align}
where $s_\mu \in \Lambda_{\sf{q}, \sf{t}}^n$ is the Schur function indexed by $\mu \vdash n$ and $\langle -, -\rangle$ is the Hall inner product.
In \Cref{sec:springer}, we explain that the right-hand side of \eqref{eq:ors-hilb} is $\Psi(\sf{a}, \cal{F}\sf{Hilb}(\sf{q}, \sf{t}))$.
So the ORS Conjecture is:
\begin{align}
	\bar{\sf{X}}_f(\sf{a}, \sf{q}, \sf{q}\sf{t}^2)
	= \Psi(\sf{a}, \cal{F}\sf{Hilb}(\sf{q}, \sf{t})).
\end{align}
In particular, if \Cref{conj:main-flag} (the Parabolic Hilb-vs-Quot Conjecture) holds, then \eqref{eq:ors-hilb} is equivalent to the following:

\begin{mainconj}[KhR-vs-Quot]\label{conj:ors-quot}
	For any $f$ as above,
	\begin{align}\label{eq:ors-quot}
		\bar{\sf{X}}_f(\sf{a}, \sf{q}, \sf{t}^2)
		= \Psi(\sf{a}, \cal{F}\sf{Quot}(\sf{q}, \sf{t})).
	\end{align}
\end{mainconj}

\begin{rem}
	When $L$ is the link closure of a braid $\beta$, the KhR homology of $L$ can be computed from the Rouquier complex of Soergel bimodules $\bar{\cal{T}}_\beta$, as we explain in \Cref{sec:conventions}.
	There is a richer invariant of $\bar{\cal{T}}_\beta$: its (dg) horizontal trace $\sf{Tr}_\ur{dg}(\bar{\cal{T}}_\beta)$.
	In \cite{ghw}, Gorsky--Hogancamp--Wedrich show that when $\beta$ has $n$ strands, $\sf{Tr}_\ur{dg}(\bar{\cal{T}}_\beta)$ decategorifies to an element of $\Lambda_{\sf{q}, \sf{t}}^n$, and the KhR homology of $L$ can be obtained by specializing $\sf{Tr}_\ur{dg}(\bar{\cal{T}}_\beta)$ along a version of $\Psi$.
	It is natural to expect that \Cref{conj:ors-quot} has a further refinement, taking $\beta$ to be the positive braid in $\{f(x, y) = 0\}$ that lifts a positive loop around $x = 0$, and comparing $\cal{F}\sf{Quot}$ directly to $\sf{Tr}_\ur{dg}(\bar{\cal{T}}_\beta)$.
	
	In \cite{trinh}, for any positive braid $\beta$ on $n$ strands, the second author introduced a (derived) scheme $\mathcal{Z}(\beta)$ with an action of $\GL_n$ and a Springer-type action of $S_n$ on its $\GL_n$-equivariant compactly-supported cohomology.
	The $S_n$-action on the associated graded of the weight filtration recovers an \emph{underived} horizontal trace.
	But there is no direct relationship between $[\mathcal{Z}(\beta)/{\GL_n}]$ and $\cal{Q}^\ell$.
\end{rem}

\subsection{}\label{subsec:intro-toric}

As we vary $\RRR$ in families, the Quot schemes $\cal{Q}^\ell$ do not deform as nicely as the Hilbert schemes $\cal{H}^\ell$, because in any versal deformation of $\RRR$, we can only deform $\SSS$ jointly with $\RRR$ in the stratum where $\delta$ is constant \cite{teissier}.
Nonetheless, \Cref{conj:ors-quot} is significantly more tractable than the original ORS Conjecture.

We will establish \Cref{conj:ors-quot} for two infinite families of plane curve germs with $\bb{C}^\times$-actions.
In what follows, we write $\cal{F}\sf{Quot}_{n, d}$, $\cal{F}\sf{Pic}_{n, d}$, $\bar{\sf{X}}_{n, d}$ in place of $\cal{F}\sf{Quot}$, $\cal{F}\sf{Pic}$, $\bar{\sf{X}}_f$ when $f(x, y) = y^n - x^d$ for some $n, d > 0$.

\begin{mainthm}\label{thm:ors-quot}
	Suppose that either of the following holds:
	\begin{enumerate}
		\item 	$d$ is coprime to $n$.
		\item 	$d$ is a multiple of $n$.
	\end{enumerate}
	Then \Cref{conj:ors-quot} holds for $f(x, y) = y^n - x^d$.
\end{mainthm}

We prove case (1) of \Cref{thm:ors-quot}, the coprime case, in \Cref{sec:coprime}.
We actually give two independent proofs:
\begin{enumerate}
	\item[(A)] 	The first extends the combinatorial commutative algebra in the proof of \cite[Cor.\@ A.5]{ors}, thereby relating $\Psi(\cal{F}\sf{Quot}_{n, d})$ to the formula for $\bar{\sf{X}}_{n, d}$ conjectured by Gorsky--Negu\c{t} in \cite{gn} and proved by Mellit in \cite{mellit_22}.
	
	\item[(B)] 	The second proof is more roundabout.
	We invoke \Cref{thm:main}, then relate $\Psi(\cal{F}\sf{Pic}_{n, d})$ to $\bar{\sf{X}}_{n, d}$ by work of Hikita \cite{hikita}, Mellit \cite{mellit_21}, Hogancamp--Mellit \cite{hm}, and Wilson \cite{wilson}.
	Our new contribution is to match the \dfemph{gap filtration} of $\CptPic/\Lattice$ induced by $\gap$ with Hikita's filtration on an isomorphic variety, up to a further involution.
	
\end{enumerate}
The Gorsky--Negu\c{t} formula in (A) implicitly involves certain semigroup modules and their generators, while the Hogancamp--Mellit recursion in (B) implicitly yields a formula for $\bar{\sf{X}}_{n, d}$ in terms of the ``cogenerators'' of these semigroup modules, by work of Gorsky--Mazin--Vazirani \cite{gmv}.
As we explain in \Cref{sec:coprime}, these formulas have the same lowest $\sf{a}$-degree, but differ in higher $\sf{a}$-degrees.

We prove case (2), where $d = nk$ for some integer $k$, in \Cref{sec:noncoprime}.
Here, the tools we need were developed in settings with extra structure: $y$-ified link homology on the KhR side, which we review in \Cref{sec:polynomial}, and torus-equivariant homology on the Quot side.
We relate $\bar{\sf{X}}_{n, nk}$ and $\Psi(\sf{a}, \cal{F}\sf{Quot}_{n, d}(\sf{q}, \sf{t}))$ to the same expression $\Psi(\sf{a}, \nabla^k p_{(1^n)})$, where $\nabla$ is Bergeron--Garsia's nabla operator and $p_{(1^n)}$ is a power-sum symmetric function, via work of Gorsky--Hogancamp \cite{GH17} and Carlsson--Mellit \cite{CM21}, respectively.

Both proof (B) of case (1) and the proof of case (2) involve comparisons to the \dfemph{affine Springer fibers} of \cite{kl}.
For the former, we match $\CptPic_{(1^n)}/\Lattice$ with the $\SL_n$ affine Springer fiber studied in \cite{hikita}.
For the latter, we match $\coprod_\ell \cal{Q}^\ell$ with the positive part of the $\GL_n$ affine Springer fiber studied in \cite{CM21}.
These steps are relegated to \Cref{sec:asf}.

\begin{rem}
	In \cite{turner}, generalizing the $\GL_3$ case of \cite{Kiv20}, Turner computes the Borel--Moore homologies of many unramified affine Springer fibers for $\GL_3$.
	This verifies the $(\sf{a}, \sf{q}) \to (0, 1)$ limit of \Cref{conj:ors-quot} for the associated plane curve germs, up to a certain localization.
\end{rem}

\subsection{}

Despite the claims in \cite[\S{1.2}]{hm} and \cite[\S{6.2}]{BLMSnotes}, we believe there is no proof of the original ORS Conjecture that covers either of the two cases in \Cref{thm:ors-quot}.
As we explain in \Cref{sec:coprime}, there does exist a combinatorial formula for $\cal{F}\sf{Hilb}_{n, d}$ when $n$ and $d$ are coprime, which was originally obtained in \cite[Cor.\@ A.5]{ors}, but it is much harder to match with $\bar{\sf{X}}_{n, d}$ than the analogous formula for $\cal{F}\sf{Quot}_{n, d}$.

Oblomkov--Rasmussen--Shende did verify their full conjecture when $f(x, y) = y^2 - x^d$ with $d$ odd.
As the map $\Psi$ loses no information for $n \leq 3$, this implies \Cref{conj:main-flag} for such $f$ via case (1) of \Cref{thm:ors-quot}.
Remarkably, we can use case (1) of \Cref{thm:ors-quot} to prove:

\begin{mainthm}\label{thm:n-equals-3}
	\Cref{conj:main-flag} holds for $f(x, y) = y^3 - x^d$ with $d$ coprime to $3$.
	Hence, the original ORS Conjecture \eqref{eq:ors-hilb} also holds for these cases.
\end{mainthm}

We give the proof at the end of \Cref{sec:coprime}.
First, we show that for $d$ coprime to $n$, the functions $\Psi(\sf{a}, \cal{F}\sf{Hilb}_{n, d}(\sf{q}, \sf{t}))$ and $\Psi(\sf{a}, \cal{F}\sf{Quot}_{n, d}(\sf{q}, \sf{q}^{\frac{1}{2}} \sf{t}))$ match in the limit where $d \to \infty$: 
See \Cref{prop:asymptotic}.
As the original functions agree with their limits up to order $d$ in $\sf{q}$, we can use symmetry properties on both sides to recover the finite identity from the asymptotic one.
On the $\cal{F}\sf{Quot}_{n, d}$ side, the necessary symmetry arises via \Cref{thm:ors-quot}(1) from KhR homology, where it was conjectured in \cite{dgr} and proved in \cite{or, ghm}.

At the end of \Cref{sec:noncoprime}, we verify the lowest $\sf{a}$-degree part of the original ORS Conjecture for the curves $y^2 = x^4$ and $y^3 = x^3$, using \cite{Kiv20}.
This proves \Cref{conj:main} for those curves, via case (2) of \Cref{thm:ors-quot}.

\begin{rem}
	In proving \Cref{thm:n-equals-3}, we also establish a closed formula for the KhR homology of $(3, d)$ torus knots that had been conjectured by Dunfield--Gukov--Rasmussen \cite[Conj.\ 6.3]{dgr}.
	See the discussion in \cite[\S{4}]{ors}.
\end{rem}

\subsection{}

Our proof of case (2) of \Cref{thm:ors-quot}, together with unpublished work of the first author, suggests a refinement of \Cref{conj:ors-quot}: one with no analogue in the Hilbert-scheme setting of \cite[Conj.\ 2]{ors}.

In general, if $L$ is a link with $b$ components, then $\bb{C}[\vec{x}] \vcentcolon= \bb{C}[x_1, \ldots, x_b]$ acts on the KhR homology of $L$ by \cite[Cor.\@ 5.4]{BLMSnotes}.
The $y$-ified KhR homology of $L$, introduced by Gorsky--Hogancamp in \cite{GH17}, is a monodromic deformation of its KhR homology that extends scalars from $\bb{C}$ to $\bb{C}[\vec{y}] \vcentcolon= \bb{C}[y_1, \ldots, y_b]$, thereby extending 
the $\bb{C}[\vec{x}]$-action to a $\bb{C}[\vec{x}, \vec{y}]$-action.

We conjecture that these actions match similar actions on the homology of the Quot schemes $\cal{Q}_\nu^\ell$.
Letting $\Lattice_{\geq 0} \subseteq \Lattice$ be the submonoid of $\vec{\varpi}^{\vec{x}}$ with $\vec{x} \in \bb{Z}_{\geq 0}^b$, we see that $\Lattice_{\geq 0}$ acts on $\coprod_\ell \cal{Q}_\nu^\ell$ and commutes with the torus $T(b) \vcentcolon= \bb{G}_m^b$ that rescales the $\varpi_i$.
Up to isomorphisms $\bb{C}[\vec{x}] \simeq \bb{C}[\Lattice_{\geq 0}]$ and $\bb{C}[\vec{y}] \simeq \ur{H}^\ast_{T(b)}(\point)$, we get a $\bb{C}[\vec{x}, \vec{y}]$-action on the sum of modified equivariant Borel--Moore homologies
\begin{align}\label{eq:xy}
	\bigoplus_\ell \gr_\ast^\sf{W} \ur{H}_\ast^{\BM,  T(b)}(\cal{Q}_\nu^\ell),
	\quad\text{where $\sf{W}_{\leq \ast}$ denotes weight filtrations}.
\end{align}
The $x_i$ shift $\ell$ by $1$ and preserve weights, and the $y_i$ preserve $\ell$ and shift weights by $-2$.
As $\nu$ varies, these $\bb{C}[\vec{x}, \vec{y}]$-modules can be packaged together into a bigraded $(\bb{C}[\vec{x}, \vec{y}] \times \bb{C}S_n)$-module $\smash{\widetilde{\sf{Quot}}}^{\vec{x}, \vec{y}}$.
The map $\Psi$ categorifies to a functor from such bigraded modules to triply-graded $\bb{C}[\vec{x}, \vec{y}]$-modules.
Abusing notation, we again write $\Psi$ to denote this functor.
We can now state the following refinement of \Cref{conj:ors-quot}, with more explicit details left to \Cref{sec:polynomial}.

\begin{mainconj}[$y$-ified KhR-vs-Quot]\label{conj:ors-quot-enhanced}
	The $y$-ified KhR homology of $L_f$ is isomorphic as a triply-graded $\bb{C}[\vec{x}, \vec{y}]$-module to $\Psi(\smash{\widetilde{\sf{Quot}}}^{\vec{x}, \vec{y}})$, after appropriate regrading.
\end{mainconj}

\subsection{Acknowledgments}

We are grateful to Francesca Carocci, Eugene Gorsky, Andy Wilson, and Zhiwei Yun for helpful discussions about \cite{teissier}, \cite{gmv}, \cite{wilson}, and \cite{cherednik}, respectively, and to Nathan Williams for informing us about rowmotion.
We also thank the referee for reading our paper carefully and offering many good suggestions.
During this work, the second author was supported by an NSF Mathematical Sciences Research Fellowship, Award DMS-2002238.

\section{Quot and Picard Schemes}\label{sec:quot}

\subsection{}

The main goal of this section is to prove \Cref{thm:main}.
We keep the definitions of $\RRR$, $\SSS$, $K$, $b$ from the introduction.

\subsection{}

First, we review the formal definition of the \dfemph{compactified Picard scheme} \cite[\S{3.10}]{my}.
Let $\fr{m}_\RRR$ be the maximal ideal of $\RRR$, and for any $\RRR$-module $E$, let $(-) \hatotimes E$ be the tensor product with $E$ completed in the $\fr{m}_\RRR$-adic topology on $E$.
Let $\CptPic_\unred$ be the functor from $\bb{C}$-algebras to sets defined by
\begin{align}
	\CptPic_\unred(A)
	&= \left\{
	\begin{array}{l}
		\text{$(A \hatotimes \RRR)$-submodules}\\
		M \subseteq A \hatotimes K
	\end{array}
	\middle|
	\begin{array}{l}
		\text{$\exists\, i$ s.t.\ $A \hatotimes \fr{m}_\RRR^i \subseteq M \subseteq A \hatotimes \fr{m}_\RRR^{-i}$}\\
		\text{and $(A \hatotimes \fr{m}_\RRR^{-i})/M$ is locally free}\\
		\text{over $A$ of finite rank}
	\end{array}
	\right\}
\end{align}
for any $\bb{C}$-algebra $A$.
An argument in \cite[\S{2}]{gortz} shows that $\CptPic_\unred$ is representable by an ind-scheme.
Let $\CptPic \subseteq \CptPic_\unred$ be the underlying reduced ind-scheme.
Taking $A = \bb{C}$ recovers
\begin{align}
	\CptPic(\bb{C}) = \CptPic_\unred(\bb{C})
	&= \{
	\text{finitely-generated $R$-submodules $M \subseteq K$}
	\mid
	KM = K
	\},
\end{align}
as in the introduction.

\begin{rem}\label{rem:reduced}
	Even though $\CptPic_\unred, \CptPic$ have the same $\bb{C}$-points, it is only $\CptPic$ that forms a scheme locally of finite type.
	For instance, if $\RRR = \bb{C}[\![x]\!]$, then $\CptPic_\unred \simeq x^\bb{Z} \times \CptPic_\unred^\mathit{nil}$, where $\CptPic_\unred^\mathit{nil}(A)$ parametrizes Laurent tails in $x$ where each coefficient is a nilpotent element of $A$; by contrast, $\CptPic \simeq x^\bb{Z}$.
\end{rem}

\subsection{}

For any integer $\gap$, let $\CptPic(\gap) \subseteq \CptPic$ be the sub-ind-scheme defined by
\begin{align}
	\CptPic(\gap)(A) = \{M \in \CptPic(A) \mid \text{$(\SSS M)/M$ is locally free over $A$ of rank $\gap$}\}.
\end{align}

\begin{prop}
	If $M \in \CptPic_\unred(A)$, then $(\SSS M)/M$ is locally free over $A$ of rank at most $\delta \vcentcolon= \sf{dim}_\bb{C}(\SSS/\RRR)$.
\end{prop}

\begin{proof}
	Observe that $(\SSS \otimes_\RRR M)/M$ is free over $A$ of rank $\delta$ because 
	\begin{align}
		(\SSS \otimes_\RRR M)/M \simeq ((A \hatotimes \SSS) \otimes_{A \hatotimes \RRR} M)/M \simeq (A \hatotimes \SSS)/(A \hatotimes \RRR).
	\end{align}
	Hence it suffices to show that $(\SSS M)/M$ is a direct summand of $(\SSS \otimes_\RRR M)/M$ as an $A$-module.
	
	Let $s_1, \ldots, s_\delta$ be a non-redundant (full) set of coset representatives for $\RRR$ in $\SSS$.
	Then $\SSS M = \sum_j {(s_j + R)M} = \sum_j s_j M$, so we can pick some subset $J \subseteq \{1, \ldots, \delta\}$, and $m_j \in M$ for $j \in J$, such that $\{s_j m_j\}_{j \in J}$ is a non-redundant set of coset representatives for $M$ in $\SSS M$.
	The $A$-linear map $(\SSS M)/M \to (\SSS \otimes_\RRR M)/M$ that sends $s_jm_j + M \mapsto s_j \otimes m_j + 1 \otimes M$ is an $A$-linear section of the natural surjective map $(\SSS \otimes_\RRR M)/M \to (\SSS M)/M$, as desired.
\end{proof}

\begin{cor}\label{cor:gap}
	$\CptPic$ is the union of the locally closed sub-ind-schemes $\CptPic(\gap)$ for $0 \leq \gap \leq \delta$. In fact, the locally closed subsets $\CptPic(\gap)$ form a stratification of $\CptPic$.
\end{cor}

\begin{proof}
	It remains to explain why the $\CptPic(\gap)$ are locally closed:
	This follows from the upper semicontinuity of rank.
\end{proof}

\subsection{}

Recall that we fix once and for all a uniformization $\SSS \xrightarrow{\sim} \prod_{i = 1}^b \bb{C}[\![\varpi_i]\!]$, and set $\Lattice = \{\varpi^{\vec{x}} \mid \vec{x} \in \bb{Z}^b\}$, where $\varpi^{\vec{x}} = \varpi_1^{x_1} \cdots \varpi_b^{x_b}$.
The group $\Lattice$ acts on $\CptPic$ by scaling.
Adapting the proof of \cite[Cor.\@ 1]{kl}, one can check that $\CptPic/\Lattice$ is a projective variety.
For all $\gap$, we see that $\CptPic(\gap)$ is $\Gamma$-stable, which lets us form the locally-closed subvariety $\CptPic(\gap)/\Lattice \subseteq \CptPic/\Lattice$.

For any finitely-generated $\RRR$-submodule $E \subseteq K$, let $\Quot^\ell(E)$ be the Quot scheme parametrizing submodules of $E$ of codimension $\ell$.
In more detail, let $\Quot^\ell_\unred(E)$ be the functor
\begin{align}
	\Quot^\ell_\unred(E)(A)
	&= \left\{
	\begin{array}{l}
		\text{$(A \hatotimes \RRR)$-submodules}\\
		M \subseteq A \hatotimes E
	\end{array}
	\middle|
	\begin{array}{l}
		\text{$\exists\, i$ s.t.\ $A \hatotimes \fr{m}_\RRR^i \subseteq M \subseteq A \hatotimes \fr{m}_\RRR^{-i}$}\\
		\text{and $(A \hatotimes \fr{m}_\RRR^{-i})/M$ is locally free}\\
		\text{over $A$ of rank $\ell$}
	\end{array}
	\right\}.
\end{align}
It is again representable by an ind--scheme.
We take $\Quot^\ell(E) \subseteq \Quot^\ell_\unred(E)$ to be the underlying reduced ind-scheme, which again gives a scheme locally of finite type. 
Note that these are not Quot schemes in the usual sense, but punctual versions.
There is a tautological inclusion $\Quot^\ell(E) \to \CptPic$.

As in the introduction, we set $\cal{Q}^\ell = \Quot^\ell(\SSS)$ and $\Lattice_{\geq 0} = \{\varpi^{\vec{x}} \mid \vec{x} \in \bb{Z}_{\geq 0}^b\}$.
The free action of $\Lattice$ on $\CptPic$ by scaling restricts to a free action of $\Lattice_{\geq 0}$ on $\coprod_\ell \cal{Q}^\ell$.
Moreover, for all $\ell$, we have
\begin{align}
	\label{eq:shift}
	(M, \vec{x}) \in \cal{Q}^\ell \times \bb{Z}_{\geq 0}^b
	&\implies \varpi^{\vec{x}} M \in \Quot^{\ell + \ur{sum}(\vec{x})}(\SSS),
\end{align}
where $\ur{sum}(\vec{x}) = e_1 + \cdots + e_b$.

\begin{lem}\label{lem:domain}
	The subscheme $\cal{D} \subseteq \coprod_\ell \cal{Q}^\ell$ defined by
	\begin{align}
		\cal{D}(A) = \{M \subseteq A \hatotimes \SSS \mid M \cap (A \hatotimes \SSS)^\times \neq \emptyset\}.
	\end{align}
	is a fundamental domain for both the $\Lattice$-action on $\CptPic$ and the $\Lattice_{\geq 0}$-action on $\coprod_\ell \cal{Q}^\ell$.
\end{lem}

\begin{proof}
	If $u = u(\varpi_1, \ldots, \varpi_b)$ belongs to $(A \hatotimes \SSS)^\times$, then the constant term of $u$ must belong to $A^\times$.
	We deduce that if $M \in \cal{D}(A)$, then $\varpi^{\vec{x}} M  \in \cal{D}(A)$ occurs only when $\vec{x}$ is the zero vector.
	Therefore $\cal{D}$ is irredundant under the action of $\Lattice$ on $\CptPic$.
	
	It remains to show that every element $M \in \CptPic(A)$, \emph{resp.}\ $M \in \coprod_\ell \cal{Q}^\ell(A)$, takes the form $\varpi^{\vec{x}} M'$ for some $M' \in \cal{D}(A)$ and $\vec{x} \in \bb{Z}^b$, \emph{resp.}\ $\vec{x} \in \bb{Z}_{\geq 0}^b$.
	Observe that $KM = A \hatotimes K$, because once we pick $i \geq 0$ such that $M \supseteq A \hatotimes \fr{m}_R^i$, we obtain $KM \supseteq A \hatotimes K\fr{m}_R^i = A \hatotimes K$.
	Therefore, $KM \ni 1$, which means we can find some $u \in (A \hatotimes K)^\times$ and $m \in M$ such that $um = 1$.
	This in turn means $m = u^{-1} \in M \cap (A \hatotimes K)^\times = M \cap \prod_{i = 1}^b A(\!(\varpi_i)\!)^\times$.
	
	In the case of $\coprod_\ell \cal{Q}^\ell$, we conclude as follows:
	Since $m \in A \hatotimes \SSS = \prod_{i = 1}^b A[\![\varpi_i]\!]$ as well, we get $m = \varpi^{\vec{x}}m'$ for some $\vec{x} \in \bb{Z}_{\geq 0}^b$ and $m' \in (A \hatotimes \SSS)^\times$ by factoring out the largest powers of the uniformizers $\varpi_i$ from $m$.
	
	In the case of $\CptPic$, we conclude as follows:
	Write $m = (m_i)_{i = 1}^b$ with $m_i \in A(\!(\varpi_i)\!)^\times$.
	The fact that $\CptPic$ is the underlying reduced ind-scheme of $\CptPic_\unred$ means that we can assume, by reduction to the $b = 1$ case in \Cref{rem:reduced}, that for all $i$, the coefficient of the lowest-degree term of $m_i$ is a unit, not a nilpotent element, of $A$.
	Now we get $m = \varpi^{\vec{x}} m'$ for some $\vec{x} \in \bb{Z}^b$ and $m' \in (A \hatotimes S)^\times$ by factoring, as before.
\end{proof}

\begin{lem}\label{lem:gap-vs-len}
	For any $\bb{C}$-algebra $A$ and $M \in \coprod_\ell \cal{Q}^\ell(A)$, we have
	\begin{align}
		M \in \cal{D}(A)
		\stackrel{(1)}{\iff} \SSS M = A \hatotimes \SSS
		\stackrel{(2)}{\iff} \rk_A((\SSS M)/M) = \rk_A((A \hatotimes \SSS)/M).
	\end{align}
	In particular, $\cal{D}(\bb{C}) = \coprod_\ell \{M \in \cal{Q}^\ell(\bb{C}) \mid \gap(M) = \ell\}$.
\end{lem}

\begin{proof}
	Equivalence (2) holds because $SM \subseteq A \hatotimes \SSS$.
	
	As for equivalence (1):
	$\SSS M = A \hatotimes \SSS$ if and only if $\SSS M \ni 1$, if and only if $sm = 1$ for some $s \in \SSS$ and $m \in M$.
	By the explicit description of $A \hatotimes \SSS$, the last condition is equivalent to $s'm' = 1$ for some $s' \in (A \hatotimes \SSS)^\times$ and $m' \in M$, which means $M \in \cal{D}(A)$.
\end{proof}

\subsection{}

Using Weierstrass preparation, we now fix an isomorphism
\begin{align}
	\RRR = \bb{C}[\![x, y]\!]/(f)
\end{align}
such that $\Spec(\RRR) \to \Spec(\bb{C}[\![x]\!])$ is a generically separable cover of degree $n$, fully ramified at $(x, y) = (0, 0)$.
In particular, we can take 
$f(x, y) = y^n + \sum_{i = 1}^n a_i(x) y^{n - i}$ for some $a_i(x) \in \bb{C}[\![x]\!]$ with $a_i(0) = 0$ for all $i$.

For any $\bb{C}$-algebra $A$ and $M \in \CptPic(A)$, we write $\bar{M} = M/xM$, as in the introduction.
We define a $y$-stable partial flag on $\bar{M}$ to be an increasing sequence of $A[y]$-submodules $F = (0 \subseteq \bar{M}^0 \subsetneq \bar{M}^1 \subsetneq \cdots \subsetneq \bar{M}^k = \bar{M})$ such that $\gr_i^F(\bar{M}) = \bar{M}^i/\bar{M}^{i - 1}$ is locally free over $A$ for all $i$.
The \dfemph{parabolic type} of $F$ is the integer composition $\nu$ of $n$ in which $\nu_i = \rk_A(\gr_i^F(\bar{M}))$.
For any such composition $\nu$, let $\CptPic_\nu$ be the ind-scheme defined by
\begin{align}
	\CptPic_\nu(A) = \left\{(M, F) \,\middle| 
	\begin{array}{l}
		M \in \CptPic_\nu(A),\\
		\text{$F$ is a $y$-stable partial flag on $\bar{M}$ of type $\nu$}
	\end{array}\right\}.
\end{align}
We define $\CptPic_\nu(\gap), \Quot_\nu^\ell(E), \cal{D}_\nu$ analogously.
Now, \Cref{cor:gap} and \Crefrange{lem:domain}{lem:gap-vs-len} imply analogues where $\CptPic_\nu$, $\CptPic_\nu(\gap)$, $\cal{Q}_\nu^\ell$, $\cal{D}_\nu$ replace $\CptPic$, $\CptPic(\gap)$, $\cal{Q}^\ell$, $\cal{D}$.

\subsection{Proof of \Cref{thm:main}}

Recall that we want to show
\begin{align}
	\sf{Quot}_\nu^\mot(\sf{q})
	&= \frac{1}{(1 - \sf{q})^b}\,\sf{Pic}_\nu^\mot(\sf{q}),
	&\text{where}\:
	&\left\{\begin{array}{l}
		\sf{Quot}_\nu^\mot(\sf{q})
		= \sum_\ell \sf{q}^\ell [\cal{Q}_\nu^\ell],\\
		\sf{Pic}_\nu^\mot(\sf{q})
		= \sum_\gap \sf{q}^\gap [\CptPic_\nu(\gap)/\Lattice].
	\end{array}\right.
\end{align}
\Cref{lem:domain} and \eqref{eq:shift} together imply that
\begin{align}
	\cal{Q}_\nu^\ell = \coprod_{\substack{(\gap, \vec{x}) \in \bb{Z}_{\geq 0} \times \bb{Z}_{\geq 0}^b \\ \gap + \ur{sum}(\vec{x}) = \ell}} \varpi^{\vec{x}} \cdot (\cal{D} \cap \cal{Q}_\nu^\gap).
\end{align}
So in the Grothendieck group of $\Schfin{\bb{C}}$, we have
\begin{align}
	\sf{Quot}_\nu^\mot(\sf{q})
	= \frac{1}{(1 - \sf{q})^b}
	\sum_\gap \sf{q}^\gap [\cal{D} \cap \cal{Q}_\nu^\gap].
\end{align}
But \Cref{lem:gap-vs-len} implies that $\cal{D} \cap \cal{Q}_\nu^\gap = \cal{D} \cap \CptPic_\nu(\gap)$ for all $\gap$.
So we also have
\begin{align}
	\sf{Pic}_\nu^\mot(\sf{q})
	= \sum_\gap \sf{q}^\gap [\cal{D} \cap \CptPic_\nu(\gap)]
	= \sum_\gap \sf{q}^\gap [\cal{D} \cap \cal{Q}_\nu^\gap]
\end{align}
in the Grothendieck group, as desired.

\begin{rem}
	Zhiwei Yun has pointed out to us that \Cref{thm:main} extends beyond the planar case to any curve germ where both sides are well-defined, \emph{i.e.}, where the functors $\CptPic/\Lattice$ and $\cal{Q}^\ell$ for $\ell \geq 0$ are all schemes of finite type.
	
	However, \Cref{conj:main} fails for non-planar germs.
	If $\RRR = \bb{C}[\![x, y, z]\!]/(xy, xz, yz)$, the union of the coordinate axes in $xyz$-space, then $\SSS = \bb{C}[\![x]\!] \times \bb{C}[\![y]\!] \times \bb{C}[\![z]\!]$.
	Using \cite[Prop.\@ 6.1]{brv}, we find that
	\begin{align}
		\sf{Hilb}(\sf{q}, \sf{t})
		= \frac{1}{(1 - \sf{q})^3}
		(1 - 2\sf{q} + \sf{q}^2 (\sf{t}^4 + \sf{t}^2 + 1) + \sf{q}^3 (\sf{t}^4 - 2\sf{t}^2)).
	\end{align}
	By contrast, \cite[Ex.\ 2.7--2.8]{yun_13}, \cite{katz}, and \Cref{thm:main} together give
	\begin{align}
		\sf{Quot}(\sf{q}, \sf{q}^{\frac{1}{2}} \sf{t})
		= \frac{1}{(1 - \sf{q})^3}
		(1 - 2\sf{q} + \sf{q}^2 (\sf{t}^2 + 1) + \sf{q}^3 (\sf{t}^4 - 2\sf{t}^2) + \sf{q}^4 \sf{t}^4).
	\end{align}
	It would be interesting to understand why the difference remains small.
\end{rem}

\subsection{Cherednik's Conjecture}\label{subsec:cherednik}

Below, we explain how \Cref{conj:main} is essentially equivalent to Conjecture 4.5 of \cite{cherednik} via \Cref{thm:main}.

For any integer $e$, let $\CptPic^e \subseteq \CptPic$ be the sub-ind-scheme defined by
\begin{align}
	\CptPic^e(A) = \{M \in \CptPic(A) \mid \text{$e = \rk_A((A \hatotimes \fr{m}_R^{-i})/M) - \rk_\bb{C}(\fr{m}_R^{-i}/\RRR)$ for all $i \gg 0$}\}.
\end{align}
These are precisely the connected components of $\CptPic$.
By definition, $\CptJac = \CptPic^0$.

The discussion in the introduction explained how \Cref{conj:main} and \Cref{thm:main} would together imply \eqref{eq:hilb-vs-p}.
In turn, the unibranch case of \eqref{eq:hilb-vs-p} implies \eqref{eq:cherednik-weight}, because if $b = 1$, then $\Lattice = \varpi^\bb{Z}$ acts simply transitively on the set of connected components of $\CptPic$, giving $\CptJac = \CptPic/\Lattice$.

Next, we explain how \eqref{eq:cherednik-weight} is related to a point-counting analogue.
Following Katz \cite{katz}, a class $[X] \in K_0(\Schfin{\bb{C}})$ is called \dfemph{strongly polynomial count} if and only if, for some finitely-generated subring $B \subseteq \bb{C}$, spreading out of $[X]$ to a class $[\cal{X}] \in K_0(\Schfin{B})$, and polynomial $p(X, \sf{t}) \in \bb{Z}[\sf{t}]$, we have $|\cal{X}_\bb{F}(\bb{F})| = p(X, q)$ for any finite field $\bb{F} = \bb{F}_q$ and ring morphism $B \to \bb{F}$.
Katz shows that in this case, $\chi(X, \sf{t}) = p(X, \sf{t}^2)$.
So if $[\cal{H}^\ell]$ and $[\CptJac(\gap)]$ are strongly polynomial count for all $\ell$ and $\gap$, then \eqref{eq:cherednik-weight} is equivalent to the statement that
\begin{align}\label{eq:cherednik}
	\sum_\ell t^\ell |\cal{H}^\ell_\bb{F}(\bb{F})||_{q \to qt} 
	\stackrel{?}{=} \frac{1}{1 - t}  \sum_\gap t^\gap |(\CptJac(\gap))_\bb{F}(\bb{F})|
\end{align}
for infinitely many (equivalently, all) finite fields $\bb{F} = \bb{F}_q$, where we have abused notation by conflating $\cal{H}^\ell$ and $\CptJac(\gap)$ with their spreadings out.

Lastly, we relate \eqref{eq:cherednik} to \cite[Conj.\ 4.5]{cherednik}.
In \emph{loc.\@ cit.}, Cherednik's $\cal{R}$ and $\cal{O} = \bb{C}[\![z]\!]$ are the respective analogues of our $\RRR$ and $\SSS = \bb{C}[\![\varpi]\!]$ over $\bb{F}$.
In particular, if they arise from $\RRR$ and $\SSS$ by spreading out, then his $J_\cal{R}(\bb{F})$ is our $\coprod_\gap \varpi^{-\gap} \CptJac(\gap)_\bb{F}(\bb{F})$, which is also $\cal{D}_\bb{F}(\bb{F})$; his $\cal{H}_\mot^0(q, t)$ is our $\sum_\gap t^\gap |\CptJac(\gap)_\bb{F}(\bb{F})|$; and his $Z(q, t)$ is our $\sum_\ell t^\ell |\cal{H}^\ell_\bb{F}(\bb{F})|$.
In this case, \eqref{eq:cherednik} coincides with \cite[Conj.\ 4.5]{cherednik}.

\subsection{}\label{subsec:examples}

We compute some minimal examples with $n = 2$ and $\nu$ trivial.

\begin{ex}\label{ex:node}
	Take $f(x, y) = y^2 - x^2$.
	Setting $\varpi = y + x$ and $\varrho = y - x$ lets us write $\RRR = \bb{C}[\![\varpi, \varrho]\!]/(\varpi\varrho)$ and $\SSS = \bb{C}[\![\varpi]\!] \times \bb{C}[\![\varrho]\!]$.
	
	For all integers $i, j$ and $\lambda \in \bb{C}^\times$, consider the $\RRR$-submodules of $K$ given by $M_{i, j, \lambda} = \langle (\varpi^i, \lambda \varrho^j)\rangle$ and $N_{i, j} = \langle (\varpi^i, 0), (0, \varrho^j)\rangle$.
	We compute that $\CptPic^e(\bb{C})$ consists of the $M_{i, j, \lambda}$ such that $i + j = e$ and the $N_{i, j}$ such that $i + j = e + 1$.
	With more work, one can check that $\CptPic^e$ is an infinite chain of projective lines intersecting transversely, where the sets $\{M_{i, j, \lambda} \mid \lambda \in \bb{C}^\times\}$ are copies of $\bb{G}_m \vcentcolon= \bb{A}^1 \setminus \{0\}$ and the points $N_{i, j}$ are the points of intersection.
	Embedding $\cal{H}^\ell$ and $\cal{Q}^\ell$ into $\CptPic$, we compute:
	\begin{itemize}
		\item 	$\cal{H}^0(\bb{C}) = \{M_{0, 0}\}$.
		
		\item 	$\cal{H}^\ell(\bb{C})$ for $\ell \geq 1$ consists of the $M_{i, j, \lambda}$ such that $i + j = \ell$ and the $N_{i, j}$ such that $i + j = \ell + 1$, for $i, j \geq 1$.
		
		\item 	$\cal{Q}^\ell(\bb{C})$ for any $\ell$ consists of the $M_{i, j, \lambda}$ such that $i + j = \ell - 1$ and the $N_{i, j}$ such that $i + j = \ell$, for $i, j \geq 0$.
		
	\end{itemize}
Finally, $\cal{D} = \{M_{0, 0, \lambda} \mid \lambda\} \sqcup \{N_{0, 0}\}$, where $\gap(N_{0, 0}) = 0$ and $\gap(M_{0, 0, \lambda}) = 1$.
We get
\begin{align}
\sf{Hilb}_{(2)}^\mot(\sf{q}) 
&\vcentcolon= \sum_\ell \sf{q}^\ell [\cal{H}^\ell]
&&= 1 + \sum_{\ell \geq 1} \sf{q}^\ell (\ell + (\ell - 1) [\bb{G}_m]),\\
\sf{Quot}_{(2)}^\mot(\sf{q}) 
&= \sum_\ell \sf{q}^\ell [\cal{Q}^\ell]
&&= \sum_{\ell \geq 0} \sf{q}^\ell (\ell + 1 + \ell [\bb{G}_m]),\\
\sf{Pic}_{(2)}^\mot(\sf{q}) 
&= \sum_\gap \sf{q}^\gap [\CptPic(\gap)/\Gamma]
&&= 1 + \sf{q}[\bb{G}_m].
\end{align}
\end{ex}

\begin{ex}\label{ex:cusp}
Take $f(x, y) = y^2 - x^3$.
Setting $x = \varpi^2$ and $y = \varpi^3$ lets us write $\RRR = \bb{C}[\![\varpi^2, \varpi^3]\!]$ and $\SSS = \bb{C}[\![\varpi]\!]$.

For all integers $i$ and $\lambda \in \bb{C}$, consider the $\RRR$-submodules of $K$ given by $M_{i, \lambda} = \langle \varpi^i + \lambda \varpi^{i + 1}\rangle$ and $N_i = \langle \varpi^i, \varpi^{i + 1}\rangle$.
We compute that $\CptPic^e(\bb{C})$ consists of the $M_{e - 1, \lambda}$, for all $\lambda \in \bb{C}$, and $N_e$.
One can check that $\CptPic^e$ is a projective line in which $\{M_{e - 1, \lambda} \mid \lambda \in \bb{C}\}$ corresponds to $\bb{A}^1$ and and $N_e$ corresponds to $\infty$.
Embedding $\cal{H}^\ell$ and $\cal{Q}^\ell$ into $\CptPic$, we compute:
\begin{itemize}
\item 	$\cal{H}^0(\bb{C}) = \{M_0\}$ and $\cal{H}^1(\bb{C}) = \{N_2\}$.

\item 	$\cal{H}^\ell(\bb{C})$ for $\ell \geq 2$ consists of the $M_{\ell, \lambda}$, for all $\lambda \in \bb{C}$, and $N_{\ell + 1}$.

\item 	$\cal{Q}^0(\bb{C}) = \{N_0\}$.

\item 	$\cal{Q}^\ell(\bb{C})$ for $\ell \geq 1$ consists of the $M_{\ell - 1, \lambda}$, for all $\lambda \in \bb{C}$, and $N_\ell$.

\end{itemize}
Finally, $\cal{D} = \{M_{0, \lambda} \mid \lambda\} \sqcup \{N_0\}$, where $\gap(N_0) = 0$ and $\gap(M_{0, \lambda}) = 1$.
We get
\begin{align}
\sf{Hilb}_{(2)}^\mot(\sf{q}) 
&\vcentcolon= \sum_\ell \sf{q}^\ell [\cal{H}^\ell]
&&= 1 + \sf{q} + \sum_{\ell \geq 2} \sf{q}^\ell [\bb{P}^1],\\
\sf{Quot}_{(2)}^\mot(\sf{q}) 
&= \sum_\ell \sf{q}^\ell [\cal{Q}^\ell]
&&= 1 + \sum_{\ell \geq 1} \sf{q}^\ell [\bb{P}^1],\\
\sf{Pic}_{(2)}^\mot(\sf{q}) 
&= \sum_\gap \sf{q}^\gap [\CptPic(\gap)/\Gamma] 
&&= 1 + \sf{q}[\bb{A}^1].
\end{align}
\end{ex}

\section{Springer Actions}\label{sec:springer}

\subsection{}\label{subsec:partial-flag}

In this section, we explain how the collection of polynomials $\{\chi(X_\nu, \sf{t})\}_\nu$, where $X_\nu$ is one of $\CptPic_\nu/\Lattice$, $\CptPic_\nu(\gap)/\Lattice$, $\Quot_\nu^\ell(E)$, \emph{etc.}, can be packaged into a single symmetric function.
We also introduce variants of these schemes that we will need in \Crefrange{sec:coprime}{sec:noncoprime}.
Throughout, we will use the formalism of quotient stacks (in the smooth topology), but keep our exposition self-contained beyond the definition of a stack via its functor of points.

\subsection{}

Fix an integer $n > 0$.
Let $\cal{N}$ be the variety of nilpotent matrices in $\fr{gl}_n$.
By definition, $[\cal{N}/{\GL_n}]$ is the algebraic stack whose $A$-points form the groupoid of pairs $(V, \theta)$, where $V$ is a locally-free $A$-module of rank $n$ and $\theta$ is a nilpotent endomorphism of $V$, and an isomorphism of pairs $(V, \theta) \xrightarrow{\sim} (V', \theta')$ is an isomorphism of $A$-modules $V \xrightarrow{\sim} V'$ that transports $\theta$ onto $\theta'$.

Recall that the $\GL_n$-orbits on $\cal{N}$ are indexed by the integer partitions of $n$ via Jordan type.
Let $\cal{O}_\lambda \subseteq \cal{N}$ be the orbit indexed by $\lambda \vdash n$.

For each integer composition $\nu$ of $n$, let $\cal{B}_\nu$ be the flag variety of parabolic type $\nu$, whose $\bb{C}$-points parametrize partial flags of type $\nu$ on $\bb{C}^n$.
Let 
\begin{align}
	\widetilde{\cal{N}}_\nu = \{(\theta, F) \in \cal{N} \times \cal{B}_\nu \mid \text{$F$ is $\theta$-stable}\}.
\end{align}
The $A$-points of $[\widetilde{\cal{N}}_\nu/{\GL_n}]$ form the groupoid of tuples $(V, \theta, F)$, where $(V, \theta) \in [\cal{N}/{\GL_n}](A)$ and $F$ is an $\nu$-stable partial flag of type $\nu$ on $V$ in the sense of \S\ref{subsec:partial-flag}.
Let $\pi = \pi_\nu : [\widetilde{\cal{N}}_\nu/{\GL_n}] \to [\cal{N}/{\GL_n}]$ be the forgetful map.
If $\lambda$ is the underlying partition of $\nu$, and $\lambda^t$ is the transpose of $\lambda$, then the image of $\pi_\nu$ is $[\overline{\cal{O}}_{\lambda^t}/{\GL_n}]$, the stack quotient of the orbit closure $\overline{\cal{O}}_{\lambda^t}$.
In particular, $\cal{B}_{(1^n)}$ is the full flag variety and $\pi_{(1^n)}$ is a stacky version of the Springer resolution of $\cal{N}$.

Let $X$ be any stack over $\bb{C}$ and $p : X \to [\cal{N}/{\GL_n}]$ a morphism.
For each $\nu$, let $X_\nu$, $\pi_X = \pi_{X, \nu}$, and $p_\nu$ be defined by the cartesian square:
\begin{equation}\label{eq:springer-pullback}
	\begin{tikzpicture}[baseline=(current bounding box.center), >=stealth]
		\matrix(m)[matrix of math nodes, row sep=2.5em, column sep=3em, text height=2ex, text depth=0.5ex]
		{ 		
			X_\nu
			&{[\widetilde{\cal{N}}_\nu/{\GL_n}]}\\	
			X
			&{[\cal{N}/{\GL_n}]}\\
		};
		\path[->, font=\scriptsize, auto]
		(m-1-1) edge node{$p_\nu$} (m-1-2)
		(m-1-1) edge node[left]{$\pi_X$} (m-2-1)
		(m-1-2) edge node{$\pi$} (m-2-2)
		(m-2-1) edge node{$p$} (m-2-2);
	\end{tikzpicture}
\end{equation}
In particular, taking $X = \CptPic$ and $p(M) = (\bar{M}, y)$ yields $X_\nu = \CptPic_\nu$.
Analogous statements hold for $\CptPic(\gap)$, $\Quot^\ell(E)$, and $\cal{D}_\nu$, as well as the quotients $\CptPic/\Lattice$, $\CptPic(\gap)/\Lattice$ once we observe that the map $p$ for $X = \CptPic$ is invariant under $\Lattice$.

\subsection{}

Now suppose that $X$ is a scheme of finite type.
In this case we write $\ur{H}_c^\ast(X)$ to denote the compactly-supported cohomology of $X$ with complex coefficients, and $\sf{W}_{\leq \ast}$ to denote its weight filtration.
The \dfemph{virtual weight polynomial} of $X$ is
\begin{align}\label{eq:virtual-weight}
	\chi(X, \sf{t}) = \sum_{i, j} {(-1)^i} \sf{t}^j \dim \gr_j^\sf{W} \ur{H}_c^i(X)
\end{align}
by definition.

For any finite group $G$, we write $K_0(G)$ to denote its representation ring.
When there is a weight-preserving action of $G$ on $\ur{H}_c^\ast(X)$, we may regard $\chi(X, \sf{t})$ as an element of $\bb{Z}[\sf{t}] \otimes K_0(G)$.

Let $\bb{K}$ be a field.
As in the introduction, let $\Lambda_\bb{K}^n = \Lambda_\bb{K}^n[\vec{Y}]$ be the vector space of degree-$n$ symmetric functions in a family of variables $\vec{Y} = (Y_i)_{i = 1}^\infty$ over $\bb{K}$.
Let $\{s_\lambda\}_{\lambda \vdash n}$, \emph{resp.}\ $\{h_\mu\}_{\mu \vdash n}$, be the basis of $\Lambda_\bb{K}^n$ of Schur functions, \emph{resp.}\ complete homogeneous symmetric functions \cite{macdonald}.
Let $\langle -, -\rangle$ be the $\bb{K}$-linear \dfemph{Hall inner product} on $\Lambda_\bb{K}^n$ defined by orthonormality of the Schur functions.
When $\bb{K} \supseteq \bb{Q}$, there is a $\bb{K}$-linear isomorphism
\begin{align}
	\cal{F} : \bb{K} \otimes K_0(S_n) \xrightarrow{\sim} \Lambda_\bb{K}^n,
\end{align}
known as the \dfemph{Frobenius character}.
It sends the irreducible character of $S_n$ indexed by $\lambda$ to the Schur function $s_\lambda$, and the character of the induced representation $\text{Ind}_{S_\nu}^{S_n}(1)$ to the complete homogeneous function $h_\mu$, where $S_\mu \subseteq S_n$ is the Young subgroup of type $\mu$.

\begin{prop}\label{prop:springer}
	If $X$ is of finite type, then there is a weight-preserving action of $S_n$ on $X_{(1^n)}$ such that $\ur{H}_c^\ast(X_\nu) = \ur{H}_c^\ast(X_{(1^n)})^{S_\nu}$ for all $\nu$.
	In particular, 
	\begin{align}
		\chi(X_\nu, \sf{t}) = \langle h_\mu, \cal{F}\chi(X_{(1^n)}, \sf{t})\rangle,
	\end{align}
	where $\mu$ is the integer partition obtained by sorting $\nu$.
	Moreover, as we run over $\nu$, these identities uniquely determine $\chi(X_{(1^n)}, \sf{t})$ as an element of $\bb{K} \otimes K_0(S_n)$.
\end{prop}

In what follows, we freely use functors between bounded derived categories of mixed complexes of sheaves with constructible cohomology, where ``mixed'' means we either use mixed Hodge modules, or spread out and reduce to a finite field to use mixed complexes of $\ell$-adic sheaves, fixing an isomorphism $\QL \simeq \bb{C}$.

\begin{proof}
	For all $\nu$, let $\cal{S}_\nu = \pi_{\nu, \ast}\bb{C}$ and $\cal{S}_{X, \nu} = \pi_{X, \nu, \ast} \bb{C}$.
	
	Note that $\cal{S}_{(1^n)}$ is the $\GL_n$-equivariant Springer sheaf.
	By the work of Lusztig \emph{et al.} \cite[Ch.\ 4]{trinh}, the underived endomorphism ring $\End(\cal{S}_{(1^n)})$ is pure of weight $0$ and isomorphic to $\bb{C}S_n$.
	This defines an $S_n$-action on $\cal{S}_{(1^n)}$.
	Since \eqref{eq:springer-pullback} is a cartesian square, base change lifts this action to $\cal{S}_{X, (1^n)} \simeq p^\ast \cal{S}_{(1^n)}$.
	Taking hypercohomology, we get an action of $S_n$ on $\ur{H}_c^\ast(X_{(1^n)})$.
	Since $\End(\cal{S}_{(1^n)})$ is concentrated in weight zero, the last action preserves weights.
	
	For general $\nu$, we have $\cal{S}_\nu \simeq \cal{S}_{(1^n)}^{S_\nu}$ by \cite[\S{2.7}]{bm}.
	So again by the cartesian square \eqref{eq:springer-pullback}, $\cal{S}_{X, \nu} \simeq \cal{S}_{X, (1^n)}^{S_\nu}$.
	Therefore
	\begin{align}
		\ur{H}_c^\ast(X, \cal{S}_{X, \nu}) \simeq \ur{H}_c^\ast(X, \cal{S}_{X, (1^n)}^{S_\nu}) \simeq \ur{H}_c^\ast(X, \cal{S}_{X, (1^n)})^{S_\nu},
	\end{align}
	where the second step uses the fact that the inclusion $\cal{S}_{X, (1^n)}^{S_\nu} \subseteq \cal{S}_{X, (1^n)}$ is split (say, via the isotypic decomposition of $\cal{S}_{X, (1^n)}$).
	Above, the first expression is $\ur{H}_c^\ast(X_\nu)$ and the last expression is $\ur{H}_c^\ast(X_{(1^n)})^{S_\nu}$.
	
	The statements about the Hall inner product and uniqueness now follow from Frobenius reciprocity and the fact that the $h_\mu$ span $\Lambda^n_\bb{K}$.
\end{proof}

\begin{rem}\label{rem:borel-moore}
	We write $\ur{H}_\ast^\BM(X)$ to denote the \dfemph{Borel--Moore homology} of $X$ with complex coefficients, defined via the hypercohomology of the dualizing sheaf on $X$.
	Verdier duality implies that $\ur{H}_c^i(X)$ and $\ur{H}_{-i}^\BM(X)$ are dual vector spaces for all $i$.
	Therefore, \Cref{prop:springer} also implies that $\ur{H}_\ast^\BM(X_\nu) = \ur{H}_\ast^\BM(X_{(1^n)})_{S_\nu}$ for all $\nu$, where $(-)_G$ denotes the coinvariants of a $G$-action.
\end{rem}

\subsection{}

For each integer $r \geq 0$, let $\cal{N}_{\len{r}} \subseteq \cal{N}$ be the union of the orbits indexed by partitions of length $r$, \emph{i.e.}, the subvariety of nilpotent matrices $\theta$ such that $\dim \ker(\theta) = r$.
Let $X_{\len{r}} = p^{-1}(\cal{N}_{\len{r}}) \subseteq X$.

As in the introduction, let $\Psi(\sf{a}, -) : \Lambda_\bb{K}^n \to \bb{K}[\sf{a}]$ be the specialization map 
\begin{align}
	\Psi(\sf{a}, -)
	=
	(1 + \sf{a})\sum_{0 \leq k \leq n - 1}
	{\sf{a}^k} \langle s_{(n - k, 1^k)}, -\rangle.
\end{align}
This map also appears in \cite[Ex.\ 4]{haglund_16} and \cite[Cor.\@ 1]{wilson}.
The next statement is a reformulation of \cite[Lem.\@ 9.3--9.4]{gors}:

\begin{lem}\label{lem:len}
	We have 
	\begin{align}
		\Psi(\sf{a}, \cal{F}\chi(X_{(1^n)}, \sf{t}))
		= \sum_{0 \leq r \leq n}
		\chi(X_{\len{r}}, \sf{t}) 
		\prod_{0 \leq j \leq r - 1}
		{(1 + \sf{a} \sf{t}^{2j})}.
	\end{align}
\end{lem}

\subsection{}

For each integer $m \geq 0$, let $P_{n - m, m} \subseteq \GL_n$ be a parabolic subgroup whose Levi quotient is isomorphic to $\GL_{n - m} \times \GL_m$: for instance, the appropriate subgroup of block-upper-triangular matrices.
Let $X_{\nest{m}}$ and $\rho_X = \rho_{X, m}$ be defined by the cartesian square:
\begin{equation}
	\begin{tikzpicture}[baseline=(current bounding box.center), >=stealth]
		\matrix(m)[matrix of math nodes, row sep=2.5em, column sep=1em, text height=2ex, text depth=0.5ex]
		{ 		
			X_{\nest{m}}
			&
			&{[\point/{P_{n - m, m}}]}\\	
			X
			&{[\cal{N}/{\GL_n}]}
			&{[\point/{\GL_n}]}\\
		};
		\path[->, font=\scriptsize, auto]
		(m-1-1) edge (m-1-3)
		(m-1-1) edge node[left]{$\rho_X$} (m-2-1)
		(m-1-3) edge (m-2-3)
		(m-2-1) edge node{$p$} (m-2-2)
		(m-2-2) edge (m-2-3);
	\end{tikzpicture}
\end{equation}
For instance, if $X = [\cal{N}/{\GL_n}]$, then $X_{\nest{m}}$ is the stack whose $A$-points form the groupoid of tuples $(V, \theta, V')$, where $(V, \theta) \in [\cal{N}/{\GL_n}](A)$ and $V'$ is an $A$-submodule of $\ker(\theta)$ such that $\ker(\theta)/V'$ is locally free over $A$ of rank $m$.

For all $r$, the map $\rho_{X, m}^{-1}(X_{\len{r}}) \to X_{\len{r}}$ is a locally trivial fibration whose fiber is the Grassmannian of codimension-$m$ subspaces of $\bb{C}^r$.
The virtual weight polynomials of Grassmannians can be computed via their Schubert stratifications, which show them to be $q$-binomial coefficients for $q = \sf{t}^2$.
So, generalizing \cite{ors, gors}, we deduce:

\begin{lem}\label{lem:nest}
	We have 
	\begin{align}
		\sum_{0 \leq m \leq n} \sf{a}^m \sf{t}^{m(m - 1)} \chi(X_{\nest{m}}, \sf{t})
		= \sum_{0 \leq r \leq n}
		\chi(X_{\len{r}}, \sf{t}) 
		\prod_{0 \leq j \leq r - 1}
		{(1 + \sf{a} \sf{t}^{2j})}.
	\end{align}
\end{lem}

\subsection{}

In the rest of the paper, we set
\begin{align}
	\cal{F}\sf{Pic}(\sf{q}, \sf{t})
	= \sum_\gap \sf{q}^\gap \cal{F}\chi(\CptPic_{(1^n)}(\gap), \sf{t}).
\end{align}
For any finitely-generated $\RRR$-module $E\subseteq K$, we set
\begin{align}
	\cal{F}\sf{Quot}_E(\sf{q}, \sf{t})
	= \sum_\ell \sf{q}^\ell \cal{F}\chi(\Quot_{(1^n)}^\ell(E), \sf{t}).
\end{align}
The symmetric functions $\cal{F}\sf{Hilb}, \cal{F}\sf{Quot}$ from the introduction are now given by $\cal{F}\sf{Hilb}(\sf{q}, \sf{t}) = \cal{F}\sf{Quot}_\RRR(\sf{q}, \sf{t})$ and $\cal{F}\sf{Quot}(\sf{q}, \sf{t}) = \cal{F}\sf{Quot}_\SSS(\sf{q}, \sf{t})$.
\Cref{conj:main-flag} can be rewritten as the single identity
\begin{align}\label{eq:conj-main-flag}
	\cal{F}\sf{Hilb}(\sf{q}, \sf{t}) \stackrel{?}{=} \cal{F}\sf{Quot}(\sf{q}, \sf{q}^{\frac{1}{2}}\sf{t}).
\end{align}
The virtual weight specialization of \Cref{thm:main} can be rewritten as
\begin{align}\label{eq:thm-main}
	\cal{F}\sf{Quot}(\sf{q}, \sf{t})
	= \frac{1}{(1 - \sf{q})^b}
	\cal{F}\sf{Pic}(\sf{q}, \sf{t}).
\end{align}
Finally, for any finitely-generated $\RRR$-module $E\subseteq K$, we spell out the meaning of $X_{\len{r}}$ and $X_{\nest{m}}$ when $X_\nu = \Quot_\nu^\ell(E)$:
\begin{itemize}
	\item 	$X_{\len{r}}$ is the locally-closed subscheme of $X = \Quot^\ell(E)$ whose $A$-points are those $M \in X(A)$ such that $M/(xM + yM) \simeq \bar{M}/y\bar{M} \simeq \ker(y \mid \bar{M})$ is locally free over $A$ of rank $r$.
	
	\item 	$X_{\nest{m}}$ is the scheme of finite type whose $A$-points parametrize pairs $(M, N)$, where $M \in \Quot^\ell(E)(A)$ and $N \in \Quot^{\ell + m}(E)(A)$ and $xM + yM \subseteq N \subseteq M$.
	Note that these containments are together equivalent to requiring that $N/(xN + yN)$ be a submodule of $M/(xM + yM)$.
	
\end{itemize}
We henceforth write $\Quot_{\len{r}}^\ell(E)$ and $\Quot_{\nest{m}}^\ell(E)$ in place of $\Quot^\ell(E)_{\len{r}}$ and $\Quot^\ell(E)_{\nest{m}}$, respectively.
\Crefrange{lem:len}{lem:nest} imply:

\begin{cor}\label{cor:nest}
	For any finitely-generated $\RRR$-module $E\subseteq K$, we have
	\begin{align}
		\Psi(\sf{a}, \cal{F}\sf{Quot}_E(\sf{q}, \sf{t}))
		= \sum_{\ell, m} \sf{q}^\ell \sf{a}^m \sf{t}^{m(m - 1)}
		\chi(\Quot_{\nest{m}}^\ell(E), \sf{t}).
	\end{align}
\end{cor}

We set $\cal{H}_{\nest{m}}^\ell = \Quot_{\nest{m}}^\ell(\RRR)$ as in the introduction, and similarly, $\cal{Q}_{\nest{m}}^\ell = \Quot_{\nest{m}}^\ell(\SSS)$.
In \cite{os, ors}, the $\cal{H}_{\nest{m}}^\ell$ are called \dfemph{nested Hilbert schemes}.

\section{Torus Knots}\label{sec:coprime}

\subsection{}

In this section, we give two independent proofs of case (1) of \Cref{thm:ors-quot}, stating in the notation of \S\ref{subsec:intro-toric} that $\bar{\sf{X}}_{n, d}(\sf{a}, \sf{q}, \sf{t}^2) = \Psi(\sf{a}, \cal{F}\sf{Quot}_{n, d}(\sf{q}, \sf{t}))$ for any $d > 0$ coprime to $n$.
The relationship between our proofs is summarized below:
\begin{equation}\label{eq:bridges}
	\begin{tikzpicture}[baseline=(current bounding box.center), >=stealth]
		\matrix(m)[matrix of math nodes, row sep=3em, column sep=3.75em, text height=2ex, text depth=0.5ex]
		{ 		
			\text{EHA}
			&\text{Hikita}
			&\cal{F}\sf{Pic}_{n, d}(\sf{q}, \sf{t})
			&\cal{F}\sf{Quot}_{n, d}(\sf{q}, \sf{t})\\	
			{}
			&\text{Cogen}
			&\bar{\sf{X}}_{n, d}(\sf{a}, \sf{q}, \sf{t}^2)
			&\text{Gen}\\
		};
		\path[->, font=\tiny, auto]
		(m-1-1) edge node[left,fill=white]{\cite{mellit_22, wilson}} (m-2-2);
		\path[<->, font=\tiny, auto]
		(m-1-1) edge node{\cite{mellit_21}} (m-1-2)
		(m-1-2) edge node{\cite{hikita} + \S\ref{subsec:proof-b}} (m-1-3)
		(m-2-2) edge node{\cite{hm}} (m-2-3)
		(m-2-3) edge node{\cite{mellit_22}} (m-2-4);
		\path[->, font=\tiny, auto, densely dotted]
		(m-1-4) edge node{$\begin{array}{l} \text{Cor \ref{cor:nest}} + {}\\ \text{Prop \ref{prop:gen}}\end{array}$} (m-2-4);
		\path[<->, font=\tiny, auto, densely dotted]
		(m-1-3) edge node{Thm \ref{thm:main}} (m-1-4);
	\end{tikzpicture}
\end{equation}
The horizontal arrows indicate identities; the vertical arrows are specializations.
The dotted arrows are new bridges.
Our first proof, labeled (A) in the introduction, follows the lower/right path from $\cal{F}\sf{Quot}_{n, d}$ to $\bar{\sf{X}}_{n, d}$.
Our second proof, labeled (B), follows the upper/left path.

\subsection{}

Let $\RRR = \bb{C}[\![x, y]\!]/(y^n - x^d)$.
Setting $x = \varpi^n$ and $y = \varpi^d$ gives $\RRR = \bb{C}[\![\varpi^n, \varpi^d]\!]$ and $\SSS = \bb{C}[\![\varpi]\!]$ and $K = \bb{C}(\!(\varpi)\!)$, generalizing \Cref{ex:cusp}. 

Note that the delta invariant of $\RRR$ is $\delta = \frac{1}{2}(n - 1)(d - 1)$ by a classical formula of Sylvester.
The number of branches is $b = 1$, so the link of the singularity has one component: \emph{i.e.}, it is a knot.

\subsection{}

Let $\bb{G}_m$ act on $\Spec(\RRR)$ according to $t \cdot (x, y) = (t^n x, t^d y)$, and on $\Spec(K)$ according to $t \cdot \varpi = t\varpi$.
These actions are compatible.
In particular, they induce a $\bb{G}_m$-action on $\CptPic_\unred$:
If $A$ is a $\bb{C}$-algebra and $t \in A^\times = \bb{G}_m(A)$ and $M \subseteq A \hatotimes K$ is an $(A \hatotimes \RRR)$-module corresponding to an $A$-point of $\CptPic_\unred$, then we define $t \cdot M$ to be the rescaling $tM$.
This action restricts to $\CptPic$.

Let $E$ be a finitely-generated $\RRR$-submodule of $K$ fixed by $\bb{C}^\times = \bb{G}_m(\bb{C})$.
The $\bb{G}_m$-action on $\CptPic$ restricts to $\Quot^\ell(E)$ for all $\ell$.
We use this action to skeletonize $\Quot^\ell(E)$ into combinatorics.
Let
\begin{align}
	\Gamma(E)
	&= \{\val_\varpi(s) \mid s \in E \setminus \{0\}\},\\
	I(E)
	&= \{\Delta \subseteq \Gamma(E) \mid \Delta + n, \Delta + d \subseteq \Delta\},\\
	I^\ell(E)
	&= \{\Delta \in I(E) \mid |\Gamma(E) \setminus \Delta| = \ell\}.
\end{align}
Note that $\Gamma(\RRR) = n\bb{Z}_{\geq 0} + d\bb{Z}_{\geq 0}$ and $\Gamma(\SSS) = \bb{Z}_{\geq 0}$.

\begin{rem}
	In general, additive submonoids of $\bb{Z}_{\geq 0}$ are also known as \dfemph{numerical semigroups}.
	A subset of $\bb{Z}$ stable under addition with a numerical semigroup $\Gamma$ is also known a \dfemph{$\Gamma$-module}.
	Thus $\Gamma(\RRR)$ is a numerical semigroup, $\Gamma(E)$ is a $\Gamma(\RRR)$-module, and $I(E)$ is the set of $\Gamma(\RRR)$-submodules of $\Gamma(E)$.
\end{rem}

For all $\Delta \in I(E)$, let
\begin{align}
	\Gen_n(\Delta)
	&= \{k \in \Delta \mid k - n \notin \Delta\},\\
	\Gen(\Delta)
	&= \{k \in \Delta \mid k - n, k - d \notin \Delta\}.
\end{align}
The elements of $\Gen_n(\Delta)$, \emph{resp.}\ $\Gen(\Delta)$, are called the \dfemph{$n$-generators} \cite{gmv}, \emph{resp.}\ \dfemph{generators}, of $\Delta$.
The following lemma can be proved by arguments completely analogous to those of \cite[\S{3}]{piontkowski}, by taking $\Quot^\ell(E)$ in place of $\CptJac$.

\begin{lem}\label{lem:paving}
	In the setup above, the $\bb{G}_m$-action on $\Quot^\ell(E)$ has isolated fixed points.
	We have a bijection from $I^\ell(E)$ to the set of fixed $\bb{C}$-points, given by
	\begin{align}
		\begin{array}{rcl}
			I^\ell(E)
			&\xrightarrow{\sim}
			&\Quot^\ell(E)^{\bb{G}_m},\\
			\Delta 
			&\mapsto 
			&M_\Delta \vcentcolon= \RRR\langle t^k \mid k \in \Delta\rangle.
		\end{array}
	\end{align}
	Moreover, $\Quot^\ell(E)$ is partitioned by the subschemes
	\begin{align}
		\bb{A}_\Delta = \{M \in \Quot^\ell(E) \mid \textstyle\lim_{t \to 0} {(t \cdot M)} = M_\Delta\},
	\end{align}
	and each $\bb{A}_\Delta$ forms an affine space.
\end{lem}

Although the partition above is similar to a Bia{\l}ynicki--Birula decomposition, it does not follow from said theorem when $\Quot^\ell(E)$ is singular.

\subsection{}

Recall the nested Quot schemes $\Quot_{\nest{m}}^\ell(E)$ that we reviewed at the end of \Cref{sec:springer}.
The diagonal $\bb{G}_m$-action on $\Quot^\ell(E) \times \Quot^{\ell + m}(E)$ restricts to an action on $\Quot_{\nest{m}}^\ell(E)$.
Let 
\begin{align}
	I_{\nest{m}}^\ell(E)
	= \{(\Delta, \Delta') \in I^\ell(E) \times I^{\ell + m}(E) \mid \Delta \supseteq \Delta' \supseteq \Delta + \Gamma_{E, > 0}\}.
\end{align}
The following lemma is proved in \cite[\S{3.3}]{ors} for $E = \RRR$, and the proof for any other $E\subseteq K$ is analogous.

\begin{lem}\label{lem:paving-nest}
	The $\bb{G}_m$-action on $\Quot_{\nest{m}}^\ell(E)$ has isolated fixed points.
	Writing $\Gamma_{E, > 0} = \Gamma(E) \setminus \{0\}$, we have a bijection
	\begin{align}
		\begin{array}{rcl}
			I_{\nest{m}}^\ell(E)
			&\xrightarrow{\sim}
			&\Quot_{\nest{m}}^\ell(E)^{\bb{G}_m},\\
			(\Delta, \Delta')
			&\mapsto 
			&(M_\Delta, M_{\Delta'}).
		\end{array}
	\end{align}
	Moreover, $\Quot_{\nest{m}}^\ell(E)$ is partitioned by the subschemes
	\begin{align}
		\bb{A}_{\Delta, \Delta'} = \{(M, M') \in \Quot_{\nest{m}}^\ell(E) \mid \textstyle\lim_{t \to 0} (M, M') = (M_\Delta, M_{\Delta'}\},
	\end{align}
	and each $\bb{A}_{\Delta, \Delta'}$ forms an affine space.
\end{lem}

\subsection{}

Given $\Delta \in I(E)$, let 
\begin{align}
	\sf{dim}_\Delta
	&= \dim(\bb{A}_\Delta),\\
	\xi_n(\Delta, k)
	&= \{j \in \Gen_n(\Delta) \mid k - d < j < k\}
	&&\text{for all $k \in \Gen(\Delta)$},\\
	\Pi_{\Delta}^\Gen(\sf{a}, \sf{t})
	&= \prod_{k \in \Gen(\Delta)}
	{(1 + \sf{a}\sf{t}^{|\xi_n(\Delta, k)|})}.
\end{align}
For $E = \RRR$, the following proposition is \cite[Cor.\@ A.5]{ors}.
To translate into the notation of \cite[\S{A}.1]{ors}, note that our $\sf{a}, \sf{t}$ correspond to their $a^2 t, t$, and hence, our $|\xi_n(\Delta, k)|$ corresponds to their $\beta_k(\Delta) - 1$.
In the proof below, we merely list the changes needed to extend the proof to any $E \in \CptPic(\bb{C})$.

\begin{prop}\label{prop:gen}
	Let $\RRR = \bb{C}[\![\varpi^n, \varpi^d]\!]$ for coprime $n, d > 0$.
	For any finitely-generated $\RRR$-submodule $E \subseteq \bb{C}(\!(\varpi)\!)$ fixed by the $\bb{C}^\times$-action rescaling $\varpi$, we have
	\begin{align}
		\Psi(\sf{a}, \cal{F}\sf{Quot}_E(\sf{q}, \sf{t}))
		&= \sum_\ell \sf{q}^\ell
		\sum_{\Delta \in I^\ell(E)}
		\sf{t}^{2\sf{dim}_\Delta} \Pi_{\Delta}^\Gen(\sf{a}, \sf{t}^2)
	\end{align}
	in the notation of \Cref{sec:springer}.
\end{prop}

\begin{proof}
	By \Cref{cor:nest}, it suffices to show that for all $\ell \geq 0$, we have
	\begin{align}\label{eq:quot-to-gen}
		\sum_{\Delta \in I^\ell(E)}
		\sf{t}^{2\sf{dim}_\Delta} \Pi_{\Delta}^\Gen(\sf{a}, \sf{t}^2)
		=
		\sum_{0 \leq m \leq n}
		\sf{a}^m \sf{t}^{m(m - 1)}
		\chi(\Quot_{\nest{m}}^\ell(E), \sf{t}).
	\end{align}
	Theorems 13 and 14 of \cite{ors} give formulas for $\sf{dim}_\Delta$ and $\dim_{\Delta, \Delta'} \vcentcolon= \dim(\bb{A}_{\Delta, \Delta'})$ in the $E = \RRR$ case.
	For general $E$, analogous proofs give the formulas
	\begin{align}\label{eq:dim-delta}
		\sf{dim}_\Delta
		&= \sum_i {(|\Gamma(E)_{> \gamma_i} \setminus \Delta| - |\Gamma(E)_{> \sigma_i} \setminus \Delta|)},\\
		\label{eq:dim-delta-nest}
		\dim_{\Delta, \Delta'}
		&=
		\sum_{\substack{i \\ \gamma_i \notin \Delta'}}
		|\Gamma(E)_{> \gamma_i} \setminus \Delta|
		+ \sum_{\substack{i \\ \gamma_i \in \Delta'}}
		|\Gamma(E)_{> \gamma_i} \setminus \Delta'|
		- \sum_i
		|\Gamma(E)_{> \sigma_i} \setminus \Delta'|
	\end{align}
	for any $\Delta \in I^\ell(E)$ with generators $\gamma_1, \ldots \gamma_r$, syzygies $\sigma_1, \ldots, \sigma_r$, and subset $\Delta' \in I^{\ell + m}(E)$ such that $(\Delta, \Delta') \in I_{\nest{m}}^\ell(E)$, where $\Gamma(E)_{> k} = \Gamma(E) \cap \bb{Z}_{> k}$.
	
	Next, Lemma A.4 of \cite{ors} shows that in the $E = \RRR$ case, if $k \in \Gen(\Delta)$, then
	\begin{align}\label{eq:beta-gen-syz}
		|\xi_n(\Delta, k)| = |\{i \mid \gamma_i < k\}| - |\{i \mid \sigma_i < k\}|,
	\end{align}
	with the same notation for generators and syzygies as before.
	Then Lemma A.1 and Theorem A.2 of \cite{ors} show that for $E = \RRR$ and any fixed $\Delta$, formulas \eqref{eq:dim-delta}, \eqref{eq:dim-delta-nest}, and \eqref{eq:beta-gen-syz} together imply that 
	\begin{align}
		\sf{t}^{2\sf{dim}_\Delta}
		\Pi_{\Delta}^\Gen(\sf{a}, \sf{t}^2)
		=
		\sum_{0 \leq m \leq n}
		\sf{a}^m \sf{t}^{m(m - 1)}
		\sum_{\Delta' \:\mid\: (\Delta, \Delta') \in I_{\nest{m}}^\ell(E)}
		\sf{t}^{2\dim_{\Delta, \Delta'}}.
	\end{align}
	The proofs of these statements for general $E$ are the same.
	By \Cref{lem:paving-nest}, summing the last identity over all $\Delta \in I^\ell(E)$ recovers \eqref{eq:quot-to-gen}.
\end{proof}

\subsection{Proof (A) of Case (1) of \Cref{thm:ors-quot}}

For any integer $j \geq 0$, we observe that $\Delta \in I^\ell(\SSS)$ implies $\Delta + j \in I^{\ell + j}(\SSS)$, while $k \in \Gen(\Delta)$, \emph{resp.}\ $\Gen_n(\Delta)$, implies $k + j \in \Gen(\Delta + j)$, \emph{resp.}\ $\Gen_n(\Delta + j)$.
Consequently:
\begin{align}
	\dim_{\Delta + j} 
	&= \sf{dim}_\Delta,\\
	\xi_n(\Delta + j, k + j) 
	&= \xi_n(\Delta, k) + j
	&&\text{for all $k \in \Gen(\Delta)$},\\
	\Pi_{\Delta + j}^\Gen 
	&= \Pi_{\Delta}^\Gen.
\end{align}
Now consider this combinatorial version of the domain $\cal{D}$ in \Cref{sec:quot}:
\begin{align}
	D_{n, d} = \{\Delta \in I(\SSS) \mid \min(\Delta) = 0\}.
\end{align}
By the observations above, the formula for $\Psi(\cal{F}\sf{Quot}(\sf{q}, \sf{t}), \sf{a})$ in \Cref{prop:gen} equals
\begin{align}\label{eq:i-to-d}
	\frac{1}{1 - \sf{q}}
	\sum_{\Delta \in D_{n, d}}
	\sf{q}^{|\bb{Z}_{\geq 0} \setminus \Delta|} 
	\sf{t}^{2\sf{dim}_\Delta}
	\Pi_{\Delta}^\Gen(\sf{a}, \sf{t}^2).
\end{align}
It remains to match this formula with the formula for $\bar{\sf{X}}_{n, d}(\sf{a}, \sf{q}, \sf{t}^2)$ for coprime $n, d$ conjectured in \cite{gn} and proved in \cite{mellit_22}.
It will be convenient to replace $\sf{t}$ with $\sf{t}^{\frac{1}{2}}$ everywhere in what follows.

In \cite{gm}, Gorsky--Mazin gave a bijection from $D_{n, d}$ to the set of $n \times d$ rational Dyck paths, under which $|\bb{Z}_{\geq 0} \setminus \Delta|$ and $\sf{dim}_\Delta$ correspond to the statistics on Dyck paths respectively denoted $\sf{area}$ and $\sf{codinv}$ in \cite{gmv}.
Explicitly, form the semi-infinite grid of unit squares in the $x, y$-plane whose vertices are the lattice points with $0 \leq x \leq d$ and $y \geq 0$.
Label the bottom left square, closest to the origin, with
\begin{wrapfigure}{r}{0.25\textwidth}
	\centering
	\includegraphics[width=0.1875\textwidth]{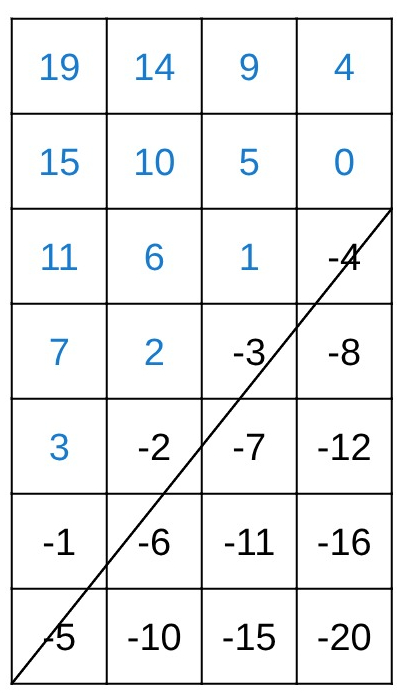}
\end{wrapfigure}
the integer $-d$; label the other squares with integers that \emph{decrease} by $d$ as we go across rows, and \emph{increase} by $n$ as we go up columns.
For instance, the grid for $(n, d) = (4, 5)$ is shown to the right, with nonnegative labels in blue.
For any $\Delta \in D_{n, d}$, the boundary of the region of squares with labels in $\Delta$ must contain a lattice path $\pi(\Delta)$ from $(x, y) = (0, 0)$ to $(x, y) = (d, n)$ that stays above the line $y = \frac{d}{n}x$, since $\Delta$ contains $0$ and every element of $\Delta$ is nonnegative.
Gorsky--Mazin's bijection sends $\Delta \mapsto \pi(\Delta)$.

\begin{rem}
	In \cite{gm}, the set $D_{n, d}$ is described as indexing the fixed points of a $\bb{G}_m$-action on $\CptJac$, rather than $\cal{D}$.
	However, this indexing really factors through the decomposition $\cal{D} = \coprod_\gap \varpi^{-\gap} \CptJac(\gap)$, corresponding to the fact that the elements of $D_{n, d}$ are what Gorsky--Mazin call \dfemph{0-normalized} modules for $\Gamma(\RRR)$.
	Compare to \S\ref{subsec:cherednik}, where a similar remark applies to Cherednik's notation.
\end{rem}

Let $\pi = \pi(\Delta)$ in what follows.
Let
\begin{align}
	v_\ast(\pi) 
	&= \{(x, y) \in \pi \mid (x - 1, y), (x, y + 1) \in \pi\},\\
	v^\ast(\pi)
	&= \{(x, y) \in \pi \mid (x + 1, y), (x, y - 1) \in \pi\}.
\end{align}
Following \cite{mellit_22}, we refer to elements of $v_\ast(\pi)$, \emph{resp.}\ $v^\ast(\pi)$, as \dfemph{inner vertices}, \emph{resp.}\ \dfemph{outer vertices}, of $\pi$.
That is, inner vertices are the bottom right corners of squares whose bottom and right edges are contained in $\pi$, while outer vertices are the top left corners of squares whose top and left edges are contained in $\pi$.
(Note that outer vertices are called ``internal vertices''(!) in \cite{gn}.) 

The squares whose bottom edges are contained in $\pi$ are precisely those labelled by elements of $\Gen_n(\Delta)$.
Of these, those whose right edges are also contained in $\pi$ are those labelled by elements of $\Gen(\Delta) \setminus \{0\}$.
Hence, there is a bijection
$
\Gen(\Delta) \setminus \{0\} \xrightarrow{\sim} v_\ast(\pi) 
$
sending any generator of $\Delta$ to the bottom right corner of the square it labels. 

To illustrate the previous two paragraphs following the figure in Example \ref{ex:7x5},  $v^\ast(\pi)$ contains the bottom left vertices of the squares labeled $9,11,3,$ and $5$, whereas $v_\ast(\pi)$ contains the bottom right vertices of the squares labeled $9, 1,$ and $3$ and the above bijection is the one given by this labeling.

For an arbitrary lattice point $p$, let $l_{d/n}(p)$ be the line of slope $\frac{d}{n}$ through $p$, and let $\kappa_\pi(p)$ be the set of horizontal unit steps of $\pi$ that intersect $l_{d/n}(p)$ in their interiors.
The following lemma is inspired by ideas from \cite{gm} and \cite[\S{A}]{ors}.

\begin{lem}\label{lem:kappa}
	If $k \in \Gen(\Delta) \setminus \{0\}$ labels a square with bottom right corner $p \in v_\ast(\pi)$, then the map $\xi_n(\Delta, k) \to \kappa_\pi(p)$ that sends $k$ to the bottom edge of the square labelled $k$ is a bijection.
	Thus
	\begin{align}
		\frac{1}{1 + \sf{a}} \Pi_{\Delta}^\Gen(\sf{a}, \sf{t})
		&= \prod_{p \in v_\ast(\pi)}
		{(1 + \sf{a}\sf{t}^{|\kappa_\pi(p)|})}.
	\end{align}
\end{lem}

\begin{proof}
	We observe that if $p$ is the bottom right corner of a square labelled $k$, then the line $l_{d/n}(p)$ intersects the bottom edge of a square labelled $j$ if and only if $k - d < j < k$.
	Indeed, this is easiest to see when $p = (n, d)$ and $k = 0$, and the general case follows from translating $l_{d/n}(n, d)$ onto $l_{d/n}(p)$.
\end{proof}

In our notation, the formula for $\bar{\sf{X}}_{n, d}$ for coprime $n, d$ in \cite{gn, mellit_22} is:
\begin{align}
	\bar{\sf{X}}_{n, d}(\sf{a}, \sf{q}, \sf{t})
	= 
	\frac{1}{1 - \sf{q}}
	\sum_{\substack{n \times d \\ \text{Dyck paths $\pi$}}}
	\sf{q}^{\sf{area}(\pi)}
	\sf{t}^{\sf{codinv}(\pi)}
	\prod_{p \in v^\ast(\pi)}
	{(1 + \sf{a}\sf{t}^{|\kappa_\pi(p)|})},
\end{align}
where, by \cite{gm}, $\sf{area}(\pi) = |\bb{Z}_{\geq 0} \setminus \Delta|$ and $\sf{codinv}(\pi) = \sf{dim}_\Delta$ whenever $\pi = \pi(\Delta)$.
See the end of \Cref{sec:conventions} for the precise matching of grading conventions.
So by \Cref{lem:kappa}, it remains to show:

\begin{lem}\label{lem:updown}
	For any $n \times d$ Dyck path $\pi$ as above,
	\begin{align}
		\prod_{p \in v_\ast(\pi)} {(1 + \sf{a}\sf{t}^{|\kappa_\pi(p)|})} = \frac{1}{1 + \sf{a}} \prod_{p \in v^\ast(\pi)} {(1 + \sf{a}\sf{t}^{|\kappa_\pi(p)|})}.
	\end{align}
\end{lem}

\begin{proof}
	Since $d$ and $n$ are coprime, no two elements of $v_\ast(\pi) \cup v^\ast(\pi)$ have the same perpendicular distance to the line $l \vcentcolon= l_{d/n}(n, d)$.
	The one farthest from $l$ must belong to $v^\ast(\pi)$.
	Let $p_0$ be this element, and let $p_1, p_2, \ldots, p_m$ be the remaining elements ordered by decreasing distance from $l$.
	For $1 \leq i \leq m$, let
	\begin{align}
		\epsilon_i = 
		\left\{\begin{array}{r@{}ll}
			-&1	&(p_{i - 1}, p_i) \in v_\ast(\pi) \times v_\ast(\pi),\\
			&0	&(p_{i - 1}, p_i) \in (v_\ast(\pi) \times v^\ast(\pi)) \cup (v^\ast(\pi) \times v_\ast(\pi)),\\
			&1	&(p_{i - 1}, p_i) \in v^\ast(\pi) \times v^\ast(\pi).
		\end{array}\right.
	\end{align}
	Let $\tau_i = \sum_{j \leq i} \epsilon_i$.
	Then for all $i$, we have $\tau_i = |\kappa_\pi(p_i)| \geq 0$.
	
	If $m = 0$, then we are done; else, we must have $\tau_1 = \tau_m = 1$.
	It follows that every value attained by the sequence $\tau_1, \ldots, \tau_m$ must occur as many times for indices $i$ with $p_i \in v_\ast(\pi)$ as for indices $i$ with $p_i \in v^\ast(\pi) \setminus \{p_0\}$.
\end{proof}

\begin{ex}
	\label{ex:7x5}
	The figure below shows a $7 \times 5$ Dyck path $\pi$ for which $|v_\ast(\pi)| = 3$ and $|v^\ast(\pi)| = 4$.
	The corresponding $\Delta \in D_{7, 5}$ yields $\Gen_7(\Delta) = \{0, 5, 3, 1, 6, 11, 9\}$ and $\Gen(\Delta) = \{0, 3, 1, 9\}$.
	\begin{equation}
		\centering
		\includegraphics[width=0.375\textwidth]{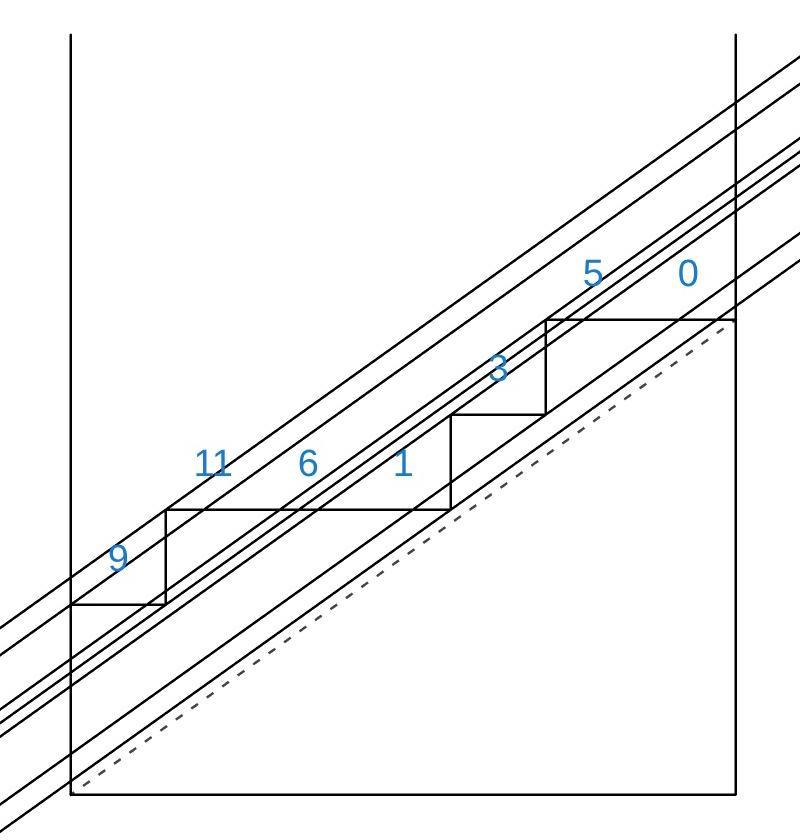}
	\end{equation}
	In the notation of \Cref{lem:updown}, $(\epsilon_i)_i = (1, 1, 0, 0, 0, -1)$ and $(\tau_i)_i = (1, 2, 2, 2, 2, 1)$.
\end{ex}

\begin{rem}
	\Cref{lem:updown} refines the last display on \cite[60]{mellit_22}, which merely asserts that $\sum_{p \in v_\ast(\pi)} {|\kappa_\pi(p)|} = \sum_{p \in v^\ast(\pi)} {|\kappa_\pi(p)|}$.
\end{rem}

\subsection{Proof (B) of Case (1) of \Cref{thm:ors-quot}}\label{subsec:proof-b}

We will explain each arrow in the left-hand portion of diagram \eqref{eq:bridges}.

First, we invoke \Cref{thm:main} to pass from $\cal{F}\sf{Quot}_{n,d}$ to $\cal{F}\sf{Pic}_{n, d}$.

Next, we explain the arrow labeled \cite{hikita}.
In \Cref{prop:asf-coprime}, we recall
the explicit isomorphism between $\CptPic_\nu$ and a parabolic affine Springer fiber for $\GL_n$ studied by Hikita \cite{hikita}, to be denoted $\hat{\cal{B}}_\nu^{\psi(d)}$. 
We refer to \Cref{sec:asf} for the notation.
Note that Hikita worked with $\SL_n$, not $\GL_n$, but we account for this difference by passing to $\CptPic^0 \simeq \CptPic/\Lattice$:
See part (3) of the proposition.

Both $\CptPic_\nu$ and $\hat{\cal{B}}_\nu^{\psi(d)}$ admit affine pavings analogous to those in \Crefrange{lem:paving}{lem:paving-nest}.
Hikita introduced an increasing filtration of $\hat{\cal{B}}_\nu^{\psi(d)}$ by unions of paving strata, which we review in \S\ref{subsec:hikita}.
Let $\sf{F}_{\leq \ast}$ be the induced filtration on the Borel--Moore homology.
Theorem 4.17 of \cite{hikita} matches the bigraded Frobenius characteristic of $\gr_\ast^\sf{F} \ur{H}_\ast^\BM(\hat{\cal{B}}_{(1^n)}^{\psi(d)})$ with a $q, t$-symmetric function defined combinatorially using labeled $n \times d$ rational Dyck paths.
This symmetric function is now known as the \dfemph{Hikita polynomial} for $(n, d)$.
It was independently introduced by Armstrong at the 2012 AMS Joint Mathematics Meetings \cite{armstrong}.

At the same time, there is a filtration of $\CptPic((1^n))/\Lattice$ by unions of the subvarieties $\CptPic(\gap)((1^n))/\Lattice$, which we review in \S\ref{subsec:gap}.
\Cref{thm:gap-vs-hikita} says that it corresponds to Hikita's filtration on $\hat{\cal{B}}_\nu^{\psi(d)}$, after postcomposing with an involution $\iota$. 
\Cref{lem:involution}(2) says that on Borel--Moore homology, $\iota$ is Springer-equivariant and preserves weights.
Due to the affine paving, the weight filtration matches the homological one.
We deduce that the Hikita polynomial for $(n, d)$ is unchanged by $\iota$, and matches 
$\cal{F}\sf{Pic}_{n, d}$ once we invoke the duality between Borel--Moore homology and compactly-supported cohomology.
Hikita's variables $t, q$ correspond to our variables $\sf{q}, \sf{t}^2$. 

To explain the arrow to ``EHA'' in the top left of \eqref{eq:bridges}:
The rational shuffle theorem for coprime $n$ and $d$, formulated by Gorsky--Negu\c{t} in \cite{gn} and proved by Mellit in \cite{mellit_21}, matches the Hikita polynomial with an expression denoted $\bb{Q}_{d, n} \cdot (-1)^n$ in \cite{bglx}.
Here, $\bb{Q}_{d, n}$ is an element of the elliptic Hall algebra (EHA), and $(-1)^n$ is a vector in the Fock-space representation of the EHA on symmetric functions.

To explain the last two arrows needed to arrive at $\bar{\sf{X}}_{n, d}(\sf{a}, \sf{q}, \sf{t}^2)$:
Mellit's proof implicitly yields a recursive formula for $\bb{Q}_{d, n} \cdot (-1)^n$, and hence $\cal{F}\sf{Pic}_{n, d}$, in terms of the Dyck-path operators from his prior work with Carlsson \cite{cm}.
This recursion is stated explicitly in \cite[Thms.\@ 2--3]{wilson}.
At the same time, in \cite{hm}, Hogancamp--Mellit establish a recursive formula for the KhR homology of the positive $(n, d)$ torus link, for arbitrary $n, d$.
As explained in \Cref{rem:cogen}, below, this yields a closed form for $\bar{\sf{X}}_{n, d}(\sf{a}, \sf{q}, \sf{t}^2)$ that we denote by ``Cogen'' in \eqref{eq:bridges}.
In \cite[Cor.\@ 1]{wilson}, Wilson shows that for $n, d$ coprime, Mellit's recursion for $\bb{Q}_{d, n} \cdot(-1)^n$ specializes under $\Psi$ to Hogancamp--Mellit's recursion for $\bar{\sf{X}}_{n, d}(\sf{a}, \sf{q}, \sf{t}^2)$.
We note that essentially the same result appears in \cite[Cor.\@ 3.4]{mellit_22}.
This completes proof (B).

\begin{rem}\label{rem:cogen}
	The closed form for $\bar{\sf{X}}_{n, d}(\sf{a}, \sf{q}, \sf{t}^2)$ resulting from \cite{hm} is due to Gorsky--Mazin--Vazirani \cite{gmv}.
	It is labeled ``Cogen'' in diagram \eqref{eq:bridges} because it uses the same set of semigroup modules $D_{n, d}$ as in proof (A), but replaces $\Pi_{\Delta}^\Gen(\sf{a}, \sf{t})$ with
	$\Pi_{\Delta}^\Cogen(\sf{a}\sf{q}^{-1}, \sf{t})$, where
	\begin{align}
		\Pi_{\Delta}^\Cogen(\sf{b}, \sf{t})
		= \prod_{k \in \Cogen(\Delta)}
		{(1 + \sf{b}\sf{t}^{\lambda(\Delta, k)})},
	\end{align}
	where the product runs over the set of \dfemph{(nonnegative) cogenerators}
	\begin{align}
		\Cogen(\Delta) = \{k \in \bb{Z}_{\geq 0} \setminus \Delta \mid k + n, k + d \in \Delta\},
	\end{align}
	and for any $k \in \Cogen(\Delta)$, we set
	\begin{align}
		\lambda(\Delta, k) 
		&= |\{j \in \Gen_n(\Delta) \mid k + n  + 1 \leq j \leq k + n + d\}|\\
		&= |\{j \in \Gen_n(\Delta) \mid k + n < j < k + n + d\}|.
	\end{align}
\end{rem}

\begin{rem}
	It is natural to ask how much of diagram \eqref{eq:bridges} generalizes to integers $n, d$ that are not coprime.
	We will address this question in a sequel paper.
	In \Cref{sec:noncoprime}, where we address the $d = nk$ case, our proof does \emph{not} involve generalizing \eqref{eq:bridges}.
	For now, we mention that:
	\begin{itemize}
		\item 	The rational shuffle conjecture was generalized to arbitrary $n, d > 0$ in \cite{bglx}.
		This is the actual result proved by Mellit in \cite{mellit_21}.
		
		\item 	\Cref{thm:main} and the Cogen formula for $\bar{\sf{X}}_{n, d}$ extend to arbitrary $n, d$.
		
		\item 	In \cite{wilson}, Wilson introduces generalizations of $\bb{Q}_{d, n} \cdot (-1)^n$ and the Hikita polynomial to arbitrary $n, d$, which differ from those in \cite{bglx}.
		He has nonetheless shown that his Hikita polynomial specializes to the Cogen formula in (ii), and hence, to $\bar{\sf{X}}_{n, d}$.
		
	\end{itemize}
\end{rem}

\subsection{Gen versus Cogen}\label{subsec:gen-vs-cogen}

This subsection is a digression on \Cref{rem:cogen}.
As mentioned, the identity matching the Gen and Cogen formulas is
\begin{align}\label{eq:gen-vs-cogen}
	\frac{1}{1 + \sf{a}} \sum_{\Delta \in D_{n, d}}
	\sf{q}^{|\bb{Z}_{\geq 0} \setminus \Delta|} 
	\sf{t}^{\sf{dim}_\Delta}
	\Pi_{\Delta}^\Gen(\sf{a}, \sf{t})
	&=
	\sum_{\Delta \in D_{n, d}}
	\sf{q}^{|\bb{Z}_{\geq 0} \setminus \Delta|}
	\sf{t}^{\sf{dim}_\Delta}
	\Pi_{\Delta}^\Cogen(\sf{a}\sf{q}^{-1}, \sf{t}).
\end{align}
It is remarkable because $\Gen$ and $\Cogen$ behave very differently.
Note that at $\sf{a} \to 0$, the terms $\Pi_n^\Gen, \Pi_n^\Cogen$ disappear above, and both sides specialize to
\begin{align}\label{eq:catalan}
	\sum_{\Delta \in D_{n, d}}
	\sf{q}^{|\bb{Z}_{\geq 0} \setminus \Delta|} 
	\sf{t}^{\sf{dim}_\Delta}.
\end{align}
Similarly, our proofs of case (1) of \Cref{thm:ors-quot} simplify drastically in the $\sf{a} \to 0$ limit; almost all of their combinatorial complexity lies in the higher $\sf{a}$-degrees.

\begin{rem}\label{rem:catalan}
	Let $C_{n, d}(q, t) = C_{n, d}(t, q)$ be the $q, t$-rational Catalan number introduced in \cite{haglund_08}.
	Via their bijection from $D_{n, d}$ to the set of $n \times d$ Dyck paths, Gorsky--Mazin showed that \eqref{eq:catalan} is $\sf{t}^\delta C_{n, d}(\sf{q}, \sf{t}^{-1})$ \cite{gm}.
\end{rem}

Below, we illustrate the contrast between $\Gen$ and $\Cogen$ in an example where $d = n + 1$.
Throughout, we label the elements of $D_{n, d}$ in the form $\Delta_{a_1, \ldots, a_n}$, where $\Gen_n = \{a_1, \ldots, a_n\}$ and $a_i + \delta - |\bb{Z}_{\geq 0} \setminus \Delta| \equiv i - 1 \pmod{n}$ for all $i$, to streamline comparison to \Cref{ex:min-delta}.

\begin{ex}\label{ex:3-4}
	Take $(n, d) = (3, 4)$.
	Then $\delta = 3$ and
	\begin{align}
		D_{3, 4} = \{\Delta_{0,4,8}, \Delta_{5,0,4}, \Delta_{1,5,0}, \Delta_{4,2,0}, \Delta_{0,1,2}\}
	\end{align}
	with these statistics:
	\begin{align}
		\begin{array}{llllll}
			\Delta
			&\sf{q}^{|\bb{Z}_{\geq 0} \setminus \Delta|}
			\sf{t}^{\sf{dim}_\Delta}
			&\Gen \setminus \{0\}
			&\frac{1}{1 + \sf{a}} \Pi_{\Delta}^\Gen
			&\Cogen
			&\Pi_{\Delta}^\Cogen\\
			\hline
			\Delta_{0,4,8}
			&\sf{q}^3 \sf{t}^3
			&\emptyset
			&1
			&\{5\}
			&1 + \sf{b}\\
			\Delta_{5,0,4}
			&\sf{q}^2 \sf{t}^2
			&\{5\}
			&1 + \sf{a}\sf{t}
			&\{1, 2\}
			&(1 + \sf{b})(1 + \sf{b}\sf{t})\\
			\Delta_{1,5,0}
			&\sf{q} \sf{t}^2
			&\{1\}
			&1 + \sf{a}\sf{t}
			&\{2\}
			&1 + \sf{b}\\
			\Delta_{4,2,0}
			&\sf{q} \sf{t}
			&\{2\}
			&1 + \sf{a}\sf{t}
			&\{1\}
			&1 + \sf{b}\\
			\Delta_{0,1,2}
			&1
			&\{1, 2\}
			&(1 + \sf{a}\sf{t})(1 + \sf{a}\sf{t}^2)
			&\emptyset
			&1
		\end{array}
	\end{align}
	Here, \eqref{eq:gen-vs-cogen} becomes
	\begin{align}
		&\sf{q}^3\sf{t}^3 + (\sf{q}^2\sf{t}^2 + \sf{q}\sf{t}^2 + \sf{q}\sf{t})(1 + \sf{a}\sf{t}) + 1(1 + \sf{a}\sf{t})(1 + \sf{a}\sf{t}^2)\\
		&\qquad= (\sf{q}^3\sf{t}^3 + \sf{q}\sf{t}^2 + \sf{q}\sf{t})(1 + \sf{a}\sf{q}^{-1}) + \sf{q}^2\sf{t}^2(1 + \sf{a}\sf{q}^{-1})(1 + \sf{a}\sf{q}^{-1}\sf{t}) + 1.
	\end{align}
\end{ex}

In general, one can check that there is a permutation $\sf{Row} : D_{n, d} \to D_{n, d}$ defined by $\sf{Cogen}(\sf{Row}(\Delta)) = \Gen(\Delta) \setminus \{0\}$.
Nathan Williams has pointed out to us that $\sf{Row}$ ought to be an example of rowmotion, a certain operation on the order ideals of a finite poset \cite{sw}.
To see how, regard $\bb{Z}_{\geq 0} \setminus \Gamma(\RRR)$ as a poset in which $j \leq k$ if and only if $k - j \in \Gamma(\RRR)$, and the sets $\bb{Z}_{\geq 0} \setminus \Delta$ for $\Delta \in D_{n, d}$ as its order ideals.
We would be curious to know whether rowmotion sheds any light on the relationship between the Gen and Cogen formulas.

\subsection{Proof of \Cref{thm:n-equals-3}}

We claim that $\cal{F}\sf{Hilb}_{ 3, d}(\sf{q}, \sf{t})	= \cal{F}\sf{Quot}_{3, d}(\sf{q}, \sf{q}^{\frac{1}{2}}\sf{t})$ for $d > 0$ coprime to $3$.
The first step is the asymptotic statement:

\begin{prop}\label{prop:asymptotic}
	For any integer $n > 0$, we have
	\begin{align}
		\lim_{\substack{d \to \infty \\ \text{$d$ coprime to $n$}}}
		\Psi(\sf{a}, \cal{F}\sf{Hilb}_{n, d}(\sf{q}, \sf{t}))
		&= \prod_{1 \leq k \leq n} \frac{1 + \sf{a}\sf{q}^{k - 1} \sf{t}^{2k - 2}}{1 - \sf{q}^k \sf{t}^{2k - 2}},\\
		\lim_{\substack{d \to \infty \\ \text{$d$ coprime to $n$}}}
		\Psi(\sf{a}, \cal{F}\sf{Quot}_{n, d}(\sf{q}, \sf{t}))
		&= \prod_{1 \leq k \leq n} \frac{1 + \sf{a}\sf{t}^{2k - 2}}{1 - \sf{q} \sf{t}^{2k - 2}},
	\end{align}
	where the limits are taken in $\bb{Q}[\![\sf{q}, \sf{t}]\!][\sf{a}]$.
\end{prop}

\begin{proof}
	Throughout, \Cref{cor:nest} allows us to replace the expressions   $\Psi(\cal{F}\sf{Hilb}_{n, d})$ and $\Psi(\cal{F}\sf{Quot}_{n, d})$ with corresponding generating functions for nested pairs of $\RRR$-modules, and \Cref{lem:paving-nest} allows us to compute the latter using the combinatorics of the monomial $\RRR$-modules.
	
	The identity for $\cal{F}\sf{Hilb}_{n, d}$ was shown in \cite{ors}: See their Proposition 6.
	(Recall that our variables $\sf{a}, \sf{q}, \sf{t}$ correspond to their variables $a^2 t, q^2, t$.)
	To prepare for the proof of the second identity, we briefly review their argument.
	
	Using ``staircase diagrams'' \cite[\S{3.2}]{ors} to index monomial ideals, or equivalently elements $\Delta \in I(\RRR)$, then invoking \Cref{lem:paving}, it is not hard to show that the identity for $\cal{F}\sf{Hilb}_{n, d}$  holds when $\sf{a} = 0$.
	Indeed, as $d \to \infty$, the defining condition that staircase width be bounded by $d$ disappears.
	
	The formula that incorporates $\sf{a}$ can be bootstrapped from the $\sf{a} = 0$ formula by systematically replacing single elements $\Delta$ with collections of pairs $(\Delta'', \Delta')$.
	Namely, if $\Delta$ is fixed, then we consider all $2^n$ ways of choosing a subset of $\{1, \ldots, n\}$, and add a column of height $h$ to the staircase of $\Delta$ for each $h$ in the subset.
	This determines some new $\Delta' \in I(\RRR)$.
	We get a larger $\Delta'' \supseteq \Delta$ by replacing each new column with a column that is one box shorter in height.
	We can then check that each $\Delta$ gives rise to $2^n$ pairs $(\Delta'', \Delta')$, that every possible pair arises this way, and that the total contribution of the pairs $(\Delta'', \Delta')$ to the series in $\sf{a}, \sf{q}, \sf{t}$ is the contribution of $\Delta$ to the $\sf{a} = 0$ series multiplied by some binomial factor.
	This factor is precisely the numerator $\prod_{k = 1}^n {(1 + \sf{a} \sf{q}^{k - 1} \sf{t}^{2k - 2})}$.
	
	Now we turn to the identity for $\cal{F}\sf{Quot}_{n, d}$.
	In place of staircases, we index elements $\Delta \in I(\SSS)$ by vectors $\vec{g} = (g_1, \ldots, g_n) \in \bb{Z}_{\geq 0}^n$, where $g_i$ is the number of elements of $\Gamma(\SSS) = \bb{Z}_{\geq 0}$ that are greater than exactly $i - 1$ of the elements of $\Gen_n(\Delta)$.
	Again, as $d \to \infty$, any constraints on the vector $\vec{g}$ disappear.
	If $\Delta$ is indexed by $\vec{g}$, then its contribution to the $\sf{a} = 0$ series is $\sf{q}^{\sum_i g_i} \sf{t}^{2 \sum_i (i - 1)g_i}$ by \Cref{lem:paving}.
	
	To bootstrap the $\sf{a}$ variable, we send $\Delta$ to the collection of all pairs $(\Delta, \Delta')$ where $\Delta$ is the same and $\Delta' \subseteq \Delta$ is obtained as follows:
	Pick a subset of $\{1, \ldots, n\}$, then form $\Gen_n(\Delta')$ from $\Gen_n(\Delta)$ by shifting up by $1$ those elements of $\Gen(\Delta)$ whose residue modulo $n$ belongs to the subset.
	By \Cref{lem:paving-nest}, the total contribution of these pairs to the series in $\sf{a}, \sf{q}, \sf{t}$ is the contribution of the original $\Delta$ to the $\sf{a} = 0$ series multiplied by the binomial factor $\prod_{k = 1}^n {(1 + \sf{a} \sf{t}^{2k - 2})}$.
\end{proof}

Observe that $\Psi(\sf{a}, \cal{F}\sf{Hilb}_{n, d}(\sf{q}, \sf{t}))$ and $\Psi(\sf{a}, \cal{F}\sf{Quot}_{n, d}(\sf{q}, \sf{t}))$ agree with their $d \to \infty$ limits up to degree $d$ in $\sf{q}$.
At the same time:

\begin{prop}\label{prop:delta}
	For any plane curve germ with complete local ring $\RRR$, the series $\Psi(\sf{a}, \cal{F}\sf{Hilb}(\sf{q}, \sf{t}))$ is determined by its expansion up to degree $\delta$ in $\sf{q}$.
	If $\RRR \simeq \bb{C}[\![\varpi^n, \varpi^d]\!]$ for coprime $n, d > 0$, then the same holds for $\Psi(\sf{a}, \cal{F}\sf{Quot}(\sf{q}, \sf{t}))$.
\end{prop}

\begin{proof}
	Observe that the expansion of a formal series $\Psi \in \bb{Z}[\![\sf{q}]\!][\sf{a}^{\pm 1}, \sf{q}^{-1}, \sf{t}^{\pm 1}]$ up to a given $\sf{q}$-degree determines the expansion of $(1 - \sf{q})^b \Psi$ up to that $\sf{q}$-degree, for any integer $b > 0$.
	
	Proposition 3 of \cite{ors} shows that if $\Psi = \Psi(\sf{a},\cal{F}\sf{Hilb}(\sf{q}, \sf{t}))$ and $b$ is the number of branches of $\RRR$, then $\sf{q}^{-\delta} (1 - \sf{q})^b \Psi$ is a Laurent polynomial in $\sf{q}$-degrees $-\delta$ through $\delta$, invariant under $\sf{q}^{-1} \mapsto \sf{q}\sf{t}^2$.
	(Again, our $\sf{q}$ is their $q^2$.)
	So in this case, the expansion of $\Psi$ up to $\sf{q}$-degree $\delta$ determines the entire series.
	
	Now take $\Psi = \Psi(\sf{a}, \cal{F}\sf{Hilb}(\sf{q}, \sf{t}))$, supposing that $\RRR \simeq \bb{C}[\![\varpi^n, \varpi^d]\!]$ for coprime $n, d > 0$.
	By case (1) of \Cref{thm:ors-quot}, $\Psi$ matches the graded dimension of the unreduced KhR homology of the $(n, d)$-torus knot, up to certain grading shifts and substitutions.
	Hence, $(1 - \sf{q})\Psi$ matches the corresponding series from \emph{reduced} KhR homology, as defined in \Cref{sec:conventions}.
	Corollary 1.0.2 of \cite{or} or Theorem 1.2 of \cite{ghm} show that the latter, normalized with our conventions and shifted by $\sf{q}^{-\frac{\delta}{2}}$, is a Laurent polynomial in $\sf{q}^{\frac{1}{2}}$-degrees $-\delta$ through $\delta$, invariant under $\sf{q}^{-\frac{1}{2}} \mapsto \sf{q}^{\frac{1}{2}}\sf{t}$.
	So again, the expansion of $\Psi$ up to $\sf{q}$-degree $\delta$ determines the entire series.
\end{proof}

Together, \Cref{prop:asymptotic} and \Cref{prop:delta} imply that if $\delta \leq d $, then
\begin{align}
	\Psi(\sf{a}, \cal{F}\sf{Hilb}(\sf{q}, \sf{t})) = \Psi(\sf{a}, \cal{F}\sf{Quot}(\sf{q}, \sf{q}^{\frac{1}{2}}\sf{t})).
\end{align}
But $\delta = \frac{1}{2}(n - 1)(d - 1)$.
So the hypothesis can be simplified to $n \leq 3$.
Finally, when $n \leq 3$, the map $\Psi$ loses no information, so we can omit it from both sides.
This proves \Cref{thm:n-equals-3}.

\section{\texorpdfstring{Polynomial Actions and $y$-ification}{Polynomial Actions and y-ification}}\label{sec:polynomial}

\subsection{}

In this section, we review the precise definition of $y$-ified Khovanov--Rozansky homology, then give a precise statement of \Cref{conj:ors-quot-enhanced}, spelling out all of the gradings involved.
This also serves as preparation for \Cref{sec:noncoprime}.

\subsection{}

We freely assume the notation of \Cref{sec:conventions}.
Thus, $T = \bb{G}_m^n$ and $\Bim{\bb{S}}$ is the category of Soergel bimodules over $\bb{S} = \ur{H}_{T}^\ast(\point)$.
We explain in \Cref{sec:conventions} that for any braid $\beta$ on $n$ strands, the Khovanov--Rozansky homology of the link closure of $\beta$ can be computed from Hochschild cohomology of the Rouquier complex $\bar{\cal{T}}_\beta$, an object of $\sf{K}^b(\Bim{\bb{S}})$.

In \cite[\S{5.1}]{BLMSnotes}, the authors explain that the term-by-term action of $\bb{S} \otimes \bb{S}^\op$ on $\bar{\cal{T}}_\beta$ factors through that of a smaller quotient.
Fix matching coordinates 
\begin{align}
	\bb{S} = \bb{C}[t_1, \ldots, t_n]
	\quad\text{and}\quad
	\bb{S}^\op = \bb{C}[t_1^\op, \ldots, t_n^\op].
\end{align}
Let $w \in S_n$ be the underlying permutation of $\beta$.
Then the actions of $t_i$ and $t_{w(i)}^\op$ on $\bar{\cal{T}}_\beta$ are homotopic for all $i$.
So up to homotopy, the $(\bb{S} \otimes \bb{S}^\op)$-action on $\bar{\cal{T}}_\beta$ factors through the quotient of $\bb{S} \otimes \bb{S}^\op$ by the ideal $\langle (t_i - t_{w(i)}^\op)_i \rangle$.

At the same time, the actions of $t_i$ and $t_i^\op$ on $\bb{S}$ coincide for all $i$.
So under the Hochschild cohomology functor $\overline{\sf{HH}} = \bigoplus_{i, j} \Ext_{\bb{S} \otimes \bb{S}^\op}^i(\bb{S}, (-)(j))$, the $(\bb{S} \otimes \bb{S}^\op)$-action on $\bar{\cal{T}}_\beta$ is transported to an action that also factors through the quotient of $\bb{S} \otimes \bb{S}^\op$ by the ideal $\langle (t_i - t_i^\op)_i \rangle$.

Thus, $\overline{\sf{HH}}(\bar{\cal{T}}_\beta)$ inherits an action of the ring of $w$-coinvariants
\begin{align}
	\bb{S}_w \vcentcolon= \bb{S}/\langle (t_i - t_{w(i)})_i \rangle.
\end{align}
This is a polynomial ring on $b$ variables, where $b$ is the number of components of the link closure of $\beta$.
It will be convenient to fix coordinates
\begin{align}
	\bb{S}_w = \bb{C}[\vec{x}] \vcentcolon= \bb{C}[x_1, \ldots, x_b]
\end{align}
so that each $x_j$ is the image of some $t_i$.
Recalling that Soergel bimodules are graded so that $\deg(t_i) = 2$, we see that $\vec{x}$ acts on $\overline{\sf{HH}}(\bar{\cal{T}}_\beta)$ with bidegree $(0, 2)$.
Hence, $\vec{x}$ acts on $\overline{\HHH}(\bar{\cal{T}}_\beta) = \bigoplus_{I, J, K} \ur{H}^K(\overline{\sf{HH}}^{I, J}(\bar{\cal{T}}_\beta))$ with tridegree $(0, 2, 0)$.

\subsection{}

In \cite{GH17}, Gorsky--Hogancamp introduced a deformation of $\overline{\HHH}$ called \dfemph{$y$-ified Khovanov--Rozansky homology}, which we will denote $\overline{\HY}$ and review below.

We write $d$ for the differential on $\bar{\cal{T}}_\beta$.
Let $h_i$ be a homotopy from the $t_i$-action on $T_\beta$ to the $t_{w(i)}^\op$-action, so that $[d, h_i] = t_i - t_{w(i)}^\op$ as operators.
We may choose the $h_i$ so that they square to zero and anticommute.
Let $\bb{S}' = \bb{C}[u_1, \ldots, u_n]$ be another copy of $\bb{S}$, and let $d' = d \otimes {\id} + \sum_i h_i \otimes u_i$ as an operator on $\bar{\cal{T}}_\beta \otimes \bb{S}'$.
We compute that $(d')^2 = \sum_i {(t_i - t_{w(i)}^\op)} \otimes u_i$.
We deduce that the induced action of $(d')^2$ on $\overline{\sf{HH}}(\bar{\cal{T}}_\beta) \otimes \bb{S}'_w$ vanishes, where $\bb{S}'_w \vcentcolon= \bb{S}'/\langle (u_i - u_{w(i)})_i\rangle$, like before.
By definition, $\overline{\HY}(\bar{\cal{T}}_\beta) = \bigoplus_{I, J, K} \overline{\HY}^{I, J, K}(\bar{\cal{T}}_\beta)$, where
\begin{align}
	\overline{\HY}^{I, J, K}(\bar{\cal{T}}_\beta)
	&= \ur{H}^K(\overline{\sf{HH}}^{I, J}(\bar{\cal{T}}_\beta) \otimes \bb{S}'_w, d').
\end{align}
We again fix coordinates $\bb{S}'_w = \bb{C}[\vec{y}] \vcentcolon= \bb{C}[y_1, \ldots, y_b]$, so that each $y_j$ is the image of some $u_i$.
The definition of $d'$ implies that $\vec{y}$ acts on the complex $(\overline{\sf{HH}}(\bar{\cal{T}}_\beta) \otimes \bb{S}'_w, d')$ with bidegree $(0, -2)$ on the first factor and cohomological degree $2$.
Hence, $\vec{y}$ acts on $\overline{\HY}(\bar{\cal{T}}_\beta)$ with tridegree $(0, -2, 2)$.

Altogether, the $y$-ified homology of $\beta$ is a triply-graded vector space $\overline{\HY}(\bar{\cal{T}}_\beta)$ equipped with a bigraded $\bb{C}[\vec{x}, \vec{y}]$-module structure, which recovers $\overline{\HHH}(\bar{\cal{T}}_\beta)$ upon passing from $\bb{C}[\vec{x}, \vec{y}]$ to $\bb{C}[\vec{x}, \vec{y}]/\langle \vec{y}\rangle = \bb{C}[\vec{x}]$.

\subsection{}

Writing $e$ for the writhe of $\beta$, as in \Cref{sec:conventions}, let $\bar{\sf{Y}}_\beta \vcentcolon= \bigoplus_{i, j, k \in \bb{Z}} \bar{\sf{Y}}_\beta^{i, \frac{j}{2}, \frac{k}{2}}$ be the $(\bb{Z} \times \frac{1}{2}\bb{Z} \times \frac{1}{2}\bb{Z})$-graded $\bb{C}[\vec{x}, \vec{y}]$-module defined by
\begin{align}
	\bar{\sf{Y}}_\beta^{i, \frac{j}{2}, \frac{k}{2}}
	= \overline{\HY}^{i, e - 2i + j - k, e - k}(\bar{\cal{T}}_\beta).
\end{align}
From the formula $\overline{\HY}^{I, J, K}(\bar{\cal{T}}_\beta) = \bar{\sf{Y}}_\beta^{I, I + \frac{J}{2} - \frac{K}{2}, e - \frac{K}{2}}$, we see that 
\begin{align}
	\bar{\sf{X}}_\beta(\sf{a}, \sf{q}, \sf{t})
	=	\sum_{i, j, k \in \bb{Z}}
	\sf{a}^i \sf{q}^{\frac{j}{2}} \sf{t}^{\frac{k}{2}}
	\dim (
	\bar{\sf{Y}}_\beta^{i, \frac{j}{2}, \frac{k}{2}} \otimes_{\bb{Z}[\vec{x}, \vec{y}]} \bb{Z}[\vec{x}]
	)
\end{align}
in the notation of \Cref{sec:conventions}.
Moreover, we see that $\vec{x}$ and $\vec{y}$ respectively act on each summand $\bar{\sf{Y}}_\beta^i \vcentcolon= \bigoplus_{j, k \in \bb{Z}} \bar{\sf{Y}}_\beta^{i, \frac{j}{2}, \frac{k}{2}}$ with bidegrees $(1, 0)$ and $(0, -1)$.

We return to our setup where $f(x, y) = 0$ is a generically separable degree-$n$ cover of the $x$-axis, embedded in the $x, y$-plane.
The preimage in the cover of a positively-oriented loop around $x = 0$ is a braid $\beta_f$ on $n$ strands such that the number $b$ of branches of $f$ is also the number of components of the link closure of $\beta$, and such that $\bar{\sf{X}}_f = \bar{\sf{X}}_{\beta_f}$.
We similarly set $\bar{\sf{Y}}_f = \bar{\sf{Y}}_{\beta_f}$.

Let $T(b) = \bb{G}_m^b$.
As explained in the introduction, once we fix identifications 
\begin{align}
	\bb{C}[\vec{x}] \simeq \bb{C}[\Lattice_{\geq 0}]
	\quad\text{and}\quad
	\bb{C}[\vec{y}] \simeq \ur{H}_{T(b)}^\ast(\point),
\end{align}
the commuting actions of $\Lattice_{\geq 0}$ and $T(b)$ on $\coprod_\ell \cal{Q}_\nu^\ell$ together produce a $\bb{C}[\vec{x}, \vec{y}]$-module structure on $\bigoplus_\ell \ur{H}_\ast^{\BM, T(b)}(\cal{Q}_\nu^\ell)$ for all compositions $\nu$ of $n$.
The variables $x_j$ and $y_j$ respectively act by $1$ and $0$ on the length $\ell$, by $0$ and $-2$ on the cohomological degree, and by $0$ and $-2$ on the weight filtration $\sf{W}_{\leq \ast}$.

Let $\sf{Q}_{\SSS, \nu}^{\vec{x}, \vec{y}} \vcentcolon= \bigoplus_{\ell, k} \sf{Q}_{\SSS, \nu}^{\vec{x}, \vec{y}, \ell, k}$ be the $\bb{Z}^2$-graded $\bb{C}[\vec{x}, \vec{y}]$-module defined by
\begin{align}
	\sf{Q}_{\SSS, \nu}^{\vec{x}, \vec{y}, \ell, k}
	= \gr_k^\sf{W} \ur{H}_\ast^{\BM, T(b)}(\cal{Q}_\nu^\ell).
\end{align}
We abbreviate by writing $\tilde{\sf{Q}}_\SSS^{\vec{x}, \vec{y}} = \sf{Q}_{\SSS, (1^n)}^{\vec{x}, \vec{y}}$.
The Springer action of $S_n$ on the Borel--Moore homology of $\coprod_\ell \cal{Q}_{(1^n)}^\ell$ lifts to its equivariant Borel--Moore homology and commutes with the $\bb{C}[\vec{x}, \vec{y}]$-action above.
So by \Cref{prop:springer}, we can use the bigraded $(\bb{C}[\vec{x}, \vec{y}] \times \bb{C}S_n)$-module formed by $\tilde{\sf{Q}}_\SSS^{\vec{x}, \vec{y}}$ to recover the bigraded $\bb{C}[\vec{x}, \vec{y}]$-modules $\sf{Q}_{\SSS, \nu}^{\vec{x}, \vec{y}}$ for all $\nu$.

Abusing notation, let $\Psi$ be the functor from bigraded $\bb{C}S_n$-modules to triply-graded vector spaces given by
\begin{align}
	\Psi(M)^{i, j, k} = \bigoplus_{j, k} 
	\Hom_{S_n}(V_{(n - i + 1, 1^{i - 1})} \oplus V_{(n - i, 1^i)}, M^{j, k}),
\end{align}
where in general, $V_\lambda$ is the irreducible representation of $S_n$ indexed by $\lambda \vdash n$.
Altogether, the most precise version of \Cref{conj:ors-quot-enhanced} is:

\begin{conj}\label{conj:ors-quot-enhanced-full}
	In the setup above,
	\begin{enumerate}
		\item 	$\bar{\sf{Y}}_f$ is supported in integral tridegrees.
		
		\item 	There is an isomorphism of $\bb{C}[\vec{x}, \vec{y}]$-modules $\bar{\sf{Y}}_f
		\xrightarrow{\sim} \Psi(\tilde{\sf{Q}}_\SSS^{\vec{x}, \vec{y}})$ 
		that sends degree $(i, j, k)$ onto degree $(i, j, 2k)$.
		In particular, $\Psi(\tilde{\sf{Q}}_\SSS^{\vec{x}, \vec{y}})$ is supported in even cohomological degrees.
		
	\end{enumerate}
\end{conj}

\begin{rem}
	In the definition of $\tilde{\sf{Q}}_\SSS^{\vec{x}, \vec{y}}$, we did not collapse the cohomological degree to an Euler characteristic, as in the definition of $\sf{Quot}(\sf{q}, \sf{t})$.
	Thus, the statement that $\Psi(\tilde{\sf{Q}}_\SSS^{\vec{x}, \vec{y}})$ is supported in even cohomological degrees is needed to ensure that \Cref{conj:ors-quot-enhanced-full} specializes to \Cref{conj:ors-quot} upon base change from $\bb{C}[\vec{x}, \vec{y}]$ to $\bb{C}[\vec{x}]$.
	An analogous statement about the cohomology of $\CptPic/\Lattice$ was shown in \cite{gmo} for certain unibranch plane curve germs, called ``generic'' germs in \cite{gmo}.
\end{rem}

\section{\texorpdfstring{$(n, nk)$ Torus Links}{(n, nk) Torus Links}}\label{sec:noncoprime}

\subsection{}

In this section, we prove case (2) of \Cref{thm:ors-quot}, stating in the notation of \S\ref{subsec:intro-toric} that $\bar{\sf{X}}_{n, nk}(\sf{a}, \sf{q}, \sf{t}^2) = \Psi(\sf{a}, \cal{F}\sf{Quot}_{n, nk}(\sf{q}, \sf{t}))$ for any integer $k > 0$.

Throughout, $f(x, y) = y^n - x^{nk}$.
For such $f$, our argument will implicitly prove \Cref{conj:ors-quot-enhanced-full}(1), as well as the matching of trigradings in \Cref{conj:ors-quot-enhanced-full}(2).
The strategy is to relate both sides to $\nabla^k p_{(1^n)} \in \Lambda_{\sf{q}, \sf{t}}^n$, where in general, $p_\lambda$ is the power-sum symmetric function indexed by $\lambda \vdash n$, and $\nabla$ is the Bergeron--Garsia operator on $\Lambda_{\sf{q}, \sf{t}}^n$ \cite{hhlru}.
We will use the theory of symmetric functions freely.
For more background on our tools, see \cite{haiman, macdonald}.

\subsection{}

In \cite{CM21}, Carlsson--Mellit computed a version of the underlying bigraded $\bb{C}S_n$-module of $\tilde{\sf{Q}}_\SSS^{\vec{x}, \vec{y}}$ for the chosen $f$.
To make this precise, let 
\begin{align}
	\tilde{\sf{Q}}_{\SSS, n, nk}^{\BM, T(n)}(\sf{q}, \sf{t})
	= \sum_{\ell, k}
	\sf{q}^\ell \sf{t}^k 
	\ur{H}_k^{\BM, T(n)}(\cal{Q}_{(1^n)}^\ell)
	\in \bb{Q}(\sf{q}, \sf{t}) \otimes K_0(S_n).
\end{align}
Recall the Frobenius character $\cal{F} : \bb{Q}(\sf{q}, \sf{t}) \otimes K_0(S_n) \to \Lambda_{\sf{q}, \sf{t}}^n$ from \Cref{sec:springer}.

\begin{prop}
	For all integers $n, k > 0$, we have
	\begin{align}
		\cal{F}\tilde{\sf{Q}}_{\SSS, n, nk}^{\BM, T(n)}(\sf{q}, \sf{t})
		= \frac{1}{(1 - \sf{q})(1 - \sf{t}^2)}
		\nabla^k p_n.
	\end{align}
\end{prop}

\begin{proof}
	Just as the ind-schemes $\CptPic_\nu$ are isomorphic to parabolic affine Springer fibers for $\GL_n$, so the ind-schemes $\coprod_\ell \cal{Q}_\nu^\ell$ are isomorphic to the \dfemph{positive} parts of certain affine Springer fibers, in the terminology of \cite{GK, CM21}.
	This can be shown by adapting the proof of \cite[Thm.\@ 1.1]{GK}.
	In \Cref{prop:asf-noncoprime}, we give the explicit isomorphisms for the case where $f(x, y) = y^n - x^{nk}$, and show that for $\nu = (1^n)$, they match the Springer actions on the two sides.
	In particular, we match $\coprod_\ell \cal{Q}_{(1^n)}^\ell$ for this choice of $f$ with the ind-scheme denoted $Z_k$ in \cite{CM21}.
	
	There is an extra $S_n$-action on the $T(n)$-equivariant Borel--Moore homology of $Z_k$ called the \dfemph{dot action}, induced by the $S_n$-action on the homotopy type of the curve $y^n = x^{kn}$ that permutes its branches.
	The dot action commutes with the Springer action.
	In this way, we can upgrade $\cal{F}\tilde{\sf{Q}}_{\SSS, n, nk}^{\BM, T(n)}(\sf{q}, \sf{t})$ to an element
	\begin{align}
		\cal{F}_{\vec{Y}, \vec{Z}}\tilde{\sf{Q}}_{\SSS, n, nk}^{\BM, T(n)}(\sf{q}, \sf{t})
		\in \Lambda_{\sf{q}, \sf{t}}^n[\vec{Y}, \vec{Z}],
	\end{align}
	where $\Lambda_{\sf{q}, \sf{t}}[\vec{Y}, \vec{Z}] = \Lambda_{\sf{q}, \sf{t}}[\vec{Y}] \otimes_{\bb{Q}(\sf{q}, \sf{t})} \Lambda_{\sf{q}, \sf{t}}[\vec{Z}]$.
	Above, $\vec{Y}$ and $\vec{Z}$ respectively record the Springer and dot actions.
	The actual statement of \cite[Thm.\@ A]{CM21} is
	\begin{align}
		\cal{F}_{\vec{Y}, \vec{Z}}\tilde{\sf{Q}}_{\SSS, n, nk}^{\BM, T(n)}(\sf{q}, \sf{t}^{\frac{1}{2}}) 
		= \nabla^k e_n\left[\frac{\vec{Y} \vec{Z}}{(1 - \sf{q})(1 - \sf{t})}\right],
	\end{align}
	in plethystic notation.
	
	We want to recover the Frobenius character in $\vec{Y}$ alone.
	To this end, it suffices to pair the right-hand side with $p_{(1^n)}[\vec{Z}]$ under the Hall inner product:
	Indeed, under $\cal{F}$, pairing with $p_{(1^n)}$ corresponds to evaluating a character of $S_n$ at the identity element.
	Note that $(g, h) \mapsto \langle g[\frac{\vec{Y}\vec{Z}}{(1 - \sf{q})(1 - \sf{t})}], h\rangle$ is a version of Macdonald's $\sf{q}, \sf{t}$-inner product \cite[\S{3.5}]{haiman}, with respect to which the power-sum symmetric functions form an orthogonal basis of $\Lambda_{\sf{q}, \sf{t}}^n$.
	Therefore
	\begin{align}
		\left\langle \nabla^k
		e_n \left[\frac{\vec{Y} \vec{Z}}{(1 - \sf{q})(1 - \sf{t})}\right],
		p_{(1^n)}[\vec{Z}]\right\rangle
		&= \nabla^k p_{(1^n)} \left[\frac{\vec{Y}}{(1 - \sf{q})(1 - \sf{t})}\right]\\
		&= \frac{1}{(1 - \sf{q})^n (1 - \sf{t})^n} 
		\nabla^k p_{(1^n)}[\vec{Y}],
	\end{align}
	where the second equality used $p_{(1^n)} = p_1^n$.
	Finally, substituting $\sf{t}^2$ for $\sf{t}$ everywhere gives the statement in the proposition.
\end{proof}

\begin{rem}
	Interestingly, the fundamental domain $\cal{D}_{(1^n)}$ from Lemma \ref{lem:domain} and its ensuing discussion appears implicitly in \cite{CM21}:
	Its complement is an open sub-ind-scheme of $Z_k$ that features heavily in the proof of \cite[Thm.\@ A]{CM21}.
\end{rem}

\begin{cor}
	For all integers $n, k > 0$, we have
	\begin{align}
		\cal{F}\sf{Quot}_{n, nk}(\sf{q}, \sf{t})
		= \frac{1}{(1 - \sf{q})^n} 
		\nabla^k p_{(1^n)}.
	\end{align}
\end{cor}

\begin{proof}
	Since the homology of $Z_k$ is pure \cite{gkm_04, gkm}, it is $T(n)$-equivariantly formal \cite[Lem.\@ 2.2]{gkm_04}.
	We deduce that if $\sf{Q}_{S, n, nk}^\BM$ is the analogue of $\tilde{\sf{Q}}_{\SSS, n, nk}^{\BM, T(n)}$ for non-equivariant Borel--Moore homology, then
	\begin{align}
		\cal{F}\sf{Q}_{\SSS, n, nk}^\BM(\sf{q}, \sf{t})
		= (1 - \sf{t}^2)^n
		\cal{F}\tilde{\sf{Q}}_{\SSS, n, nk}^{\BM, T(n)}(\sf{q}, \sf{t}) 
		=  \frac{1}{(1 - \sf{q})^n} 
		\nabla^k p_{(1^n)}.
	\end{align}
	Next, recall that Borel--Moore homology and compactly-supported cohomology with complex coefficients are dual to each other.
	Finally, since both are supported in even degrees \cite[38]{CM21}, and in degree $i$, pure of weight $i$ \cite[Cor.\@ 1.3]{gkm}, we know that $\sum_i \sf{t}^i \dim \ur{H}_c^i(Z_k) = \chi(Z_k, \sf{t})$.
\end{proof}

\subsection{}

Turning to the KhR side, observe that Gorsky--Hogancamp computed the $y$-ified KhR homology of the $(n, nk)$ torus link in \cite{GH17}, obtaining its usual KhR homology as a corollary.
After including the correct denominators, \cite[Thm.\@ 7.13]{GH17} says
\begin{align}
	\bar{\sf{Y}}_{n, nk}(\sf{a}, \sf{q}, \sf{t})
	\vcentcolon= \sum_{i, j, k} \sf{a}^i \sf{q}^{\frac{j}{2}} \sf{t}^{\frac{k}{2}}
	\dim(\bar{\sf{Y}}_f^{i, \frac{j}{2}, \frac{k}{2}})
	= \frac{1}{(1 - \sf{q})^n (1 - \sf{t})^n} \Psi(\nabla^k p_{(1^n)}, \sf{a}).
\end{align}
Similarly, after correction, \cite[Thm.\@ 7.14]{GH17} says
\begin{align}
	\bar{\sf{X}}_{n, nk}(\sf{a}, \sf{q}, \sf{t})
	= (1 - \sf{t})^n \bar{\sf{Y}}_{n, nk}(\sf{a}, \sf{q}, \sf{t})
	= \frac{1}{(1 - \sf{q})^n} \Psi(\nabla^k p_{(1^n)}, \sf{a}).
\end{align}
Again, we refer to \Cref{sec:polynomial} and \Cref{sec:conventions} to match our grading conventions with those in \cite{GH17}.
This concludes the proof of case (2) of \Cref{thm:ors-quot}.

\subsection{}

To conclude this section, we verify the $\sf{a} = 0$ limit of \cite[Conj.\ 2]{ors} for two plane curve germs of the form $y^n = x^{nk}$.
By way of case (2) of \Cref{thm:ors-quot}, this also verifies \Cref{conj:main} in these cases.

\begin{ex}
	Take $n = 2$ and $k = 2$.
	By \cite[Ex.\ 6.18]{Kiv20},
	\begin{align}
		\sf{Hilb}(\sf{q}, \sf{t})
		= \frac{1}{(1 - \sf{q})^2}
		(1 - \sf{q} + \sf{q}^2 \sf{t}^2 - \sf{q}^3 \sf{t}^2 + \sf{q}^4 \sf{t}^4).
	\end{align}
	At the same time, the recursion of \cite{hm, gmv} gives
	\begin{align}
		\bar{\sf{X}}_{2, 4}(\sf{a}, \sf{q}, \sf{t})
		= \frac{1}{(1 - \sf{q})^2}
		(1 + \sf{q}(\sf{t} - 1) + \sf{q}^2 (\sf{t}^2 - \sf{t})).
	\end{align}
	These series agree under $(\sf{q}, \sf{t}) \mapsto (\sf{q}, \sf{q}\sf{t}^2)$.
\end{ex}

\begin{ex}
	Take $n = 3$ and $k = 1$.
	By \cite[Ex.\ 6.17]{Kiv20},
	\begin{align}
		\sf{Hilb}(\sf{q}, \sf{t})
		=	\frac{1}{(1 - \sf{q})^3}
		\pa{\begin{array}{l}
				1 
				- 2\sf{q}
				+ \sf{q}^2 (\sf{t}^2 + 1)
				+ \sf{q}^3 (\sf{t}^4 - 2\sf{t}^2)
				+ \sf{q}^4 (\sf{t}^4 + \sf{t}^2)\\
				\qquad- 2 \sf{q}^5 \sf{t}^4
				+ \sf{q}^6 \sf{t}^6
		\end{array}}.
	\end{align}
	At the same time, by \cite[Ex.\ 32]{gmv},
	\begin{align}
		\bar{\sf{X}}_{3, 3}(\sf{a}, \sf{q}, \sf{t})
		= \frac{1 + \sf{q}\sf{t}}{1 - \sf{q}}
		+ \frac{\sf{q} \sf{t}^2 + 2 \sf{q}^2 \sf{t}^2}{(1 - \sf{q})^2}
		+ \frac{\sf{q}^3 \sf{t}^3}{(1 - \sf{q})^3}.
	\end{align}
	Again these agree under $(\sf{q}, \sf{t}) \mapsto (\sf{q}, \sf{q}\sf{t}^2)$.
\end{ex}

\section{Affine Springer Fibers}\label{sec:asf}

\subsection{}

In this section, we establish the comparisons to affine Springer fibers needed in \Crefrange{sec:coprime}{sec:noncoprime}.
For the general relationship between local compactified Jacobians and affine Springer fibers, see \cite{laumon}.

\subsection{}

Suppose that $G$ is a complex reductive algebraic group.
Its \dfemph{loop group} $\hat{G}$ (resp. \dfemph{arc group} $\hat{K}$)  is defined by  $\hat{G}(A) = G(A(\!(x)\!))$ (resp. $\hat{K}(A) = G(A[\![x]\!])$) for all $\bb{C}$-algebras $A$.
Thus there is a projection map $\hat{K} \to G$ that sends $g(x) \mapsto g(0)$.

Henceforth, let $G = \GL_n$ and $\fr{g} = \fr{gl}_n$.
Each integer composition $\nu$ of $n$ defines a block-upper-triangular parabolic subgroup $P_\nu \subseteq G$.
Its preimage $\hat{K}_\nu \subseteq \hat{K}$ is called the corresponding \dfemph{parahoric subgroup}.
The \dfemph{partial affine flag variety of $G$ of parabolic type} $\nu$ is the fpqc quotient $\hat{\cal{B}}_\nu = \hat{G}/\hat{K}_\nu$, which turns out to be an ind-scheme.
For any $\gamma \in \fr{g}(\bb{C}[\![x]\!])$, let
\begin{align}
	\hat{\cal{B}}_{\nu, \unred}^\gamma = \{g\hat{K}_\nu \in \hat{\cal{B}}_\nu \mid \Ad(g^{-1})\gamma \in \ur{Lie}(\hat{K}_\nu)\}.
\end{align}
The underlying reduced ind-scheme $\hat{\cal{B}}_\nu^\gamma \subseteq \hat{\cal{B}}_{\nu, \unred}^\gamma$ is called the \dfemph{affine Springer fiber} over $\gamma$ \dfemph{of parabolic type} $\nu$.
Since $\hat{G}/\hat{K}$ is also known as the \dfemph{affine Grassmannian}, we set $\Grass = \hat{G}/\hat{K} = \hat{\cal{B}}_{(n)}$ and $\Grass^\gamma = \hat{\cal{B}}_{(n)}^\gamma$.

\subsection{The Functor $\mathcal{L}$, and $y^n = x^{nk}$}\label{subsec:asf-noncoprime}

There is a well-known lattice description of the above spaces.
Namely, let $\cal{L}$ be the functor from $\bb{C}$-algebras to sets defined by
\begin{align}
	\cal{L}(A) = \left\{
	\begin{array}{l}
		\text{$A[\![x]\!]$-submodules}\\
		L \subseteq A(\!(x)\!)^n
	\end{array}
	\middle|
	\begin{array}{l}
		\text{$\exists\, i$ such that $x^i A[\![x]\!]^n
			\subseteq L \subseteq x^{-i} A[\![x]\!]^n$}\\
		\text{and $(x^{-i} A[\![x]\!]^n)/L$ is locally free over $A$}\\
		\text{of finite rank}
	\end{array}
	\right\}
\end{align}
for any $\bb{C}$-algebra $A$.
For any $\nu$, let $\cal{L}_\nu$ be the functor defined by
\begin{align}
	\cal{L}_\nu(A) = \left\{(L, F) \,\middle| 
	\begin{array}{l}
		L \in \cal{L}(A),\\
		\text{$F$ is a partial flag on $\bar{L} \vcentcolon= L/x L$ of type $\nu$}
	\end{array}\right\}.
\end{align}
Let $F^\std$ be the unique partial flag on $\bb{C}^n$ of type $\nu$ that has stabilizer $P_\nu$ under \emph{right} multiplication (of row vectors) by $\GL_n(\bb{C})$.
\begin{lem}\label{lem:l}
	For each integer composition $\nu$ of $n$, there is an isomorphism of fpqc sheaves $\hat{\cal{B}}_\nu \xrightarrow{\sim} \cal{L}_\nu$ that sends
	\begin{align}\label{eq:affine-flag-to-l}
		g\hat{K}_\nu \mapsto (L_g, F_g) \vcentcolon= (\bb{C}[\![x]\!]^n \cdot g^{-1}, F^\std \cdot g^{-1})
	\end{align}
	for all $g\hat{K}_\nu \in \hat{\cal{B}}_\nu(\bb{C})$.
	In particular, $\cal{L}$ is representable by an ind-scheme.
\end{lem}

Let $\cal{L}_+ \subseteq \cal{L}$ be the sub-ind-scheme defined by
$
\cal{L}_+(A) = \{L \in \cal{L}(A) \mid L \subseteq A[\![x]\!]^n \}.
$
We define the \dfemph{positive part} of $\hat{\cal{B}}_\nu$ to be the corresponding sub-ind-scheme $\hat{\cal{B}}_{\nu, +} \subseteq \hat{\cal{B}}_\nu$.
Similarly, we define the \dfemph{positive part} of $\hat{\cal{B}}_\nu^\gamma$ to be $\hat{\cal{B}}_{\nu, +}^\gamma = \hat{\cal{B}}_\nu^\gamma \cap \hat{\cal{B}}_{\nu, +}$.
We set $\Grass_+ = \hat{\cal{B}}_{(n), +}$ and $\Grass_+^\gamma = \hat{\cal{B}}_{(n), +}^\gamma$.

Fix a primitive $n$th root of unity $\zeta \in \bb{C}^\times$.
For any integer $k > 0$, let 
\begin{align}
	\gamma(k) = \diag(x^k, \zeta x^k, \ldots, \zeta^{n - 1} x^k) \in \fr{g}(\bb{C}[\![x]\!]).
\end{align}
We see that the centralizer of $\gamma(k)$ in $\hat{G}$ is precisely $\hat{T} \subseteq \hat{G}$, where $T \subseteq G$ is the maximal torus of diagonal matrices.
The $\hat{T}$-action on $\hat{\cal{B}}_\nu^{\gamma(k)}$ by left multiplication restricts to a $T$-action on $\hat{\cal{B}}_{\nu, +}^{\gamma(k)}$.
We note that the ind-scheme $\hat{\cal{B}}_{(1^n), +}^{\gamma(k)}$ is denoted $Z_k$ in \cite{CM21}.

\begin{prop}\label{prop:asf-noncoprime}
	Suppose that $\RRR = \bb{C}[\![x, y]\!]/(y^n - x^{nk})$ for some integer $k > 0$.
	Fix an identification $\SSS = \bb{C}[\![x]\!]^n$, hence an identification $T(n) = T$.
	Then:
	\begin{enumerate}
		\item 	The map \eqref{eq:affine-flag-to-l} restricts to isomorphisms $\hat{\cal{B}}_\nu^{\gamma(k)} \xrightarrow{\sim} \CptPic_\nu$ and $\hat{\cal{B}}_{\nu, +}^{\gamma(k)} \xrightarrow{\sim} \coprod_\ell \cal{Q}_\nu^\ell$.
		
	\end{enumerate}
	Let $\hat{\cal{B}}_{\nu, +}^{\gamma(k), \ell} \subseteq \hat{\cal{B}}_{\nu, +}^{\gamma(k)}$ correspond to $\cal{Q}_\nu^\ell \subseteq \coprod_\ell \cal{Q}_\nu^\ell$ under the isomorphism in (1).
	Then the isomorphism matches:
	\begin{enumerate}\setcounter{enumi}{1}
		\item 	The $T$-actions on $\hat{\cal{B}}_{\nu, +}^{\gamma(k), \ell}$ and $\cal{Q}_\nu^\ell$.
		
		\item 	The Springer actions of $S_n$ on the $T$-equivariant Borel--Moore homologies of $\cal{B}_{(1^n), +}^{\gamma(k), \ell}$ and $\cal{Q}_{(1^n)}^\ell$, for all $\ell$.
		
	\end{enumerate}
\end{prop}

\begin{proof}
	Parts (i) and (ii) follow from the definitions:
	Compare to \cite[Thm.\@ 1.1]{GK}.
	To prove part (iii), observe that the usual Springer action on the Borel--Moore homology of $\hat{\cal{B}}_{(1^n), +}^{\gamma(k), \ell}$ arises from \Cref{prop:springer} and \Cref{rem:borel-moore} via the outer rectangle in the following diagram, where every square is cartesian:
	\begin{equation}
		\begin{tikzpicture}[baseline=(current bounding box.center), >=stealth]
			\matrix(m)[matrix of math nodes, row sep=2.5em, column sep=3em, text height=2ex, text depth=0.5ex]
			{ 		
				{[\hat{\cal{B}}_{(1^n), +}^{\gamma(k), \ell}/T]}
				&{[\cal{Q}_{(1^n)}^\ell/T]}
				&{[\widetilde{\cal{N}}_{(1^n)}/{\GL_n}]}\\	
				{[\Grass_+^{\gamma(k), \ell}/T]}
				&{[\cal{Q}^\ell/T]}
				&{[\cal{N}/{\GL_n}]}\\
			};
			\path[->, font=\scriptsize, auto]
			(m-1-1) edge node{$\sim$} (m-1-2)
			(m-2-1) edge node{$\sim$} (m-2-2)
			(m-1-1) edge (m-2-1)
			(m-1-2) edge (m-2-2)
			(m-1-2) edge (m-1-3)
			(m-2-2) edge (m-2-3)
			(m-1-3) edge node{$\pi$} (m-2-3);
		\end{tikzpicture}
	\end{equation}
	(Above, $\Grass_+^{\gamma(k), \ell} \vcentcolon= \hat{\cal{B}}_{(n), +}^{\gamma(k), \ell}$.)
\end{proof}

\begin{rem}
	In \cite{bl}, Boixeda Alvarez--Losev construct commuting actions of two trigonometric double affine Hecke algebras (DAHAs) on the $T$-equivariant Borel--Moore homology of certain equivalued affine Springer fibers, for a certain torus $T$.
	
	In the $\GL_n$ case, their affine Springer fibers are precisely our $\hat{\cal{B}}_{(1^n)}^{\gamma(k)}$, and their $T$ is our $T$.
	Via \Cref{prop:asf-noncoprime}, the monodromic action of the cocharacter lattice and the action of equivariant cohomology in \cite{bl} respectively correspond to the $\Lattice_{\geq 0}$- and $\ur{H}_T^\ast(\point)$-actions on the $T$-equivariant Borel--Moore homology of $\cal{Q}_{(1^n)}^\ell$ in \Cref{sec:polynomial}.
	The monodromic action of the finite Weyl group corresponds to the dot action in \Cref{sec:noncoprime}.
\end{rem}

\subsection{The Functor $\mathcal{M}$, and $y^n = x^d$}\label{subsec:asf-coprime}

Let $(\sf{v}_i)_{i = 1}^n$ be the standard ordered basis of $\bb{C}^n$.
Writing $x = \varpi^n$, we have an isomorphism of $\bb{C}(\!(x)\!)$-vector spaces 
\begin{align}
	\abel : \bb{C}(\!(x)\!)^n = \bb{C}(\!(x)\!) \otimes \bb{C}^n \xrightarrow{\sim} \bb{C}(\!(\varpi)\!)
	\quad\text{defined by $\abel(\sf{v}_i) = \varpi^{i - 1}$}.
\end{align}
Let $\cal{M}$ be the functor from $\bb{C}$-algebras to sets defined by
\begin{align}
	\cal{M}(A) = \left\{
	\begin{array}{l}
		\text{$A[\![\varpi^n]\!]$-submodules}\\
		M \subseteq A(\!(\varpi)\!)
	\end{array}
	\middle|
	\begin{array}{l}
		\text{$\exists\, j$ s.t.\ $\varpi^j A[\![\varpi]\!] \subseteq M \subseteq \varpi^{-j} A[\![\varpi]\!]$}\\
		\text{and $(\varpi^{-j} A[\![\varpi]\!])/M$ is locally free}\\
		\text{over $A$ of finite rank}
	\end{array}
	\right\}.
\end{align}
Thus $\cal{M}$ is the analogue of $\CptPic_\unred$ with $\bb{C}[\![\varpi]\!]$ in place of $\RRR$.
For any $\nu$, let $\cal{M}_\nu$ be the functor defined by
\begin{align}
	\cal{M}_\nu(A) = \left\{(M, F) \,\middle| 
	\begin{array}{l}
		M \in \cal{M}(A),\\
		\text{$F$ is a partial flag on $\bar{M} \vcentcolon= M/\varpi^n M$ of type $\nu$}
	\end{array}\right\}.
\end{align}
Then $\abel$ induces an isomorphism 
\begin{align}\label{eq:l-to-m}
	\sf{Ab} : \cal{L}_\nu \xrightarrow{\sim} \cal{M}_\nu.
\end{align}
We now define the element of $\fr{g}(\bb{C}[\![x]\!])$ studied in \cite{hikita}.
Let $(X^\bullet, \Phi, X_\bullet, \Phi^\vee)$ be the root datum of $G$ with respect to the maximal torus of diagonal matrices.
Let $\alpha_1, \ldots, \alpha_{n - 1} \in \Phi$ be the simple roots with respect to the upper-triangular Borel subgroup $P_{(1^n)} \subseteq G$, and let $\rho^\vee = \frac{1}{2} \sum_i \alpha_i^\vee \in \Phi^\vee$, where $\alpha_i^\vee$ is the coroot corresponding to $\alpha_i$.
For any $d > 0$ coprime to $n$, let $m, b$ be the integers such that $d = mn + b$ and $0 < b < n$, as in \cite{hikita}.
For each root $\alpha$, let $e_\alpha \in \fr{g}(\bb{C})$ be the zero-one matrix that generates the root subspace $\fr{g}_\alpha \subseteq \fr{g}$, and for each integer $j$, let $e_j = \sum_{\alpha \mid \langle \alpha, \rho^\vee\rangle = j} e_\alpha$.
Finally, let
\begin{align}\label{eq:hikita}
	\psi(d) = x^m e_b + x^{m + 1} e_{b - n}.
\end{align}
In what follows, we will need the composition of isomorphisms
\begin{align}\label{eq:affine-flag-to-m}
	\hat{\cal{B}}_\nu \xrightarrow{\text{Lem \ref{lem:l}}} \cal{L}_\nu \xrightarrow{\sf{Ab}} \cal{M}_\nu \xrightarrow{\varpi^\delta} \cal{M}_\nu,
\end{align}
where the last map is multiplication by $\varpi^\delta$, and $\delta = \frac{1}{2}(n - 1)(d - 1)$, as in \Cref{sec:coprime}.
We write the map on $\bb{C}$-points as $g\hat{K}_\nu \mapsto (M_g, F_g)$.

\begin{prop}\label{prop:asf-coprime}
	Suppose that $\RRR = \bb{C}[\![\varpi^n, \varpi^d]\!]$ for some $d > 0$ coprime to $n$.
	Then:
	\begin{enumerate}
		\item 	The map \eqref{eq:affine-flag-to-m} restricts to an isomorphism $\hat{\cal{B}}_\nu^{\psi(d)} \xrightarrow{\sim} \CptPic_\nu$.
	\end{enumerate}
	Let $\CptPic_\nu^e \subseteq \CptPic_\nu$ be the preimage of $\CptPic^e \subseteq \CptPic$, and let $\hat{\cal{B}}_\nu^{\psi(d), e} \subseteq \hat{\cal{B}}_\nu^{\psi(d)}$ correspond to $\CptPic_\nu^e \subseteq \CptPic_\nu$ under the isomorphism in (1).
	Then:
	\begin{enumerate}\setcounter{enumi}{1}
		\item 	The isomorphism in (1) matches the Springer action of $S_n$ on the Borel--Moore homologies of $\cal{B}_{(1^n)}^{\psi(d), e}$ and $\CptPic_{(1^n)}^e$, for all $e$.
		
		\item 	$\hat{\cal{B}}_\nu^{\psi(d), 0}$ is the affine Springer fiber studied by Hikita in \cite{hikita}.
		
	\end{enumerate}
\end{prop}

\begin{proof}
	Part (i):
	It suffices to work on $\bb{C}$-points.
	By checking on the basis $(\sf{v}_i)_i$, we find that $\abel$ transports the action of $\psi(d)$ on $\bb{C}(\!(x)\!)^n$ by right multiplication onto the action of $\varpi^d$ on $\bb{C}(\!(\varpi)\!)$ by multiplication.
	Therefore,
	\begin{align}
		g\hat{K}_\nu \in \cal{B}_\nu^{\psi(d)}(\bb{C})
		&\iff \text{$(\bb{C}[\![x]\!]^n \cdot g^{-1}, F^\std \cdot g^{-1})$ is $\gamma$-stable}\\
		&\iff \text{$(M_g, F_g)$ is $\RRR$-stable}
	\end{align}
	for all $g\hat{K}_\nu \in \hat{\cal{B}}_\nu(\bb{C})$ and fixed $e \in \bb{Z}$.
	
	Part (ii):
	Similar to the proof of part (iii) of \Cref{prop:asf-noncoprime}, but replacing the diagram there with this one:
	\begin{equation}\label{eq:asf-to-pic}
		\begin{tikzpicture}[baseline=(current bounding box.center), >=stealth]
			\matrix(m)[matrix of math nodes, row sep=2.5em, column sep=3em, text height=2ex, text depth=0.5ex]
			{ 		
				\hat{\cal{B}}_{(1^n)}^{\psi(d), e}
				&\CptPic_{(1^n)}^e
				&{[\widetilde{\cal{N}}_{(1^n)}/{\GL_n}]}\\	
				\Grass^{\psi(d), e}
				&\CptPic^e
				&{[\cal{N}/{\GL_n}]}\\
			};
			\path[->, font=\scriptsize, auto]
			(m-1-1) edge node{$\sim$} (m-1-2)
			(m-2-1) edge node{$\sim$} (m-2-2)
			(m-1-1) edge (m-2-1)
			(m-1-2) edge (m-2-2)
			(m-1-2) edge (m-1-3)
			(m-2-2) edge (m-2-3)
			(m-1-3) edge node{$\pi$} (m-2-3);
		\end{tikzpicture}
	\end{equation}
	(Above, $\Grass^{\psi(d), e} \vcentcolon= \hat{\cal{B}}_{(n)}^{\psi(d), e}$.)
	
	Part (iii): 
	The multiplication by $\varpi^\delta$ in the last arrow of \eqref{eq:affine-flag-to-m} ensures that $\hat{\cal{B}}_\nu^{\psi(d), 0}$ contains the identity coset $\hat{K}_\nu \in \hat{\cal{B}}_\nu$.
	As a consequence, $\hat{\cal{B}}_\nu^{\psi(d), 0}$ belongs to the connected component of $\hat{\cal{B}}_\nu$ that corresponds to the partial affine flag variety of $\SL_n$ of parabolic type $\nu$.
	The latter is defined analogously to the partial affine flag variety of $G = \GL_n$, which means that $\hat{\cal{B}}_\nu^{\psi(d), 0}$ is precisely the affine Springer fiber over $\psi(d)$ with structure group $\SL_n$.
\end{proof}

\section{\texorpdfstring{Filtrations on $\mathrm{H}^\ast(\CptPic/\Lattice)$}{Filtrations on H*(P/Λ)}}\label{sec:filtration}

\subsection{}

In this section, we discuss the following filtrations on the variety $\CptPic/\Lattice$ or its cohomology:
\begin{enumerate}
	\item
	The gap filtration on the variety, defined in terms of the function $c(M) = \sf{dim}_\bb{C}(\SSS M/M)$ from the introduction.
	
	\item 
	The Hikita filtration \cite{hikita}, defined on the variety for $\RRR = \bb{C}[\![\varpi^n, \varpi^d]\!]$ with $n, d$ coprime, by intersecting the affine Springer fiber from \S\ref{subsec:asf-coprime} with increasing unions of affine Schubert cells.
	
	\item
	The perverse filtration on cohomology, defined in terms of a versal deformation of a global curve $C$ into which $\Spec(\RRR)$ embeds.
	
\end{enumerate}
In \Cref{thm:gap-vs-hikita}, we relate (1) and (2) by way of an involution $\iota$, as needed in \S\ref{subsec:proof-b}.
The involution $\iota$ is related to a duality studied in \cite{gm_14}, but to our knowledge, our work is the first time it has been used to relate the filtrations above.

\subsection{The Gap Filtration}\label{subsec:gap}

Let $\RRR = \bb{C}[\![x]\!][y]/(f)$ be any generically separable degree-$n$ cover of the $x$-axis, fully ramified at $(x, y) = (0, 0)$.
For any integer composition $\nu$ of $n$, we define the \dfemph{gap filtration} on $\CptPic_\nu$ to be its increasing filtration by the subvarieties
\begin{align}
	\CptPic_{\nu, \leq \gap} = \bigcup_{\gap' \leq \gap} {\CptPic_\nu(\gap')}.
\end{align}
It descends to a filtration of $\CptPic_\nu/\Lattice$ by subvarieties $\CptPic_{\nu, \leq \gap}/\Lattice$.
We define $\sf{Q}_{\leq \ast}$ to be the increasing filtration on the Borel--Moore homology of $\CptPic_\nu/\Lattice$ where
\begin{align}
	\sf{Q}_{\leq \gap}\, \ur{H}_\ast^{\BM}(\CptPic_\nu/\Lattice)
	= \im(\ur{H}_\ast^{\BM}(\CptPic_{\nu, \leq \gap}/\Lattice) \to \ur{H}_\ast^{\BM}(\CptPic_\nu/\Lattice)).
\end{align}
We define $\sf{Q}^{\geq \ast}$ to be the decreasing filtration on the cohomology of $\CptPic_\nu/\Lattice$ where
\begin{align}
	\sf{Q}^{\geq \gap}\, \ur{H}^\ast(\CptPic_\nu/\Lattice) = \ker(\ur{H}^\ast(\CptPic_\nu/\Lattice) \to \ur{H}^\ast(\CptPic_{\nu, \leq \gap}/\Lattice)).
\end{align}
Since compactly-supported cohomology is dual to Borel--Moore homology, and $\CptPic_\nu/\Lattice$ is proper, $\sf{Q}^{\geq \gap}$ is orthogonal to $\sf{Q}_{\leq \gap}$ for all $\gap$.
We note in passing that these definitions still make sense for non-planar $\RRR$.

Let $\CptPic_{\nu, \leq \gap}^e \subseteq \CptPic_\nu^e \subseteq \CptPic_\nu$ be the respective preimages of $\CptPic_{\leq \gap}^e \subseteq \CptPic^e \subseteq \CptPic$, like in the notation of \Cref{prop:asf-coprime}.
Any isomorphism $\CptPic^0 \xrightarrow{\sim} \CptPic^e$ induced by multiplication by uniformizer will preserve $\gap$, hence restrict to an isomorphism $\CptPic_{\leq \gap}^0 \xrightarrow{\sim} \CptPic_{\leq \gap}^e$.
This reduces the study of the gap filtration to the study of $\CptPic_{\leq \gap}^0$. 
Recall that in the unibranch case, $\CptPic^0 \simeq \CptPic/\Lattice$, and we prefer to write $\CptJac$ in place of $\CptPic^0$.

\subsection{The Gap Filtration for $y^n = x^d$}\label{subsec:gap-coprime}

Suppose that $\RRR = \bb{C}[\![\varpi^n, \varpi^d]\!]$ with $n, d$ coprime.
Recall that in this case, there is a $\bb{G}_m$-action on $\CptPic$ induced by scaling $\varpi$, which necessarily stabilizes the connected component $\CptJac$.
As in \Cref{sec:coprime}, let 
\begin{align}
	I^\delta(\SSS) 
	&\vcentcolon= \{\Delta \subseteq \bb{Z}_{\geq 0} \mid \text{$\Delta + n, \Delta + d \subseteq \Delta$ and $|\bb{Z}_{\geq 0} \setminus \Delta| = \delta$}\}.
\end{align}
By \cite[\S{3}]{piontkowski}, the setup of \Cref{lem:paving} restricts to a bijection 
\begin{align}
	\begin{array}{rcl}
		I^\delta(\SSS)  &\xrightarrow{\sim} &\CptJac^{\bb{G}_m},\\
		\Delta &\mapsto &M_\Delta
	\end{array}
\end{align}
that partitions $\CptJac$ into the affine spaces $\bb{A}_\Delta$ for $\Delta \in I^\delta(\SSS)$.

\begin{lem}\label{lem:gap-vs-min}
	If $\RRR = \bb{C}[\![\varpi^n, \varpi^d]\!]$ with $n, d$ coprime, then 
	\begin{align}
		\gap(M) = \delta - \min(\Delta)
		\quad\text{for all $\Delta \in I^\delta(\SSS)$ and $M \in  \bb{A}_\Delta(\bb{C})$}.
	\end{align}
	In particular, $\CptJac(\gap) = \bigcup_{\Delta \mid \min(\Delta) = \delta - \gap} \bb{A}_\Delta$ and $\CptJac_{\leq \gap} = \bigcup_{\Delta \mid \min(\Delta) \geq \delta - \gap} \bb{A}_\Delta$.
\end{lem}

\begin{proof}
	In the notation of \Cref{sec:quot}, we have $\varpi^{-{\min(\Delta)}} M \in \cal{D}$ for all $\Delta \in I^\delta(\SSS)$ and $M \in \bb{A}_\Delta$.
	Now observe that $\gap(M) = \gap(\varpi^{-{\min(\Delta)}} M) = -{\min(\Delta)} + \ell(M) = -{\min(\Delta)} + \delta$, where the second equality holds by \Cref{lem:gap-vs-len}.
\end{proof}

\subsection{The Hikita Filtration}\label{subsec:hikita}

Next, we (re)turn to Hikita's work in \cite{hikita}. We follow the same setup as in the previous subsection.
In the notation of \S\ref{subsec:asf-coprime}, recall that \Cref{prop:asf-coprime} gives us an isomorphism
\begin{align}
	\begin{array}{rcl}
		\Grass^{\psi(d), 0} &\xrightarrow{\sim} &\CptJac,\\
		g\hat{K} &\mapsto &M_g,
	\end{array}
\end{align}
where $\Grass^{\psi(d), 0}$ is the affine Springer fiber over $\psi(d)$ with structure group $\SL_n$.
Hikita first defines a filtration of $\Grass^{\psi(d), 0}$, then lifts it to $\hat{\cal{B}}_\nu^{\psi(d), 0}$ along the projection $\hat{\cal{B}}_\nu \to \hat{\cal{B}}_{(n)} = \Grass$.
Thus, as with the gap filtration, we can largely reduce to studying the $\nu = (n)$ case.

Recall that the partition of $\Grass$ into $\hat{I}$-orbits, where $\hat{I} = \hat{K}_{(1^n)}$ acts on $\Grass$ by left multiplication, forms a stratification into \dfemph{affine Schubert cells}, which are affine spaces:
\begin{align}
	\Grass &= \coprod_{\mu \in X_\bullet} \Grass_\mu,
	\quad\text{where $\hat{\cal{G}}_\mu \vcentcolon= \hat{I}x^\mu\hat{K}/\hat{K}$}.
\end{align}
Above, $X_\bullet$ is the same cocharacter lattice as in \S\ref{subsec:asf-coprime}, and for any $\mu \in X_\bullet$, we write $x^\mu$ to mean the image of $x$ under $\mu : \hat{\bb{G}}_m \to \hat{G}$.
The affine Grassmannian of $\SL_n$ is the sub-ind-scheme $\Grass_{\SL_n} \subseteq \Grass$ given by
\begin{align}
	\Grass_{\SL_n} = \coprod_{\mu \in X_\bullet^0} \Grass_\mu,
	\quad\text{where $X_\bullet^0 \vcentcolon= \{\mu \in X_\bullet \mid \mu_1 + \cdots + \mu_n = 0\}$}.
\end{align}
The proof of \cite[Prop.\@ 4.1]{hikita} shows that there is a bijection $a : X_\bullet^0 \xrightarrow{\sim} \bb{Z}_{\geq 0}^{n - 1}$ defined as follows:
$a_i(0, \ldots, 0) = 0$ for all $i$, and if $\mu \neq (0, \ldots, 0)$, then
\begin{align}\label{eq:hikita-a}
	&(a_1, \ldots, a_{n - k}, a_{n - k + 1}, \ldots, a_{n - 1})\\
	&\qquad= (\mu_{k + 1} - \mu_k - 1, \ldots, \mu_n - \mu_k - 1, \mu_1 - \mu_k, \ldots, \mu_{k - 1} - \mu_k),
\end{align}
where $k$ is the largest index in $\{1, \ldots, n\}$ such that $\mu_k = \min_i \mu_i$.
Note that since $\mu_1 + \cdots + \mu_n = 0$, we must have $\mu_k < 0$.
For all $a \in \bb{Z}_{\geq 0}^{n - 1}$, let $|a| = a_1 + \cdots + a_{n - 1}$.
For any integer $\gap$, let
\begin{align}
	\Grass_{\SL_n, \leq \gap} = \bigcup_{\substack{\mu \in X_\bullet^0 \\ |a(\mu)| \leq \gap}} \Grass_\mu.
\end{align}
Following \cite[Cor.\@ 4.7]{hikita}, we define the \dfemph{Hikita filtration} on $\Grass^{\psi(d), 0}$ to be its increasing filtration by the subvarieties $\Grass_{\leq \gap}^{\psi(d), 0} = \Grass^{\psi(d), 0} \cap \Grass_{\SL_n, \leq \gap}$.

For each integer composition $\nu$ of $n$, we define $\hat{\cal{B}}_{\nu, \leq \gap}^{\psi(d), 0}$ to be the preimage of $\Grass_{\leq \gap}^{\psi(d), 0}$ along the projection $\hat{\cal{B}}_\nu \to \Grass$.
We define the \dfemph{Hikita filtration} on $\hat{\cal{B}}_\nu^{\psi(d), 0}$ to be its increasing filtration by these subvarieties.
This recovers the definition for $\nu = (1^n)$ in the proof of \cite[Thm.\@ 4.17]{hikita}.

\subsection{The Involution $\iota$}

For any $g \in G$, let $g^\tau$ be the ``anti-transpose'' given by $g^\tau = J g^t J$, where $g^t$ is the usual transpose and $J \in G$ the matrix with $1$'s along the anti-diagonal and $0$'s elsewhere.
The map $\iota : G \to G$ given by
\begin{align}
	\iota(g) \vcentcolon= (g^\tau)^{-1} = (g^{-1})^\tau
\end{align}
is an involutory automorphism with differential $\iota : \fr{g} \to \fr{g}$ given by 
\begin{align}
	\iota(\gamma) = -\gamma^\tau.
\end{align}
We extend these automorphisms to $\hat{G}$ and its Lie algebra by linearity and completion.
We see that $\iota(\hat{K}) = \hat{K}$ and $\iota(\hat{K}_{(1^n)}) = \hat{K}_{(1^n)}$, from which we deduce that $\iota$ descends to involutions of $\Grass$ and $\hat{\cal{B}}_{(1^n)}$.

From the definition \eqref{eq:hikita}, we also see that $\iota(\psi(d)) = -\psi(d)$.
We deduce that the affine Springer fibers $\Grass^{\psi(d)}$ and $\hat{\cal{B}}_{(1^n)}^{\psi(d)}$ are stable under $\iota$, as are their $\SL_n$ variants $\Grass^{\psi(d), 0}$ and $\hat{\cal{B}}_{(1^n)}^{\psi(d), 0}$.

\begin{lem}\label{lem:involution}
	The involutions above have the following properties:
	\begin{enumerate}
		\item 	For all $\mu \in X_\bullet$, we have $\iota(\Grass_\mu) = \Grass_{\iota(\mu)}$, where
		\begin{align}
			\iota(\mu_1, \ldots, \mu_n) = (-\mu_n, \ldots, -\mu_1).
		\end{align}
		\item 	For any integer $e$, the involution on the Borel--Moore homology of $\hat{\cal{B}}_{(1^n)}^{\psi(d), e}$ induced by $\iota$ is equivariant with respect to the Springer action of $S_n$.
		Moreover, it preserves the homological degree and weight filtration.
	\end{enumerate}
\end{lem}

In preparation for the proof of part (2), we set up some notation.
Recall that $\Grass^{\psi(d), 0} \subseteq \Grass_{\SL_n}$, and hence,
\begin{align}\label{eq:strata}
	\Grass^{\psi(d), 0} = \coprod_{\mu \in X_\bullet^0} \bb{A}_\mu,
	\quad\text{where $\bb{A}_\mu \vcentcolon= \Grass^{\psi(d), 0} \cap \Grass_\mu$}.
\end{align}
Let $X_\bullet^{\psi(d), 0} \subseteq X_\bullet^0$ be the subset of cocharacters $\mu$ for which $\bb{A}_\mu$ is nonempty.
It is explained in \cite[\S{2.3}]{hikita}, following \cite{gkm}, that these $\bb{A}_\mu$ are affine spaces.
Moreover, \cite[Thm.\@ 2.7]{hikita} is an explicit combinatorial formula for their dimensions, which shows that
\begin{align}\label{eq:involution-dim}
	\dim(\bb{A}_\mu) = \dim(\bb{A}_{\iota(\mu)})
\end{align}
for all $\mu \in X_\bullet^{\psi(d), 0}$.

\begin{proof}[Proof of \Cref{lem:involution}]
	Part (1) follows from computing $\iota(x^\mu) = x^{\iota(\mu)}$.
	To show part (2):
	First, recall that the Springer action in question is defined via \Cref{prop:springer} and \Cref{rem:borel-moore} via the outer rectangle of \eqref{eq:asf-to-pic}.
	The bottom arrow of this outer rectangle sends $g\hat{K} \mapsto [\Ad(g^{-1})\psi(d)\,\text{mod $x$}]$.
	So we must show that the residues of $\Ad(g^{-1})\psi(d)$ and $\Ad(\iota(g)^{-1})\psi(d)$ mod $x$ have the same Jordan types as nilpotent elements of $\fr{g}$.
	This follows from computing
	\begin{align}
		\Ad(\iota(g)^{-1})\psi(d) = -{\Ad}(\iota(g)^{-1})\iota(\psi(d)) = -\iota(\Ad(g^{-1})\psi(d)),
	\end{align}
	then observing that $\iota$ commutes with reduction mod $x$ and preserves the Jordan types of nilpotent elements.
	
	The fact that the involution on $\ur{H}_\ast^\BM(\hat{\cal{B}}_{(1^n)}^{\psi(d), e})$ preserves the homological degree and weight filtration follows from $\hat{\cal{B}}_{(1^n)}^{\psi(d), e}$ being paved by the affine spaces $\bb{A}_\mu$, together with the identity \eqref{eq:involution-dim}.
\end{proof}

In the notation from the end of \S\ref{subsec:gap}, set $\CptJac_\nu = \CptPic_\nu^0$ and $\CptJac_{\nu, \leq \gap} = \CptPic_{\nu, \leq \gap}^0$.
Together with \Cref{lem:involution}(2), the following result completes a necessary step in proof (B) of case (1) of \Cref{thm:ors-quot}.

\begin{thm}\label{thm:gap-vs-hikita}
	Suppose that $\RRR = \bb{C}[\![\varpi^n, \varpi^d]\!]$ for some $d > 0$ coprime to $n$.
	Then the composition of isomorphisms
	\begin{align}\label{eq:gap-vs-hikita}
		\hat{\cal{B}}_{(1^n)}^{\psi(d), 0} 
		\xrightarrow{\iota} \hat{\cal{B}}_{(1^n)}^{\psi(d), 0}
		\xrightarrow{\text{Prop \ref{prop:asf-coprime}}} \CptJac_{(1^n)}
	\end{align}
	restricts to an isomorphism $\iota(\hat{\cal{B}}_{(1^n), \leq \gap}^{\psi(d), 0}) \xrightarrow{\sim} \CptJac_{(1^n), \leq \gap}$ for all $\gap$.
\end{thm}

The proof will occupy the rest of this subsection. Since $\CptJac_{(1^n), \leq \gap}$ and $\hat{\cal{B}}_{(1^n), \leq \gap}^{\psi(d), 0}$ are respectively the preimages of $\CptJac_{\leq \gap}$ and $\Grass_{\leq \gap}^{\psi(d), 0}$, the $\iota$-equivariance of the projection $\hat{\cal{B}}_{(1^n)}^{\psi(d), 0} \to \Grass^{\psi(d), 0}$ and the commutativity of the left square of \eqref{eq:asf-to-pic} allows us to replace $\nu = (1^n)$ with $\nu = (n)$.

We will match the strata $\bb{A}_\mu \subseteq \Grass^{\psi(d), 0}$ from \eqref{eq:strata} with the strata $\bb{A}_\Delta \subseteq \CptJac$ from \Cref{lem:gap-vs-min}; this implies the statement by the definition of the filtrations in question.
Let $- \cdot_{1/n} -$ denote the $\bb{G}_m$-action on $\hat{G}$ defined by
\begin{align}\label{eq:1-over-n}
	t \cdot_{1/n} g(x) \vcentcolon= t^{2\rho^\vee} g(t^{2n}\varpi) t^{-2\rho^\vee}
\end{align}
for all $t \in \bb{G}_m$ and $g \in \hat{G}$.
It descends to a $\bb{G}_m$-action on $\Grass$ that we again denote by $- \cdot_{1/n} - $.
As explained in \cite{hikita, gkm}, we have
\begin{align}
	\bb{A}_\mu &= \{g\hat{K} \in \Grass^{\psi(d), 0} \mid \lim_{t \to 0} {(t \cdot_{1/n} g\hat{K})} = x^\mu \hat{K}\}
	&&\text{for all $\mu \in X_\bullet^0$}.
\end{align}
So to match the strata it suffices to match $- \cdot_{1/n} - $ with the $\bb{G}_m$-action on $\CptJac$ in \S\ref{subsec:gap-coprime}.

\begin{prop}\label{prop:1-over-n}
	The map \eqref{eq:gap-vs-hikita} transports the $\bb{G}_m$-action $- \cdot_{1/n} -$ on $\Grass$ onto the $\bb{G}_m$-action on $\cal{M}_{(n)}$ induced by $t \cdot_2 \varpi \vcentcolon= t^2\varpi$.
\end{prop}

\begin{proof}
	It suffices to work on $\bb{C}$-points.
	First, $\iota$ is equivariant under $- \cdot_{1/n} -$ because $\iota(c^{2\rho^\vee}) = c^{2\rho^\vee}$, so we can replace \eqref{eq:gap-vs-hikita} with \eqref{eq:affine-flag-to-m}.
	Observe that if $g = g(x) \in G(\bb{C}(\!(x)\!))$, and $g'(x) = t \cdot_{1/n} g(x)$ for some $t \in \bb{C}^\times$, then $g'(x)^{-1} = t \cdot_{1/n} g^{-1}(x)$.
	Thus the entries of the matrix $g'(x)^{-1}$ are given by
	\begin{align}
		(g'(x)^{-1})_{i, j}
		= t^{2(j - i)} (g(t^{2n}\varpi)^{-1})_{i, j}.
	\end{align}
	We deduce that
	\begin{align}
		\varpi^\delta\,
		\abel(\sf{v}_i \cdot (t \cdot_{1/n} g(x))^{-1})
		&= 	\varpi^\delta
		\sum_j (g'(\varpi^n)^{-1})_{i, j} \varpi^j\\
		&=	t^{-2(i - 1)} \varpi^\delta
		\sum_j {(g((t^2 \varpi)^n)^{-1})_{i, j}} (t^2 \varpi)^{j - 1}\\
		&= t^{-2\delta - 2(i - 1)} (t \cdot_2 \varpi^\delta\, \abel(\sf{v}_i \cdot g(x)^{-1})).
	\end{align}
	The outer monomials in the first and last expressions are just nonzero scalars.
	So the vector subspaces of $\bb{C}(\!(\varpi)\!)$ formed by $M_{t \cdot_{1/n} g}$ and $t \cdot_2 M_g$ coincide.
\end{proof}

In the notation of \S\ref{subsec:gap-coprime}, let $\Delta : X_\bullet^{\psi(d), 0} \to I^\delta(\SSS)$ be defined by
\begin{align}\label{eq:mu-to-delta}
	\Gen_n(\Delta(\mu)) = \{n\mu_i + n - i + \delta \mid 1 \leq i \leq n\}.
\end{align}
Then the map $x^\mu \mapsto M_{\Delta(\mu)}$ is precisely the effect of \eqref{eq:gap-vs-hikita} on the $(- \cdot_{1/n} -)$-fixed points of $\Grass^{\psi(d), 0}$, as we can check from the definition of $\iota$ and \eqref{eq:affine-flag-to-l}--\eqref{eq:l-to-m}.
Now \eqref{eq:involution-dim} and \Cref{prop:1-over-n} imply:

\begin{cor}
	The map $\Delta : X_\bullet^{\psi(d), 0} \to I^\delta(\SSS)$ is bijective, and for all $\mu \in X_\bullet^{\psi(d), 0}$, \eqref{eq:gap-vs-hikita} restricts to an isomorphism $\bb{A}_\mu \xrightarrow{\sim} \bb{A}_{\Delta(\mu)}$.
\end{cor}

To finish the proof of \Cref{thm:gap-vs-hikita}, it remains to show that for all $\mu \in X_\bullet^{\psi(d), 0}$, we have $|a(\mu)| = \gap(M_{\Delta(\mu)})$.
By \Cref{lem:gap-vs-min} and \eqref{eq:mu-to-delta}, this is equivalent to:

\begin{lem}
	For all $\mu \in X_\bullet^{\psi(d), 0}$, we have
	\begin{align}
		|a(\mu)| = -{\min\{n\mu_i + n - i \mid 1 \leq i \leq n\}}.
	\end{align}
\end{lem}

\begin{proof}
	If $\mu = (0, \ldots, 0)$, then both sides equal $0$.
	If $\mu \neq (0, \ldots, 0)$, then \eqref{eq:hikita-a} gives
	\begin{align}
		|a(\mu)| = (\mu_1 + \cdots + \mu_n) - (n\mu_k + n - k),
	\end{align}
	where $k$ is the largest index in $\{1, \ldots, n\}$ such that $\mu_k = \min_i \mu_i$.
	Since $\mu \in X_\bullet^0$, the right-hand side above simplifies to $-(n\mu_k + n - k)$.
\end{proof}

\begin{rem}
	It is natural to ask what the involution $\iota$ on $\Grass^{\psi(d), 0}$ looks like after being transported through \eqref{eq:affine-flag-to-m}, to an involution on $\CptJac$.
	From \eqref{eq:mu-to-delta}, we can check that it is precisely the duality that Gorsky--Mazin denote by $\Delta \mapsto \widehat{\Delta}$ in \cite{gm_14}.
	Explicitly, for any $\Delta \in I^\delta(\SSS)$, we have $\Gen_n(\widehat{\Delta}) = \{d(n - 1) - k \mid k \in \Gen_n(\Delta)\}$.
\end{rem}

\begin{ex}\label[ex]{ex:min-delta}
	Take $(n, d) = (3, 4)$.
	We compute 
	\begin{align}
		X_\bullet^{\psi(d), 0}
		&= \{(0, 0, 0), (-1, 0, 1), (-1, 1, 0), (0, -1, 1), (1, 0, -1)\},\\
		I^\delta(\SSS)
		&= \{\Delta_{3, 4, 5}, \Delta_{6, 4, 2}, \Delta_{3, 7, 2}, \Delta_{6, 1, 5}, \Delta_{0, 4, 8}\}.
	\end{align}
	Above, we have labeled the elements of $I^\delta(\SSS)$ in the form $\Delta_{b_1, b_2, b_3}$, where $\Gen_n = \{b_1, b_2, b_3\}$ and $b_i \equiv i - 1 \pmod{n}$ for all $i$.
	Compare the statistics below to \Cref{ex:3-4}:
	\begin{align}
		\begin{array}{lllllll}
			\mu 
			&a(\mu)
			&|a(\mu)|
			&(n\mu_i + n - i)_i
			&\Delta(\mu)
			&\min(\Delta(\mu))\\
			\hline
			(0, 0, 0)
			&(0, 0)
			&0
			&(2, 1, 0)
			&\Delta_{3,4,5}
			&3\\
			(-1, 0, 1)
			&(0, 1)
			&1
			&(-1, 1, 3)
			&\Delta_{6,4,2}
			&2\\
			(-1, 1, 0)
			&(1, 0)
			&1
			&(-1, 4, 0)
			&\Delta_{3,7,2}
			&2\\
			(0, -1, 1)
			&(1, 1)
			&2
			&(2, -2, 3)
			&\Delta_{6,1,5}
			&1\\
			(1, 0, -1)
			&(2, 1)
			&3
			&(5, 1, -3)
			&\Delta_{0,4,8}
			&0
		\end{array}
	\end{align}
\end{ex}

\subsection{The Perverse Filtration}

We return to the setup of \S\ref{subsec:gap}, where $f(x, y)$ is arbitrary.
For simplicity, we ignore the map to the $x$-axis in what follows.
Besides $\sf{Q}^{\geq \ast}$, there is another filtration on the cohomology of $\CptPic/\Lattice$, defined as follows by Maulik--Yun:

Fix a complex, integral, projective curve $C$, whose normalization has genus zero, and which is smooth away from a unique planar singularity given in local coordinates by $f(x, y) = 0$.
We emphasize that while $C$ is integral, the germ $f$ can still have multiple branches.
Fix an embedding of $C$ into a family of curves $\cal{C}$, whose base is irreducible, and which satisfies conditions (A1)--(A4) in \cite[\S{2.1}]{my}.

Let $\CptJac(C)$ be the compactified Jacobian of $(C, s)$ \cite{aik}.
In this setting, Section 2.14 of \cite{my} defines an increasing perverse filtration $\sf{P}_{\leq \ast}$ on $\ur{H}^\ast(\CptJac(C))$, in terms of the perverse truncation of the pushforward of the constant sheaf along the structure map of $\CptJac(\cal{C})$.
Proposition 2.15 of \cite{my} shows that $\sf{P}_{\leq \ast}$ is invariant under base change of the family of curves, so it is canonical.
It is strictly compatible with the weight filtration $\sf{W}_{\leq \ast}$.
Finally, the proof of Theorem 3.11 of \cite{my} shows that there is a weight-preserving isomorphism $\ur{H}^\ast(\CptJac(C)) \simeq \ur{H}^\ast(\CptPic/\Lattice)$, canonical up to the choice of uniformization that defines the $\Lattice$-action on $\CptPic$.
Following Maulik--Yun, we normalize $\sf{P}_{\leq \ast}$ so that it sits in degrees $0$ through $2\delta$.

For any filtration $\sf{F}_{\leq \ast}$ on the cohomology of $\CptPic/\Lattice$, strictly compatible with the weight filtration, we may form the \dfemph{virtual Poincar\'e polynomial}
\begin{align}
	\sf{P}^{\virtual, \sf{F}}(\sf{q}, \sf{t})
	= \sum_{i, j, k}
	{(-1)^i}
	\sf{q}^j
	\sf{t}^k
	\dim 
	\gr_j^\sf{F}
	\gr_k^\sf{W}
	\ur{H}^i(\CptPic/\Lattice).
\end{align}
Explicitly, Theorem 3.11 of \cite{my} and Theorem \ref{thm:main} imply that
\begin{align}
	\sf{Hilb}(\sf{q}, \sf{t}) = \frac{1}{(1 - \sf{q})^b}\,
	\sf{P}^{\virtual, \sf{P}}(\sf{q}, \sf{t})
	\quad\text{and}\quad
	\sf{Quot}(\sf{q}, \sf{t})= \frac{1}{(1 - \sf{q})^b}\,
	\sf{P}^{\virtual, \sf{Q}}(\sf{q}, \sf{t}).
\end{align}
We deduce that:
\begin{cor}\label{cor:perverse}
	\Cref{conj:main} is equivalent to $\sf{P}^{\virtual, \sf{P}}(\sf{q}, \sf{t})
	= \sf{P}^{\virtual, \sf{Q}}(\sf{q}, \sf{q}^{\frac{1}{2}}\sf{t})$.
\end{cor}

It is natural to make the following stronger conjecture, which also extends a conjecture in unpublished notes of Yun beyond the unibranch case.
\begin{conj}\label{conj:perverse}
	The weight grading on $\gr_\ast^\sf{W} \ur{H}^\ast(\CptPic/\Lattice)$ is supported in even degrees.
	Moreover, $\gr_{j + k}^\sf{P} \gr_{2k}^\sf{W} \ur{H}^\ast(\CptPic/\Lattice) \simeq \gr_\sf{Q}^j \gr_{2k}^\sf{W} \ur{H}^\ast(\CptPic/\Lattice)$ for all $j, k$.
\end{conj}
The motivation behind \Cref{conj:perverse} is that it would strictly imply the statement above, and hence, \Cref{conj:main}.
We emphasize again that while $\sf{P}_{\leq \ast}$ is defined via auxiliary global methods, $\sf{Q}^{\geq \ast}$ is intrinsic and purely local.
For this reason, \Cref{cor:perverse} seems remarkable to us.

\appendix
\section{Gradings on Link Homology}\label{sec:conventions}

\subsection{}

In this appendix, we specify our grading conventions for Khovanov--Rozansky homology; compare them to those of other published works; and illustrate on the smallest examples (unknot, Hopf link, trefoil, $(3, 4)$ torus knot) to aid the reader's sanity.
Our exposition closely follows \cite[\S{1.6}]{GH17}, but we correct some mistakes:
See \Crefrange{rem:mistake-dgr}{rem:mistake-misc}.

\subsection{Soergel Bimodules}

Let $T = \bb{G}_m^n$, and let $\bb{S} \vcentcolon= \ur{H}_{T}^\ast(\point) = \bb{C}[t_1, \ldots, t_n]$.
We regard $\bb{S}$ as a graded ring, with $\deg(t_i) = 2$ for all $i$.
Thus the $S_n$-action on $T$ that permutes coordinates also preserves the grading on $\bb{S}$.
Let $s_i \in S_n$ be the transposition that swaps $t_i$ and $t_{i + 1}$.

In the category of graded $\bb{S}$-bimodules, we write $(m)$ for the grading shift $\bb{B}(m)^i = \bb{B}^{i + m}$.
Let $\Bim{\bb{S}}$ be the full subcategory generated by the identity bimodule $\bb{S}$ and the bimodules $\bb{S} \otimes_{\bb{S}^{s_i}} \bb{S}(1)$ for all $i$ under isomorphisms, direct sums, tensor products $\otimes = \otimes_{\bb{S}}$, direct summands, and grading shifts.
Objects of $\Bim{\bb{S}}$ are called \dfemph{Soergel bimodules}.
We write $\sf{K}^b(\Bim{\bb{S}})$ for the bounded homotopy category, a monoidal additive category under $\otimes$.

Let $\Br_n$ be the group of braids on $n$ strands up to isotopy.
Any braid $\beta \in \Br_n$ defines an object $\bar{\cal{T}}_\beta \in \sf{K}^b(\Bim{\bb{S}})$ called the Rouquier complex of $\beta$.
See, \emph{e.g.}, \cite[\S{2.1}]{GH17} for the precise definition.

Let $\Vect_2$ be the category of $\bb{Z}^2$-graded vector spaces that are finite-dimensional in each bidegree, such that the first grading is bounded below and the second is bounded.
Let $\overline{\sf{HH}} = \overline{\sf{HH}}^{\ast, \ast} : \Bim{\bb{S}} \to \Vect_2$ be the \dfemph{Hochschild cohomology} functor:
\begin{align}
	\overline{\sf{HH}}^{i,j}(\bb{B}) = \Ext_{\bb{S} \otimes_\bb{C} \bb{S}^\op}^i(\bb{S}, \bb{B}(j)).
\end{align}
These Ext's can be computed using a Koszul resolution of $\bb{S}$ over $\bb{S} \otimes_\bb{C} \bb{S}^\op$, which shows that the Ext grading sits in degrees $0$ through (at most) $n$.

Let $\Vect_3$ be the category of $\bb{Z}^3$-graded vector spaces that are finite-dimensional in each tridegree, such that the first grading is bounded below and the other two gradings are bounded.
Let $\overline{\HHH} = \overline{\HHH}^{\ast, \ast, \ast}$ be the composition of functors
\begin{align}
	\sf{K}^b(\Bim{\bb{S}}) \xrightarrow{\overline{\sf{HH}}} \sf{K}^b(\sf{Vect}_2) \xrightarrow{\ur{H}^\ast} \Vect_3.
\end{align}
Explicitly, the gradings are ordered so that $\overline{\HHH}^{I, J, K} = \ur{H}^k(\overline{\sf{HH}}_n^{I, J})$.

The story above can be redone with the quotient torus $T_0 \vcentcolon= T/T^{S_n}$ in place of $T$.
Note that $T_0$ is just the image of $T$ along the quotient map $\GL_n \to \PGL_n$.
Replacing $T$ with $T_0$ entails replacing $\bb{S}$ with its subring $\bb{S}_0 \vcentcolon= \ur{H}_{T_0}^\ast(\point)$.
We write $\cal{T}_\beta$, $\sf{HH}$, $\HHH$ for the objects that respectively replace $\bar{\cal{T}}_\beta$, $\overline{\sf{HH}}$, $\overline{\HHH}$.

Let $L$ be the link closure of $\beta$.
In \cite{kh}, Khovanov proved that $\HHH(\cal{T}_\beta)$ matches the \dfemph{reduced} version of the triply-graded homology of $L$ proposed in \cite{dgr} and constructed in \cite{khr}, up to an affine transformation of the trigrading.
One can show that
\begin{align}\label{eq:unreduced-vs-reduced-homology}
	\overline{\HHH}(\bar{\cal{T}}_\beta) \simeq \overline{\HHH}(\bar{\cal{T}}_{\id}) \otimes \HHH(\cal{T}_\beta),
\end{align}
and that in consequence, $\overline{\HHH}(\bar{\cal{T}}_\beta)$ matches the \dfemph{unreduced} version of the homology constructed in \cite{khr}, up to similar regradings.

\subsection{The Main Dictionary}

For any $\beta \in \Br_n$, let
\begin{align}
	\overline{\hhh}_\beta(A, Q, T) 
	&= \sum_{I,J,K}
	A^I Q^J T^K\, 
	\dim \overline{\HHH}^{I, J, K}(\bar{\cal{T}}_\beta),\\[1ex]
	\hhh_\beta(A, Q, T) 
	&= \sum_{I,J,K}
	A^I Q^J T^K\,
	\dim \HHH^{I, J, K}(\cal{T}_\beta).
\end{align}
That is:
\begin{enumerate}
	\item 	$\overline{\hhh}_\beta(A, Q, T)$ is the series denoted $\cal{P}_\beta(Q, A, T)$ in \cite[\S{A}]{eh} and \cite[\S{1.6}]{GH17}, and $\hhh_\beta$ is the analogue of $\overline{\hhh}$ for reduced homology.
\end{enumerate}
We write:
\begin{enumerate}\setcounter{enumi}{1}
	\item 	$\bar{\cal{P}}_L^\norm(A, Q, T)$ for the series denoted $\cal{P}_L^\norm(Q, A, T)$ in \cite{GH17}.
	
	\item 	$\cal{P}_{L, \ORS}(a, q, t)$ for the series denoted $\cal{P}(L)$ in \cite{ors}.
	It is denoted $\cal{P}(L^-)$ in \cite{dgr}, where $L^-$ is the chiral mirror of $L$.
	
	\item 	$\bar{\cal{P}}_{L, \ORS}(a, q, t)$ for the series denoted $\bar{\cal{P}}(L)$ in \cite{ors}, which satisfies
	\begin{align}\label{eq:unreduced-vs-reduced-series}
		\bar{\cal{P}}_{L, \ORS}(a, q, t) = \bar{\cal{P}}_{U, \ORS}(a, q, t) \cal{P}_{L, \ORS}(a, q, t).
	\end{align}

\end{enumerate}

\begin{rem}\label{rem:mistake-dgr}
	Contrary to statements suggested by \cite[651]{ors} and \cite[\S{1.6}]{GH17}, the series $\bar{\cal{P}}_{L, \ORS}$ does \emph{not} match the series called the unreduced superpolynomial of $L^-$ and denoted $\bar{\cal{P}}(L^-)$ in \cite{dgr}, even after further regrading.
	Indeed, the series denoted $\cal{P}(L^-)$ and $\bar{\cal{P}}(L^-)$ in \cite{dgr} are \emph{not} proportional to each other by any constant factor, as can be checked from Propositions 6.1 and 6.2 of \cite{dgr}.
\end{rem}

Let $e$ be the \dfemph{writhe} of $\beta$, meaning its net number of crossings counted with sign, and let $b$ be the number of components of $L$.
After correction, \cite[\S{1.6}]{GH17} states:
\begin{align}
	\bar{\cal{P}}_L^\norm(A, Q, T)
	&=
	(A^{\frac{1}{2}})^{e - n + b}
	Q^{-e + 2n - 2b}
	(T^{\frac{1}{2}})^{-e - n + b}\,
	\overline{\hhh}_\beta(A, Q, T),
	\\[1ex]
	\label{eq:ors-vs-norm-hhh}
	\bar{\cal{P}}_{L, \ORS}(a, q, t)
	&=
	a^{-b} q^b\,
	\bar{\cal{P}}_L^\norm(a^2 q^2 t, q, t^{-1})\\
	&=
	a^{e - n} q^n t^e\,
	\overline{\hhh}_\beta(a^2 q^2 t, q, t^{-1}).
\end{align}
By combining the last identity above with \eqref{eq:unreduced-vs-reduced-homology}--\eqref{eq:unreduced-vs-reduced-series}, we get a reduced version:
\begin{align}
	\cal{P}_{L, \ORS}(a, q, t)
	&=
	a^{e - n + 1} 
	q^{n - 1}
	t^e\,
	\hhh_\beta(a^2 q^2 t, q, t^{-1}).
\end{align}
In general, we will not work with $\bar{\cal{P}}_L^\norm$.
Moreover, we will not discuss at all the normalizations used in the series $\cal{P}(U), \cal{P}(T(2, 3))$ in \cite[Rem.\ 1.27]{GH17}.

\begin{rem}\label{rem:mistake-misc}
	Above, \eqref{eq:ors-vs-norm-hhh} fixes a few more typos in \cite[\S{1.6}]{GH17}:
	
	First, the discussion on \cite[599]{GH17} relates their series $\cal{P}_L^\norm$ to the series we call $\bar{\cal{P}}_{L, \ORS}$, not to the superpolynomial in \cite{dgr}.
	As explained in \Cref{rem:mistake-dgr}, the latter two are different.
	Next, the identity relating $\cal{P}_L^\norm$ and $\bar{\cal{P}}_{L, \ORS}$ in \emph{loc.\@ cit.}\@ has the wrong prefactor.
	There, the authors express $\bar{\cal{P}}_{L, \ORS}$ in terms of variables $r, \alpha, Q, T$, which correspond to our $b, a, q, t^{-1}$, respectively.
	Their prefactor $Q^{2r} \alpha^{-r}$ should be $Q^r \alpha^{-r}$.
	
	By way of comparison:
	The variables $\alpha, Q, T$ in \cite[\S{A}]{eh} also correspond to our $a, q, t^{-1}$.
	Hence, their series $\cal{P}_L(Q, \alpha, T)$ is our series $\bar{\cal{P}}_{L, \ORS}(a, q, t)$.
	The identity relating $\cal{P}_\beta$ and $\cal{P}_L$ in \emph{loc.\@ cit.}\@ is correct.
\end{rem}

\begin{ex}\label{ex:unknot}
	The unknot $U$ is the knot closure of the identity in $\Br_1$, for which $(n, e, b) = (1, 0, 1)$.
	The Hochschild cohomology of the identity Soergel bimodule is
	\begin{align}
		\overline{\sf{HH}}_1^{\ast, j}(\bb{S}) = \left\{\begin{array}{ll}
			\bb{S}
			&j = 0,\\
			\bb{S}(2)
			&j = 1,\\
			0
			&j \neq 0, 1.
		\end{array}\right.
	\end{align}
	Thus $\bar{\cal{P}}_U^\norm(A, Q, T) = \overline{\hhh}_{\id}(A, Q, T) = \dfrac{1 + AQ^{-2}}{1 - Q^2}$, from which
	\begin{align}
		\bar{\cal{P}}_{U, \ORS}(a, q, t)
		&= \frac{a^{-1} + a t}{q^{-1} - q}.
	\end{align}
\end{ex}

\subsection{``Our'' Series}

For any braid $\beta \in \Br_n$ with writhe $e$ whose link closure $L$ has $b$ components, let
\begin{align}
	\bar{\sf{X}}_\beta(\sf{a}, \sf{q}, \sf{t})
	&\vcentcolon=
	\sf{t}^{\frac{e}{2}}\,
	\overline{\hhh}_\beta(\sf{a}\sf{q}, \sf{q}^{\frac{1}{2}}, \sf{q}^{\frac{1}{2}} \sf{t}^{-\frac{1}{2}}),\\[1ex]
	\sf{X}_\beta(\sf{a}, \sf{q}, \sf{t})
	&\vcentcolon=
	\frac{\bar{\sf{X}}_\beta(\sf{a}, \sf{q}, \sf{t})}{\bar{\sf{X}}_{\id}(\sf{a}, \sf{q}, \sf{t})}
	= 
	\sf{t}^{\frac{e}{2}}\,
	\hhh_\beta(\sf{a}\sf{q}, \sf{q}^{\frac{1}{2}}, \sf{q}^{\frac{1}{2}} \sf{t}^{-\frac{1}{2}}).
\end{align}
Above, note that $\sf{X}_{\id}(\sf{a}, \sf{q}, \sf{t}) = \dfrac{1 + \sf{a}}{1 - \sf{q}}$.
We can check that
\begin{align}
	\bar{\cal{P}}_{L, \ORS}(a, q, t)
	&=
	(aq^{-1})^{e - n}\,
	\bar{\sf{X}}_\beta(a^2 t, q^2, q^2 t^2),\\
	\cal{P}_{L, \ORS}(a, q, t)
	&=
	(aq^{-1})^{e - n + 1}\,
	\sf{X}_\beta(a^2 t, q^2, q^2 t^2).
\end{align}
It turns out that in the rest of this paper, $\bar{\sf{X}}_\beta$ and $\sf{X}_\beta$ are the most convenient series for us to use.

In particular, suppose that $f(x, y) \in \bb{C}[\![x]\!][y]$ such that $f(x, y) = 0$ defines a generically separable, degree-$n$ cover of the $x$-axis, fully ramified at $(x, y) = (0, 0)$.
Then the preimage in the cover of a positively-oriented loop around $x = 0$ is a braid $\beta_f \in \Br_n$, whose link closure is the link $L_f$ introduced in \S\ref{subsec:ors}.
We see that $\bar{\sf{X}}_{\beta_f}$ is precisely the series $\bar{\sf{X}}_f$ introduced in \eqref{eq:khr-normalization}.

\subsection{Torus Links}

For integers $n, d > 0$, let $T_{n, d}$ be the positive $(n, d)$ torus link, considered negative in \cite{dgr}.
Its number of components is $b = \gcd(n, d)$.
Taking $f(x, y) = y^n - x^d$ in the construction above shows that $T_{n, d}$ is the link closure of a braid $\beta_{n, d} \in \Br_n$ for which $e = (n - 1)d$.
Let 
\begin{align}
	\delta = \frac{1}{2}(e - n + b) = \frac{1}{2}(nd - n - d + \gcd(n, d)).
\end{align}
Let $\bar{\sf{X}}_{n, d} = \bar{\sf{X}}_{\beta_{n, d}}$, as in the rest of this paper, and $\sf{X}_{n, d} = \sf{X}_{\beta_{n, d}}$.

\begin{ex}
	For the Hopf link $T_{2, 2}$, we have
	\begin{align}
		\sf{X}_{2, 2}(\sf{a}, \sf{q}, \sf{t})
		&= 1 + \frac{\sf{q}\sf{t}}{1 - \sf{q}} + \frac{\sf{a}\sf{t}}{1 - \sf{q}},\\
		\cal{P}_{T_{2, 2}, \ORS}(a, q, t)
		&= aq^{-1} + \frac{aq^3 t^2}{1 - q^2} + \frac{a^3 q t^3}{1 - q^2}.
	\end{align}
\end{ex}

\begin{ex}
	For the trefoil $T_{2, 3}$, we have
	\begin{align}
		\sf{X}_{2, 3}(\sf{a}, \sf{q}, \sf{t})
		&= 1 + \sf{q}\sf{t} + \sf{a}\sf{t},\\
		\cal{P}_{T_{2, 3}, \ORS}(a, q, t)
		&= a^2(q^{-2} + q^2 t^2) + a^4 t^3.
	\end{align}
	The latter series is \cite[Ex.\ 3.3]{dgr}.
\end{ex}

\begin{ex}
	For the $(3, 4)$ torus knot $T_{3, 4}$, we have
	\begin{align}
		\sf{X}_{3, 4}(\sf{a}, \sf{q}, \sf{t})
		&= 1 + \sf{q} \sf{t} + \sf{q} \sf{t}^2 + \sf{q}^2 \sf{t}^2 + \sf{q}^3 \sf{t}^3
		+ \sf{a} (\sf{t} + \sf{t}^2 + \sf{q} \sf{t}^2 + \sf{q} \sf{t}^3 + \sf{q}^2 \sf{t}^3) 
		+ \sf{a}^2 \sf{t}^3,\\
		\cal{P}_{T_{3, 4}, \ORS}(a, q, t)
		&= a^6(q^{-6} + q^{-2} t^2 + t^4 + q^2 t^4 + q^6 t^6)\\
		&\qquad
		+ a^8 (q^{-4} t^3 + q^{-2} t^5 + t^5 + q^2 t^7 + q^4 t^7)
		+ a^{10} t^8.
	\end{align}
	The latter series is \cite[Ex.\ 3.4]{dgr}.
\end{ex}

In \Cref{sec:coprime}, we implicitly need the following identities that match $\bar{\sf{X}}_{n, d}, \sf{X}_{n, d}$ with other series in the literature.
\begin{enumerate}
	\item 	Let $\tilde{P}_{n, m}(u, q, t)$ be the series in \cite{gn}.
	For coprime $n, d$, we have
	\begin{align}
		\bar{\sf{X}}_{n, d}(\sf{a}, \sf{q}, \sf{t})
		=
		\frac{\sf{t}^\delta}{1 - \sf{q}}\,\tilde{P}_{n, d}(-\sf{a}, \sf{q}, \sf{t}^{-1}).
	\end{align}
	
	\item 	Let $\cal{P}_{m, n} = \cal{P}_{m, n}(a, q, t)$ be the series in \cite{mellit_22}.
	For coprime $n, d$, we have
	\begin{align}
		\bar{\sf{X}}_{n, d}(\sf{a}, \sf{q}, \sf{t})
		&= (-\sf{a}^{-1} \sf{q}^{\frac{1}{2}} \sf{t}^{\frac{1}{2}})^\delta\,
		\cal{P}_{n, d}(-\sf{a}, \sf{q}, \sf{t}^{-1}).
	\end{align}
	Note that the substitution sends $t \mapsto \sf{q}$ and $q \mapsto \sf{t}^{-1}$, not vice versa.
	
	\item 	Let $\hat{P}_{0^M, 0^N}(q, t, a), \hat{Q}_{0^M, 0^N}(q, t, a), R_{0^M, 0^N}(q, t, a)$ be the series in \cite{gmv}.
	For any $n, d$, we have
	\begin{align}
		\frac{1}{1 + \sf{a}}\,
		\bar{\sf{X}}_{n, d}(\sf{a}, \sf{q}, \sf{t})
		&= 
		\frac{1}{1 - \sf{q}}\,
		\sf{X}_{n, d}(\sf{a}, \sf{q}, \sf{t}^{-1})\\
		&=
		R_{0^n, 0^d}(\sf{q}, \sf{t}^{-1}, \sf{a}\sf{q}^{-1})\\
		&=
		\hat{Q}_{0^n, 0^d}(\sf{q}, \sf{t}^{-1}, \sf{a}\sf{q}^{-1})
		&\text{by}
		&\:\text{\cite[Cor.\@ 5.10]{gmv}}\\
		&=
		\sf{q}^{-d - n} \hat{P}_{0^n, 0^d}(\sf{q}, \sf{t}, \sf{a}\sf{q}^{-1})
		&\text{by}
		&\:\text{\cite[(11)]{gmv}}.
	\end{align}
	
\end{enumerate}

\bibliographystyle{abbrvurl}
\bibliography{quot_v3}

\end{document}